%% file: main.tex
\documentclass[11pt, reqno]{amsart} 
\usepackage[dvipsnames]{xcolor}
\input{preamble}
\makeatletter
\newcommand{\customlabel}[2]{%
   #2\def\@currentlabel{#2}\label{#1}%
}
\usepackage{comment}
\usepackage{amsfonts, mathrsfs, amsmath}
\usepackage[dvipsnames]{xcolor}
\usetikzlibrary{shapes, arrows.meta, positioning, calc, backgrounds, lindenmayersystems, decorations.pathreplacing}

\title{Quantitative Brownian regularity of the KPZ fixed point with arbitrary initial data}

\author{Pantelis Tassopoulos}
\address{Department of Pure Mathematics and Mathematical Statistics, 
University of Cambridge, Cambridge, United Kingdom}
\email{pkt28@cam.ac.uk}

\author{Sourav Sarkar}
\address{Department of Pure Mathematics and Mathematical Statistics, 
University of Cambridge, Cambridge, United Kingdom}
\email{ss2871@cam.ac.uk}

\begin{document}

\date{}

\begin{abstract}
We show that the spatial increments of the KPZ fixed point starting from arbitrary initial data, exhibit strong quantitative comparison against rate two Brownian motion on compacts. The above estimates are uniform in the initial data supported in some compact set. 
As applications, we obtain a one-sided large deviation inequality for spatial increments of the KPZ fixed point and show that the Wiener density of the centred KPZ fixed point started from arbitrary initial data has finite entropy.
\color{black}
\end{abstract}

\maketitle

\tableofcontents

\section{Introduction}\label{intro}

In 1986, Kardar, Parisi and Zhang \cite{kardar1986dynamic} predicted that many planar random growth processes possess universal scaling behaviour. In particular, models in the KPZ universality class have an analogue of the height function which is conjectured to converge at large time and small length scales under the KPZ $1:2:3$ scaling (i.e. $h(t, x)\mapsto \varepsilon h(\varepsilon^{-3}t, \varepsilon^{-2}x)$, as $\varepsilon \searrow 0$) to a universal object $\mathfrak{h}_t(\cdot)$ called the KPZ fixed point. Matetski-Quastel-Remenik \cite{matetski2016kpz} constructed the KPZ fixed point as a Markov process in $t$, and they showed that it is a limit of the height function evolution of the totally asymmetric simple exclusion process (TASEP) with arbitrary initial condition. Later in \cite{nica2020one}, Nica-Quastel-Remenik constructed the KPZ fixed point as a scaling limit of Brownian last passage percolation (LPP). For an introduction to the KPZ universality class, see \cite{quastel2011introduction}, \cite{ferrari2010random}, \cite{romik2015surprising}, \cite{corwin2016kardar}, \cite{weiss2017reflected}, \cite{ganguly2021random} and \cite{ZygourasalgKPZ}.

The directed landscape $\mathcal L $ was constructed from Brownian last passage percolation (BLPP) in \cite{DOV} as a four-parameter scaling limit of the Brownian last passage value from different spatial locations and curves in the Brownian environment. It is conjectured to be the full scaling limit of all KPZ models. Reinforcing this claim, in \cite{dauvergne2025characterizationdirectedlandscapekpz}, the authors provide a general framework for proving convergence to the directed landscape and apply it to a range of models, proving convergence. It is a random continuous function from $$
\R^4_\uparrow = \{(p; q) = (x, s; y, t) \in \R^4 : s < t\}
$$
to $\R$. They showed that the KPZ fixed point also admits a variational formula in terms of the directed landscape, where for initial data $h_0:\mathbb R \to \mathbb R\cup \{-\infty\}$ the KPZ fixed point can be expressed as
\begin{equation*}\mathfrak{h}_t(y) = \sup_{x\in \R} (h_0(x) + \mathcal{L}(x, 0; y, t))\,,\end{equation*}
for all $y\in \R$ almost surely. This, and the metric composition law inherited from the Brownian LPP, means the directed landscape can be interpreted as a stochastic semi-group. For the narrow wedge initial condition, $h_0(0)=0$ and $h_0(x)=-\infty$ elsewhere, $\mathfrak{h}_1(\cdot)=\mathcal{A}_1(\cdot)$ is the parabolic Airy$_2$ process, that is the top line of the Airy line ensemble. For $h_0\equiv 0$, the flat initial condition, $\mathfrak{h}_1(\cdot)$ is called the Airy$_1$ process. 

The directed landscape at unit time $\mathcal{L}(\cdot, 0; \cdot, 1)$, is also called the \emph{Airy sheet}, and denoted by $\mathcal{S}(\cdot\,, \,\cdot)$. In \cite{DOV}, the authors obtain a coupling between the Airy sheet and the differences in last passage values with respect to the Airy line ensemble. 

Sarkar and Vir\'ag show in \cite{sarkar2021brownian}, that the spatial increments of the KPZ fixed point at any fixed time for general initial data are absolutely continuous with respect to Brownian motion on compacts. One would like to know for which $p\in (1,\infty)$, the Radon-Nikodym derivative of spatial increments of the KPZ fixed point is in $L^p$. In fact, in Conjecture $1.3$ in \cite{hammond2019patchwork}, Hammond conjectures that the Radon-Nikodym derivative of the KPZ fixed point started from arbitrary initial data is in $L^p$ for all $p>1$. This would be a desirable property to have since it would quantitatively strengthen the relationship between low-probability events of Brownian motion and that of the KPZ fixed point and give a way to bound the KPZ fixed point probabilities of events by the corresponding probabilities for Brownian motions (which are easy to compute). This will have a plethora of applications, including tail bounds on the fluctuations of the KPZ fixed point, occurrence of near-maxima at some distance from the unique maximizer of the KPZ fixed point in a compact interval, bounding the number of such near maxima etc. (see, for example, \cite{hammond2019patchwork} for some of these applications for the Airy$_2$ process, all of which will then extend to the KPZ fixed point with general initial data). 

Our result in this paper is an interpolation, (in the sense of Corollary \ref{cor: interpolation orlicz},) between \cite[Theorem 1.2]{sarkar2021brownian} and the above conjecture (Conjecture $1.3$ in \cite{hammond2019patchwork}).

More generally, in the setting of two finite measures $\nu\ll\mu$ ($\mu$ absolutely continuous with respect to $\nu$), one wants if possible to quantify the relationship between the $\delta>0$ and $\varepsilon>0$ so that the implication $\mu(A)<\delta$ guarantees $\nu(A)<\varepsilon$ for all measurable $A$\footnote{Recall the definition of absolute continuity of measures $\nu$ with respect to $\mu$, namely, that for all $\varepsilon>0$, there exists $\delta>0$ such that for all $A$ measurable, if $\mu(A)<\delta$, then $\nu(A)<\varepsilon$.}. This can be achieved if, for instance, one imposes that the Radon-Nikodym derivative $\diff\nu/\diff\mu\in L^p (
\mu)$, for some $p>1$. Then, for $A$ measurable, 
\begin{equation}\label{eq: rn measure bnd}
\nu(A) = \displaystyle \int_{A}\frac{\diff\nu}{\diff\mu}\diff\mu
\leq \Bigl(\int_{A}\Bigl(\frac{\diff\nu}{\diff\mu}\Bigr)^pd\mu\Bigr)^{\frac{1}{p}}(\mu(A))^{\frac{p-1}{p}}
= \left\lVert\frac{\diff\nu}{\diff\mu}\right\rVert_{L^p(\mu)}(\mu(A))^{1-\frac{1}{p}} \,,
\end{equation}
by applying H\"{o}lder's inequality. One can also easily verify that the above inequality is also sufficient to deduce that the Radon-Nikodym derivative exists and $\diff\nu/\diff\mu\in L^{p-}$ \footnote{Indeed, if \eqref{eq: rn measure bnd} holds for some $c> 0$, then we have by Markov's inequality for all $t> 0$, $t\mu(\diff\nu/\diff \mu \ge t)\le \int\diff\nu/\diff \mu\cdot \mathbf{1}(\diff\nu/\diff \mu\ge t)\diff \mu= \nu(\diff\nu/\diff \mu\ge t)\le c(\mu(\diff\nu/\diff \mu\ge t))^{1-\frac{1}{p}}$ and so $\mu(\diff\nu/\diff \mu \ge t)\le c'/t^p$. }. One can relax this type of inequality and impose the following comparison of two measures for all $A$ measurable (in an appropriate measure space) for some Borel function $f:\R_{\ge 0}\to \R_{\ge 0}$ satisfying $\lim_{t \searrow 0}f(t) = 0$, 
\begin{equation}\label{eq: quantitative br reg}
\nu(A) = O(f(\mu(A)))\,.
\end{equation}
When $\mu$ is replaced with various restrictions of the Wiener measure on compacts, we will call this type of estimate a form of \emph{quantitative Brownian regularity} with \emph{rate function} $f$.

In this paper, we prove a form of quantitative Brownian regularity of the KPZ fixed point with a certain rate function $f$ as defined below in Theorem~\ref{thm: main informal}. Such quantitative control on the Radon-Nikodym derivative requires much more detailed analysis of the geometry of the geodesics and many novel ideas and techniques than what is needed to prove just the absolute continuity of the KPZ fixed point. In page 3 in \cite{sarkar2021brownian}, the authors point out that stronger control on the Radon-Nikodym derivative will require new ideas and estimates. Indeed,
except for the use of the Brownian Gibbs property (in fact, we need its strong formulation from the recent paper \cite{dauvergne2024wienerdensitiesairyline}) and the coupling between the Airy sheet and Airy line ensemble, our proof of Theorem~\ref{thm: main informal} in this paper proceeds along very different lines to that in \cite{sarkar2021brownian}. 

Before proceeding further, we need to discuss a bit about the initial condition of the KPZ fixed point. For $t>0$, recall the definition of \emph{$t$-finitary} initial data from \cite{sarkar2021brownian}.
\begin{definition}($t$-finitary)\label{def: finitary}
    Let $h_0:\R\to \R \cup \{-\infty\}$ be a locally bounded measurable function such that
    \begin{equation*}
    \displaystyle\lim_{|x|\to \infty}\frac{h_0(x) - x^2/t}{|x|} = -\infty\,.
    \end{equation*}
\end{definition}

This condition on the initial data (for any $t>0$ fixed) is both necessary and sufficient to guarantee that the KPZ fixed point (at time $t>0$) is globally finite, see \cite[Proposition 6.1]{sarkar2021brownian}.

Now we state the main result of the paper informally. See Theorem~\ref{thm: KPZ reg finitary} for the proper statement of the result. 
{
\begin{theorem}[Quantitative Brownian regularity]\label{thm: main informal}
The spatial increments on a compact interval of the KPZ fixed point at time $1$ started from arbitrary (finitary in the above sense) initial data, exhibit a form of quantitative Brownian regularity with rate function (as defined in \eqref{eq: quantitative br reg}) of the form
\begin{equation*}
f(\mu(A))=\exp\left(-d\log^{r}\big(1/\mu(A)\big)\right)\,, 
\end{equation*}
for all $A$ Borel measurable sets on paths and some positive constants  $d>0, r\in (0,1)$, where $\mu$ denotes an appropriate restriction of the rate two Wiener measure. 
\end{theorem}
}

Moreover, the above rate function turns out to be \textbf{uniform} in the class of initial data $h_0$ supported on some fixed compact set. We will give some applications of the uniform quantitative Brownian regularity in Section \ref{sec: applications}. We give a one-sided large deviation inequality, uniform integrability of Radon-Nikodym derivatives and the integrability of some super-linearly growing convex function of the Radon-Nikodym derivatives, which in particular implies finiteness of entropy.

We believe that the bounds in Theorem~\ref{thm: main informal} can be improved to the point where the exponents $r>1$ and $d>0$ can be tuned appropriately to reduce to the type of bound as in (\ref{eq: rn measure bnd}), which would give the conjectured $L^{\infty-}(\mu)$ control.

The variational characterisation of the KPZ fixed point and the coupling of the Airy sheet with the Airy line ensemble and the Brownian Gibbs property enjoyed by the Airy line ensemble, together imply that a question on Brownian absolute continuity of the KPZ fixed point can ultimately be transferred to that of an `inhomogeneous' Brownian LPP, see Definition \ref{def: inhom brownian lpp}. This was done in \cite[Theorem~4.3]{sarkar2021brownian}, where it was shown that away from zero, inhomogeneous Brownian LPP is absolutely continuous with respect to Brownian motion on compacts. A quantitative Brownian regularity of the KPZ fixed point would thus require, as a first step, a strong control on the Radon-Nikodym derivative of the inhomogeneous BLPP with respect to Brownian motion. This is established in our companion paper \cite[Theorem~7.1]{tassopoulos2025inhomogeneousbrownian} (stated here as Theorem~\ref{thm: main companion}), where it is shown that the Radon-Nikodym derivative of the law of the spatial increments (with endpoints away from zero) of the inhomogeneous BLPP  against the Wiener measure $\mu$ on compacts is in $L^{\infty-}(\mu)$, and in particular, that for any fixed $p>1$, one has that the $L^p$ norm is at most of the order $O_p(\mathrm{e}^{d_pm^2\log m})$ for some $p$-dependent constant $d_p>0$.

Next, the coupling of the Airy sheet with the Airy line ensemble, see Definition \ref{def: Airy sheet}, allows us to express the KPZ fixed point (at unit time) on a compact interval $[1,y_0]$, for $y_0>1$ as 
\begin{equation}\label{eq: KPZ fixed point}
\begin{array}{ll}
    \mathfrak{h}(y) &= \displaystyle\max_{x\in \R}(h_0(x)+\mathcal{S}(x,y))\,,  \quad \text{ for } y\in [1,y_0] \,.
\end{array}
\end{equation}
The right hand side of the equation (\ref{eq: KPZ fixed point}) can be expressed as a maximisation problem of a random depth of last passage percolation values over the Airy line ensemble with random boundary data. 

Now, what remains is to obtain quantitative control of the random depth and the boundary data. Both require a refinement of the picture of the geodesic geometry in the Airy line ensemble. The former will amount to obtaining concentration bounds for the transversal fluctuations of semi-infinite geodesics, which is Theorem~\ref{thm: intercept tail bound}. 

\subsection{Organisation of the paper}
First, in Section \ref{sec: notation} we establish notation that will be used throughout. In Section \ref{sec: preliminaries} we provide necessary background material including estimates of Radon-Nikodym derivatives of the laws of Brownian bridges against that of Brownian motion and some path-wise properties of the Airy line ensemble and its last passage values. Section \ref{sec: geod geom} is devoted to studying geodesic geometry in the Airy line ensemble. In particular, we obtain \emph{exponentially stretched} tail bounds on intercepts of semi-infinite geodesics, Theorem~\ref{thm: intercept tail bound}. Then, in Section \ref{sec: finite depth truncation}, we use the variational formula for the KPZ fixed point and the coupling in Definition \ref{def: Airy sheet} and rely on a series of favourable events and technical inputs from \cite{wu2025applicationsoptimaltransportdyson}, thereby reducing the problem to estimating the Radon-Nikodym derivatives of inhomogeneous Brownian LPP with non-decreasing initial data. For this we use \cite[Theorem~7.1]{tassopoulos2025inhomogeneousbrownian} as a crucial input to obtain an analytically tractable quantitative comparison of finite depth truncations of the KPZ fixed point against the rate two Wiener measure in Theorem~\ref{thm: finite depth KPZ estimates a priori}. In Section \ref{sec: brownian regularity combined} we combine the above estimates to obtain the quantitative Brownian regularity of the KPZ fixed point started from compactly supported upper semi-continuous initial data, namely Theorem~\ref{thm: KPZ law local uniform comparison}. It is here where the tail bounds on the transversal fluctuations play a critical role. Then, by a localisation argument, we extend this quantitative comparison to all finitary initial data to get the main result of this paper, Theorem~\ref{thm: KPZ reg finitary}. In Section \ref{sec: applications}, we give some applications of the main result Theorem \ref{thm: KPZ reg finitary}. In particular,  we provide a one sided large deviation inequality for the KPZ fixed point, Corollary \ref{cor: ldp one sided}, and deduce the finiteness of the relative entropy of the law of the increments of the KPZ fixed point against the Wiener measure on compacts. Finally, in Section \ref{sec: outlook} we briefly outline possible avenues of strengthening the comparison of the KPZ fixed point with respect to Brownian motion. Below is a flowchart depicting the main ingredients in the proof of Theorem~\ref{thm: KPZ reg finitary}.

\begin{figure}[ht]
 \begin{center}
 \begin{NoHyper}
\resizebox{\textwidth}{!}{\begin{tikzpicture}[
    node distance=1.5cm and 1.5cm, 
    startstop/.style={rectangle, minimum width=3cm, minimum height=1cm, text centered, draw=black},
    process/.style={rectangle, minimum width=3cm, minimum height=1cm, text centered, draw=black, align=center},
    arrow/.style={thick,->,>=stealth,shorten >=2pt,shorten <=2pt} 
  ]

  \node (start) [startstop, align=center] 
    {Brownian regularity of the KPZ fixed point\\ started from finitary initial data};

    \node (finitary reg) [process, below=1.6cm of start, align=center] 
    {Brownian regularity of the KPZ fixed point\\ started from compactly-supported initial data};

  \node (above) [process, below=1.6cm of finitary reg] 
    {Brownian regularity of finite depth truncations of the KPZ fixed point };

  \node (coupling) [process, below=1.6cm of above] 
    {Brownian regularity of Airy line ensemble last passage values};

  \node (extra) [process, below left= 0.8cm and -15cm of coupling] 
    {Geodesic geometry\\
     of the parabolic Airy line ensemble:\\
     transversal fluctuation};

  \node (gibbs) [process, below left= 0.8cm and -5cm of coupling] 
    {Brownian Gibbs resampling\\
     and inhomogeneous\\
     Brownian LPP $L^p$ norm estimates};

   \draw [arrow] (gibbs) -- node[midway, left]
    {\small\cite[Theorem~7.1]{tassopoulos2025inhomogeneousbrownian}} (coupling);

  \draw [arrow] (coupling) -- node[midway, right]
    {Theorem~\ref{thm: finite depth KPZ bounds}} (above);
  \draw [arrow] (extra) to[out=85, in=0, looseness=1.0] 
    node[pos=0.5, sloped, transform shape, anchor=south, font=\scriptsize, yshift=2pt, align=center]{Theorem~\ref{thm: intercept tail bound}} (finitary reg); 
  \draw [arrow] (above) -- node[midway, right]{Theorem~\ref{thm: KPZ law local uniform comparison}} (finitary reg); 
  \draw [arrow] (finitary reg) -- node[midway, right]{Theorem~\ref{thm: KPZ reg finitary}} (start); 
\end{tikzpicture}}
\end{NoHyper}
\end{center}
   \caption{Flowchart of main steps in the proof of Theorem~\ref{thm: KPZ reg finitary}.}
  \label{fig:flowchart}
\end{figure}

\subsection{Related works}

The Brownian nature of models in the KPZ universality class, including the Airy line ensemble and the KPZ fixed point, has been a subject of intense research in recent times. Aside from integrable inputs, see for instance \cite{baik1999distribution, matetski2016kpz, liu2019multi} and \cite{johansson2019multi, johansson2017two, johansson2018two}, probabilistic and geometric methods have featured prominently ever since Corwin and Hammond proved in \cite{corwin2014brownian} that the parabolic Airy line ensemble admits a Brownian Gibbs resampling property. For a more detailed account of recent developments, one can consult the work of Calvert, Hammond and Hegde \cite{hammond2019patchwork} and the references therein.

One version of local Brownianness is to show that the local limits of the
$\text{Airy}_2$ process (the narrow wedge solution to the KPZ fixed point at unit time) converge in law to a Brownian motion, \cite{hagg2008local}, \cite{cator2015local}, \cite{quastel2013local}. In fact, \cite{quastel2013local} also establishes H\"{o}lder $1/2$-
continuity of the $\text{Airy}_2$ and $\text{Airy}_1$ processes (solution to KPZ fixed point at unit time started from flat, i.e. $h_0\equiv 0$ initial data). The  H\"{o}lder $1/2$- continuity and the locally Brownian nature (in
terms of convergence of the finite dimensional distributions) were established in \cite{matetski2016kpz}. Such H\"{o}lder continuity
results and local limits for certain initial conditions have also been established in \cite{pimentel2018local} and \cite{pimentel2020brownian} (see also \cite{johansson2017two},
\cite{johansson2018two}). A stronger notion of the locally Brownian nature is absolute continuity with respect to Brownian motion
on compact intervals, which we call Brownian on compacts. That the $\text{Airy}_2$ process is Brownian on compacts
was first proved in \cite{corwin2014brownian} using the Brownian Gibbs property.

Very recently, the locally Brownian nature of the Airy line ensemble (and so for $\text{Airy}_2$ process) has been considerably strengthened in \cite{dauvergne2024wienerdensitiesairyline}, where Dauvergne gave an explicit form for the density of the finite depth truncations of increments of the Airy line ensemble against Brownian motion on compacts and established ways of estimating inverse acceptance probabilities following ideas from the `tangent method'. This geometric approach, first introduced in \cite{aggarwal2022arcticboundariesicemodel}  was also used in \cite{ganguly2024sharpuppertailbehavior} to provide sharp estimates for the one-point density of the KPZ fixed point started from fairly general initial data at the origin. Also, in \cite{wu2025applicationsoptimaltransportdyson}, Wu introduced ideas from the theory of optimal transportation to the study of spatial regularity of the Airy line ensemble and has provided sub-Gaussian tail bounds (with universal coefficients) on the modulus of continuity of any given level of the Airy line ensemble on a compact interval.

However, for general initial conditions, the picture is less clear with more questions open. A result providing a more quantitative notion of Brownian regularity, called \emph{patchwork quilt of Brownian fabrics}, was established in Hammond \cite{hammond2019patchwork} and \cite{hammond2019patchwork}. Roughly the result states that the KPZ fixed point $\mathfrak{h}(\cdot)\equiv \mathfrak{h}_1(\cdot)$ on a unit interval is the result of `stitching' a random number of profiles, where each profile is absolutely continuous with respect to a Brownian motion with Radon-Nikodym derivative in $L^p$ for all $p<3$. The author conjectured (Conjectured $1.3$ in \cite{hammond2019patchwork}) that one can dispense with these random patches and establish $L^p$ estimates for all $p>1$ for the Radon-Nikodym derivative, a problem which remains open. A first step in this direction was the proof of absolute continuity on a single non-random patch for general initial conditions, which has been established in \cite[Theorem~1.2]{sarkar2021brownian}, using methods different from those in \cite{hammond2019patchwork}.

Our main result in Theorem~\ref{thm: KPZ reg finitary} of this paper strengthens quantitatively the absolute continuity result of \cite{sarkar2021brownian} for finitarty initial data (see Definitions \ref{def: max support} and \ref{def: finitary}). Our proof of this result crucially depends upon refining certain aspects of the construction of the directed landscape in \cite{DOV}, the variational characterisation of the KPZ fixed point from \cite{sarkar2021brownian}, the Brownian Gibbs property of the parabolic Airy line ensemble established in \cite{corwin2014brownian}, the strong comparison against Brownian motion on compacts of inhomogeneous Brownian LPP (the Radon-Nikodym derivative of the law of the spatial increments against the Wiener measure $\mu$ on compacts being in $L^{\infty-}(\mu)$) established in \cite{tassopoulos2025inhomogeneousbrownian} (stated here as Theorem~\ref{thm: main companion}), as well as technical inputs from \cite{dauvergne2021} and \cite{wu2025applicationsoptimaltransportdyson} used in estimating Brownian inverse acceptance probabilities with random boundary points and global modulus of continuity estimates for the stationary version of the Airy line ensemble respectively. 

\section{Notation}\label{sec: notation}

We introduce some notation and conventions we will be using throughout.

When in some estimates a constant appears that will depend on some parameters $a,b,c,\cdots$, it will be denoted by $C_{a,b,c,\cdots}$, unless otherwise specified. Constants without subscripts are deemed to be universal. Additionally, for ease of notation, such constants are allowed to change from line to line. Moreover,  such constants may be dropped and instead replaced with the symbols $\lesssim_{a,b,c,\cdots}(\equiv O_{a,b,c,\cdots}(\cdot))$ and $\gtrsim_{a,b,c,\cdots}$ for some parameters $a,b,c,\cdots$ which stand for $\le C_{a,b,c,\cdots}$ and $\ge C'_{a,b,c,\cdots}$ for some positive constants $C_{a,b,c,\cdots}, C'_{a,b,c,\cdots}$ respectively.

We take the set of natural numbers $\N$ to be $\{1,2, \ldots \}$. For $k\in \N$, we use an underbar to denote a $k$-vector, that is, $\underline{x}\in \R^k$. We denote the integer interval $\{i,i+1, \ldots, j\}$ by $\llbracket i, j\rrbracket$. A $k$-vector $\underline x = (x_1, \ldots, x_k)\in \R^k$ is called a $k$-decreasing list if $x_1>x_2> \ldots >x_k$. For a set $I\subseteq \R$, let $I^k_> \subseteq I^k$ be the set of $k$-decreasing lists of elements of $I$, and $I^k_{\geq}$ be the analogous set of $k$-non-increasing lists.

The symbols $\cdot \land \cdot , \cdot \lor \cdot$ denote $\min\{\cdot, \cdot \}$ and $\max\{\cdot, \cdot \}$ respectively. For any $a
\in \R$, $a_+$ denotes $a\lor 0$ and $a_-$ denotes $-a\lor 0$.

 We define the affinely shifted bridge version, that is zero at both endpoints, of a real-valued function $f$ on an interval $[a,b]$, $f^{[a,b]}:[a,b]\to \R$ by
\begin{equation}\label{eq: affine shift}
f^{[a,b]}(x) := f(x) - \frac{x-a}{b-a}\cdot f(b) - \frac{b-x}{b-a}\cdot f(a)
\end{equation}
for all $x\in [a,b]$. 

We now turn to some notational conventions for the path spaces that will be used throughout. For general domains of paths $J$, we denote the space of continuous paths, in the usual topologies, by $\mathscr{C}_{*,*}(J, \R)$. More specifically, if the domain is an interval $[a,b]\subseteq \R$, we denote the space of continuous functions with domain $[a,b]$ which vanish at $a$ by $\mathscr{C}_{0,*}([a,b], \R)$. For random functions taking values in these spaces, we will always endow them with their respective Borel $\sigma$-algebras generated by the topologies of uniform convergence (which makes them into Polish spaces). Similarly, for $k\ge 1, a < b$, define $\mathscr{C}^k_{*,*}([a,b], \R)\equiv \prod_{i=1}^k \mathscr{C}_{*,*}([a,b], \R)$ and equip it with the product of the uniform topologies. Furthermore, for $a<b$, $k \in \N$ and $\underline{x},\underline{y}\in \R^k_{>}$, let $\mathscr{C}^k_{\underline{x}, \underline{y}}([a,b], \R)$ denote the space $\Big\{g \in \mathscr{C}^k_{*, *}([a,b], \R): \forall i\in \llbracket 1, k\rrbracket,\,  g_{i}(a) = x_i \; \text{and}\; g_{i}(b) = y_i\Big\}$. 

We say that a Brownian motion or a Brownian bridge has \emph{rate $v$} if its quadratic variation in an interval $[s,t]$ is equal to $v(t-s)$. We say that a Dyson's Brownian motion or a Brownian $k$-melon has rate $v$ if the component Brownian motions have rate $v$. From now on, all Brownian motions are rate two unless stated otherwise. 

For $0\leq a <b$, in analogy to the above, let $\mathfrak{B}^{[a,b]}_{*, *}(\cdot)$ denote the law of a rate two Brownian motion on $[0,\infty)$ starting from the origin restricted to the interval $[a,b]$ (the two star symbols indicate that the Brownian motion starts from the origin at time zero, which might be outside of the interval $[a,b]$). When $k\ge 1$ independent copies are considered, we will be using the usual product measure notation $(\mathfrak{B}^{[a,b]}_{*, *})^{\otimes k}$. Moreover, for $\underline{x}, \underline{y}\in \R^k$ let $\mathfrak{B}^{[a,b]}_{\underline{x}, \underline{y}}(\cdot)$ denote the law of $k$ independent rate two Brownian bridges on $[a,b]$ with endpoints $(a, \underline{x})$ and $(b, \underline{y})$, hence it is a measure on $\mathscr{C}^k_{\underline{x}, \underline{y}}([a,b]\,, \R)$ equipped with the usual Borel $\sigma$-algebra on the product topology of local uniform convergence. 

For $k\in \N$, $a<b$, $\underline{x},\underline{y}\in \R^k_{>}$ and  $f:[a,b] \to \R$ a measurable function such that $x_{k}>f(a)$ and $y_k>f(b)$, the \emph{non-crossing} event on any fixed union of finite sub-intervals $J \subseteq [a,b]$ is denoted by
\begin{align}\label{eq: noint}
    \mathrm{NoInt}(J\,,f) :=\Big\{&g \in \mathscr{C}^k_{*, *}(J, \R):  \textrm{ for all }   r\in J, 1\leq i<j\leq k, \nonumber \\
     &g_{i}(r) > g_{j}(r)\,, g_i(a)=x_i, g_i(b)=y_i\textrm{ and } g_k(r)>f(r) \Big\}\,.
\end{align} 
In what is to follow, the probability $ \mathfrak{B}^{[a,b]}_{\underline{x},\underline{y}}(\mathrm{NoInt}(J\,,f))$ is called an {\it acceptance probability}. Roughly speaking, it is the probability of the event that a collection of $k$ independent Brownian bridges on $[a,b]$ with endpoints $\underline{x}, \underline{y}$ do not intersect, and also stay above the `lower barrier' $f$ on $J$. We note this event has a positive probability owing to standard facts of Brownian bridges, see Section 2.2.2 in \cite{corwin2014brownian}.

\section{Preliminaries}\label{sec: preliminaries}
In this section, we will recall some basic definitions that appear in the KPZ universality class, namely, last passage percolation, the Pitman transform and melons; and collect some basic results that will be useful later;  these include some elementary estimates involving Brownian bridge, Radon-Nikodym derivatives (against Brownian motion) estimates and pathwise properties of the Airy line ensemble. We start with the central probabilistic object of study, namely random line ensembles.

\subsection{Line ensembles} The following definition makes precise the notion of a \emph{random line ensemble}, a probabilistic object of central importance in the KPZ universality class. It is a random variable taking values in an indexed (at most countably infinite) family of continuous paths defined on a common subset of $\R$.

\begin{definition}[Random ensemble]\label{def: random ensemble}
    Let $\Sigma$ be a (possibly infinite) interval of $\Z$, and let $\Lambda$ be an interval of $\R$. Consider the set \(X\equiv \mathscr{C}^\Sigma\) of continuous functions $f: \Lambda \times\Sigma\rightarrow \R$. We endow it with the topology of uniform convergence on compact subsets of $\Lambda \times\Sigma$. Let $C$ denote the sigma-field  generated by Borel sets in $X$.

    A {\it $\Sigma$-indexed line ensemble} $L$ is a random variable defined on a probability space $(\Omega,\mathcal{B},\PP)$, taking values in $X$ such that $L$ is a $(\mathcal{B},C)$-measurable function. Furthermore, we write $L_i\equiv(L(\omega))(i,\cdot)$ for the line indexed by $i\in\Sigma$. 
\end{definition}

\subsection{Last passage percolation} We begin with the collection of some preliminary facts regarding last passage percolation (sometimes abbreviated as LPP in the paper) over ensembles of functions following \cite{DOV}.

Formally, let \(I\subset \mathbb{Z}\) be a possibly finite index set and define the space \(\mathscr{C}^{I}\) of sequences of continuous functions with real domains, that is, the space
\begin{equation*}f: \mathbb{R}\times I\to \mathbb{R}\quad (x,i)\mapsto f_i(x)\,.\end{equation*}

\begin{definition}[Path] Let \(x\leq y \in \mathbb{R}\), and \(m\leq \ell \in \mathbb{Z}\) respectively. A \textbf{path}, from \((x,\ell)\) to \((y,m)\) is a non-increasing  function \(\pi: [x,y] \to \mathbb{N}\) which is cadlag on \((x,y)\) and takes the values \(\pi(x)= \ell\) and \(\pi(y)= m\).
\label{def: path}
\end{definition}

\begin{remark}
    The convention that the paths be non-increasing is so that they match the natural indexing of the Airy line ensemble, see Section \ref{subsec: Airy line ensemble}. 
\end{remark}

In what is to follow, since we will primarily be considering the Airy line ensemble (see Section \ref{subsec: Airy line ensemble} for a definition), we will take the indexing set to be \(I = \mathbb{N}\). We now define an important quantity associated to each such path, namely, its \emph{length} as the sum of increments of \(f\) along \(\pi\). This also leads one to naturally define a derived quantity, namely the \emph{last passage value}.

\begin{definition}[Length]\label{def: length} Let \(x\leq y \in \mathbb{R}\) and \(m < k\in\mathbb{Z}\). For each \(m\leq i <k\), let \(t_{k-i}\) denote the jump of the path \(\pi\), on an ensemble \((f_i)_{i\in I}\), from \(f_{i+1}\) to \(f_{i}\). Then the length of \(\pi\) is defined as
\begin{equation*}
\ell(\pi) = f_m(y)-f_m(t_{k-m}) + \displaystyle\sum_{i = 1}^{\ell-m-1}(f_{k-i}(t_{i+1})-f_{k-i}(t_{i}))+f_{k}(t_{1})-f_{k}(x)\,.
\end{equation*}
\end{definition}

\begin{figure}[ht]
  \centering
    \begin{tikzpicture}[scale=1.0]

    \def\startx{1}
    \def\endx{9}
    \def\startg{2.5}
    \def\endg{7.5}
    \def\spacing{2.0}
    \def\amp{0.15}

    \pgfmathsetseed{2025}

    \pgfmathsetmacro{\tOne}{\startg+1*(\endg-\startg)/4}
    \pgfmathsetmacro{\tTwo}{\startg+2*(\endg-\startg)/4}
    \pgfmathsetmacro{\tThree}{\startg+3*(\endg-\startg)/4}

    \foreach \i in {1,...,4} {
      \draw[thick, blue!60]
        plot[domain=\startx:\endx, samples=100, smooth]
          (\x, { \i*\spacing - 2.5 - 0.05*(\i)*(\x-5)^2 + (2*rnd-1)*\amp });

      \pgfmathsetmacro{\segA}{\startg+(\i-1)*(\endg-\startg)/4}
      \pgfmathsetmacro{\segB}{\startg+\i*(\endg-\startg)/4}
      \draw[thick, Green]
        plot[domain=\segA:\segB, samples=100, smooth]
          (\x, { \i*\spacing - 2.5 - 0.05*(\i)*(\x-5)^2 });

      \pgfmathsetmacro{\xmid}{0.5*(\segA+\segB)}
      \pgfmathsetmacro{\ytoplabel}{\i*\spacing - 2.5 - 0.05*(\i)*(\xmid-5)^2 + 0.3}
      \node[Green, above] at (\xmid, \ytoplabel) {$\Delta_{\i}$};

      \ifnum\i>1
          \pgfmathsetmacro{\jumpx}{\startg+(\i-1)*(\endg-\startg)/4}
          \pgfmathsetmacro{\ylow}{(\i-1)*\spacing - 2.5 - 0.05*(\i-1)*(\jumpx-5)^2}
          \pgfmathsetmacro{\yhigh}{\i*\spacing - 2.5 - 0.05*(\i)*(\jumpx-5)^2}

          \draw[Green, dashed] (\jumpx,\ylow) -- (\jumpx,\yhigh);
          \draw[Green, thick, fill=white] (\jumpx,\ylow) circle (2pt);
          \fill[Green] (\jumpx,\yhigh) circle (2pt);
      \fi
    }

    \pgfmathsetmacro{\ystart}{1*\spacing - 2.5 - 0.05*(1)*(\startg-5)^2}
    \pgfmathsetmacro{\yend}{4*\spacing - 2.5 - 0.05*(4)*(\endg-5)^2}
    \fill[Green] (\startg,\ystart) circle (2pt);
    \fill[Green] (\endg,\yend) circle (2pt);

    \node[left] at (\startg, \ystart+0.2) {$(x,4)$};
    \node[right] at (\endg, \yend) {$(y,1)$};

    \draw[->] (\startg-0.5,-1.5) -- (\endg+0.5,-1.5);

    \pgfmathsetmacro{\ytOne}{1*\spacing - 2.5 - 0.05*(1)*(\tOne-5)^2}
    \pgfmathsetmacro{\ytTwo}{2*\spacing - 2.5 - 0.05*(2)*(\tTwo-5)^2}
    \pgfmathsetmacro{\ytThree}{3*\spacing - 2.5 - 0.05*(3)*(\tThree-5)^2}

    \foreach \px/\py/\lab in {
        \startg/\ystart/$x$,
        \tOne/\ytOne/$t_1$,
        \tTwo/\ytTwo/$t_2$,
        \tThree/\ytThree/$t_3$,
        \endg/\yend/$y$
    } {
      \draw[dotted, shorten >=2pt] (\px,-1.5) -- (\px,\py);
      \filldraw[black] (\px,-1.5) circle (2pt);
      \node[below] at (\px,-1.5) {\lab};
    }

\end{tikzpicture}
   \caption{Visualisation of a possible path (\color{Green}green\color{black}) “embedded” in the Airy line ensemble (\color{blue}blue\color{black}), here \((\mathcal{A}_1, \mathcal{A}_2, \mathcal{A}_3, \mathcal{A}_4)\) from top to bottom, and \(m = 1, k = 4\) (see Section \ref{subsec: Airy line ensemble}). Here $\Delta_1 = \mathcal{A}_4(t_1)-\mathcal{A}_4(x)$, $\Delta_2 = \mathcal{A}_3(t_2)-\mathcal{A}_3(t_1)$, $\Delta_3 = \mathcal{A}_2(t_3)-\mathcal{A}_2(t_2)$, $\Delta_4 = \mathcal{A}_1(y)-\mathcal{A}_1(t_3)$ and $\ell = \sum_{i=1}^4 \Delta_i.$}
\end{figure}

\begin{definition}[Last passage value]\label{Definition: last passage}
    With \(x\leq y, m<k\) as before and \(f\in \mathscr{C}^{I}\), define the \textbf{last passage value} of \(f\) from \((x,k)\) to \((y,m)\) as
    \begin{equation*}
    f[(x,k)\to(y,m)] \stackrel{\mathrm{def}}{=}\displaystyle \sup_{\pi}\ell(\pi)\,,
    \end{equation*}
where the supremum is over precisely the paths \(\pi\) from \((x,k)\) to \((y,m)\).
\end{definition}
\begin{remark}
    Any path \(\pi\) from \((x,k)\) to \((y,m)\) such that its length is equal to its last passage value is called a \textbf{geodesic}. To establish the existence of geodesics one can proceed by first noticing that the length of a path \(\ell(\pi)\), can be viewed as a function on the subset \(\mathcal{Z}\) of non-increasing cadlag functions with fixed endpoints in \(\mathbf{D}\), the space of cadlag functions \(\mathbf{D}\stackrel{\mathrm{def}}{=}\mathbf{D}([x,y], \N)\). When endowed with respect to the  Skorokhod topology, which is metrisable, the above function is continuous. Since \(\mathcal Z\) is closed with respect to the above topology of \textquotedblleft jump times\textquotedblright, a compactness argument using Arzela-Ascoli, see \cite[ch. 3]{billingsley2013convergence}, implies that the supremum over admissible paths is indeed attained.
\end{remark}

Last passage percolation enjoys the following \textbf{metric composition law}, Lemma 3.2 in DOV \cite{DOV}.

\begin{lemma}[Metric composition law]\label{Lemma: Metric Composition}
    Let \(x\leq y \in \mathbb{R}\), \(m < \ell\in\mathbb{Z}\) and \(f\in \mathscr{C}^I\). If \(k\in \{m, \dots, \ell\}\), then we have
    \begin{equation*}
    f[(x,\ell)\to(y,m)] = \displaystyle \sup_{z\in[x,y]}(f[(x,\ell)\to(z,k)]+f[(z,k)\to(y,m)])\,,
    \end{equation*}
    and if \(k\in \{m+1, \dots, \ell\}\), then 
    \begin{equation*}
    f[(x,\ell)\to(y,m)] = \displaystyle \sup_{z\in[x,y]}(f[(x,\ell)\to(z,k)]+f[(z,k-1)\to(y,m)])\,.
    \end{equation*}
    Furthermore for any \(z\in [x,y]\), 
    \begin{equation}\label{eq: composition}
    f[(x,\ell)\to(y,m)] = \displaystyle \sup_{k\in \{m, \dots, \ell\}}(f[(x,\ell)\to(z,k)]+f[(z,k)\to(y,m)])
    \end{equation}
\end{lemma}

We are now in a position to state the main result of \cite{tassopoulos2025inhomogeneousbrownian} which gives pathwise estimates for the Radon-Nikodym derivatives of Brownian LPP started from inhomogeneous `initial data', that will be crucial in obtaining quantitative Brownian regularity of the KPZ fixed point.

First we define inhomogeneous Brownian LPP started from non-increasing initial data.

\begin{definition}[Inhomogeneous Brownian LPP]\label{def: inhom brownian lpp}
    Fix $m\ge 1$, $B_1, \cdots, B_m$ be independent Brownian motions starting from the origin, $\underline{g} = (g_\ell)_{\ell =1}^m \in\R^m_{\ge}$ and $B=(B_1,
    \ldots,B_m)$. Then, the process 
     \begin{equation*}
     \displaystyle\max_{1\le \ell \le m}(g_{\ell} + B[(0,\ell)\to (y, 1)])\,, \qquad  y\in [0,\infty)
     \end{equation*}
     is called the \textbf{inhomogeneous Brownian LPP} started from the initial data $\underline{g}$.
\end{definition}

Now we can proceed to the statement of the main result of \cite{tassopoulos2025inhomogeneousbrownian}.

\begin{theorem}(\cite[Theorem~7.1]{tassopoulos2025inhomogeneousbrownian})\label{thm: main companion}
     Fix $m\ge 1$, $(g_\ell)_{\ell =1}^m \in\R^m_{\ge}$ and let $\mathfrak{h}(\cdot)\equiv \mathfrak{h}_1(\cdot)$ be the inhomogeneous Brownian LPP started from the initial data $\underline{g}$. Then, for all $0<\ell<r<\infty$, we have that the Radon-Nikodym derivative of the law of $\mathfrak{h}(\cdot)\equiv \mathfrak{h}_1(\cdot)$ against a rate two Brownian motion starting from the origin $\mu$ on $[\ell,r]$ is in $L^{\infty-}(\mu\vert_{[\ell, r]})$. In particular, with $\xi_{\ell, r, m, \underline{b}}$ denoting the law of $H$ as defined above on $[\ell, r]$
     \begin{equation*}
         \norm{\frac{\diff \xi_{\ell, r, m, \underline{b}}}{\diff\mu\vert_{[\ell, r]}}}_{L^p(\mu\vert_{[\ell, r]})} =  O_p(\mathrm{e}^{d_pm^2\log m}),\quad \mathrm{for all}\quad p>1.
     \end{equation*}
     for some universal in $m\in \N$ (though possibly $p$-dependent) constant $d_p>0$ for all $p>1$.

     In particular, we obtain the estimates
     \begin{equation*}
     \begin{array}{cc}
         \norm{\frac{\diff {\xi}_{\ell, r, m, \underline{b}}}{\diff\mu}}_{L^p(\mu)} &= \displaystyle\prod^m_{i=1}\exp\big(-(b_i-b_m)^2/(4\ell)\big)\cdot \left(\frac{(b_1-b_m)}{2\ell}\lor 1\right)^{m^2}\\
         & \quad \cdot O_{p,\ell, r}\Bigg(\mathrm{e}^{d m^2\log m + c_{\ell} \big(\sum_{i=1}^m (b_i-b_m)\big)^2}\Bigg),
    \end{array}
     \end{equation*}
     for some constants $c_{\ell,r},d>0$ independent of $m\in \N$ and all $p>1$.
\end{theorem}

\subsection{Pitman transform}

Recall that with $f=(f_1,f_2)$ where $f_i:[0,\infty)\mapsto \R$ for $i=1,2$, for $f\in \mathscr{C}^2_{*,*}([0,\infty))$, we define $\mathrm \mathrm{W}f=(\mathrm{W} f_1,\mathrm{W} f_2)\in \mathscr{C}^2_{*,*}([0,\infty))$, the \textbf{Pitman transform} of $f$ as follows. For $x<y\in [0,\infty)$, define the maximal gap size
\begin{equation*}G(f_1,f_2)(x,y)\equiv\max\left(\max_{s\in [x,y]}\big(f_2(s)-f_1(s)\big)\,,\,0\right)\,.\end{equation*}
Then define for all $t\in [0,\infty)$
\begin{equation}\label{eq: pitmantrans}
\mathrm{W} f_1(t)=f_1(t)+G(f_1,f_2)(0,t)\,, \quad \mbox{ and }\quad \mathrm{W} f_2(t)=f_2(t)-G(f_1,f_2)(0,t)\,.
\end{equation}

One can express the top line of the Pitman transform in terms of last passage values. 

\begin{lemma}\label{lemma: Pitman melon}
    Let \(f\in \mathscr{C}^2_{*,*}([0,\infty))\) and let \(Wf = (Wf_1, Wf_2)\) be as above. Then for all \(t\in[0,\infty)\),
    \begin{equation*}
    Wf_1(t) = \displaystyle \max_{i=1,2}\{f_i(0) + f [(0, i) \to (t, 1)]\} \,.
    \end{equation*}
\end{lemma}

\begin{proof}
    By definition, 
    \begin{align*}
        Wf_1(t)&= f_1(t)+G(f_1, f_2)(0,t)= f_1(t) + \max\{\displaystyle \max_{s\in[x,y]}(f_2(s)-f_1(s)), 0 \}\\
        &= \max\{\displaystyle \max_{s\in[x,y]}(f_2(s)+f_1(t)-f_1(s)), f_1(t) \}\,.
    \end{align*}
    From \eqref{eq: composition}, we get \(f_1(t) = f_1(0) + f [(0, 1) \to (t, 1)]\) and \begin{equation*} \displaystyle \max_{s\in[0,t]}(f_2(s) + f_1(t) - f_1(s)) = f_2(0) + f [(0, 2) \to (t, 1)]\,.\end{equation*} Combining the above gives the result. 
\end{proof}

Particularly in the case where \(f_1(0) = f_2(0) = 0\), we obtain that 
\begin{equation*}
Wf_1(t) = f [(0, 2) \to (t, 1)]\,.
\end{equation*}
\(Wf\) is commonly referred to as the \textbf{2-melon} (which will be generalised in the following section to the so-called $n$-melons) of $f$, since paths in \(Wf\) avoid each other and thus resemble the stripes of a watermelon.

\subsection{Dyson Brownian motion} 

Fix any $\varepsilon, t > 0$ and let $B^n$ be the collection of $n$ independent Brownian motions with initial conditions $B^n_i(0) = 0$ conditioned not to intersect on $[\varepsilon, t]$ (note the non-intersection event has positive probability). Then, as $\varepsilon \searrow 0$, $t \nearrow \infty$, Kolmogorov's extension theorem~gives that the $B^n$ converges in law to a limiting process, namely, $n$-level Dyson Brownian motion.

An alternative construction is to first take $x\in \R^n_{>}$ and with $ \PP_x$ denoting the law of $n$ independent Brownian motions $B$ started at $x$ and $\hat{\PP}_x$ the law of the {Doob's} $h$-transform of $B$ started at $x$, where $h(x_1, x_2, \cdots, x_n) = \prod_{1\le i < j\le n}(x_i-x_j)_+$. Then the weak limit of $\hat{\PP}_x$ as $\R^n_> \ni x \to 0$ can be realised as a random ensemble with law on paths $\hat{\PP}_{0^+}$ which agrees with the $n$-level Dyson Brownian motion starting from the origin. The advantage of this construction is that it is more amenable to Radon-Nikodym derivative estimates.

It is worth mentioning that the Dyson Brownian motion was initially described as the eigenvalues of $n \times n$ time-dependent Hermitian matrices with entries independent complex-valued Brownian motion, \cite{dyson1962brownian}. 

\subsection{Melons}\label{subsec: melons}An application of the above that is of interest is that of two independent standard Brownian motions (starting from zero) \(B = (B_1, B_2)\).  Let \(\hat{B}=(\hat{B}_1, \hat{B}_2)\) be two independent Brownian motions conditioned not to collide, in the sense of Doob (a $2$-Dyson Brownian motion). Then, the law of the melon \(WB\) as defined above in (\ref{eq: pitmantrans}) is the same as that of \(\hat{B}\). In \cite{o2002representation}, a generalisation was proved for \(n\) Brownian motions, using a continuous analogue of the Robinson–Schensted–Knuth (RSK) correspondence, where each level in the \emph{$n$-melon} $WB^n = (WB^n_1, WB^n_2, \cdots, WB^n_n)$ is obtained from a family of $n$ Brownian motions by a sequence of deterministic operations that are analogous to the sorting algorithm `bubble sort' where the top curve $WB^n_1$ coincides with the top level of an $n$-Dyson Brownian motion. The term melon comes from the ordering of paths: for some continuous $n-$tuple $f$, $(Wf)^n_1 \ge (Wf)^n_2 \ge \dots \ge (Wf)^n_n$ and their initial value which is $0$, which means they look like stripes on a watermelon. When clear from context, we will abuse notation and drop the superscript, writing instead $Wf$.

In particular, \cite[Proposition 4.1]{DOV} gives an important property of melon paths in that they preserve last passage values (with no restriction on their starting point). In particular, 
\begin{equation*}
WB[(0,n)\to (t,1)] = B[(0,n)\to (t,1)]\,,\quad t\ge 0\,.
\end{equation*}
Using the fact that $WB^n(0) = \underline{0}$ and the ordering of melon paths, one gets that the left-hand-side of the above equation is just $WB_1(t)$. Thus the top line of melon paths is completely characterised in terms of Brownian last passage percolation. For a more complete definition of melons involving the remaining lines, see \cite[sec. 2]{DOV} and \cite{o2002representation}.

\subsection{Airy line ensemble and the Brownian Gibbs property}\label{subsec: Airy line ensemble}
After appropriate rescaling, $WB^n$ converges in law as $n\to \infty$ to a non-intersecting random
ensemble \(\mathcal{A} = (\mathcal{A}_1, \mathcal{A}_2, \dots)\) in $\mathscr{C}^{\N}$ (see Theorem~2.1 in \cite{DOV}), such that $\mathcal{A}_1 > \mathcal{A}_2 > \cdots$. The random ensemble $\mathcal{A}$ is called the (\textbf{parabolic)
Airy line ensemble}. It was introduced by Pr\"{a}hofer and Spohn \cite{prahofer2002scale} in the version \((\mathcal{A}^{\mathrm{stat}}_i)_{i\in \N}\stackrel{\mathrm{def}}{=}(\mathcal{A}_i(\cdot)+(\cdot)^2)_{i\in \N}\), which is stationary
in time, see also \cite{corwin2014brownian} and \cite{corwin2014ergodicity}. We will thus call it the \textbf{stationary Airy line ensemble}. The top line $\mathcal{A}_1$ is known as the parabolic $\text{Airy}_2$ process that appears as the
limiting spatial fluctuation of random growth models starting from a single point.

\begin{theorem}\label{thm: melon scaling Airy}
    Let \(WB^n\)  be a Brownian \(n\)-melon.  Define the rescaled melon \(A^n = (A^n_1, \dots, A^n_n)\) by
\begin{equation*}
A^n_i(y) = n^{1/6} \left((WB^n)_i(1 + 2yn^{-1/3}) - 2\sqrt{n} - 2yn^{1/6} \right).
\end{equation*}
Then \(A^n\) converges to a random sequence of functions \(\mathcal{A} = (\mathcal{A}_1, \mathcal{A}_2, \dots) \in \mathscr{C}^\mathbb{N}\) in law with respect to product of uniform-on-compact topology on \(\mathscr{C}^\mathbb{N}\). For every \(y \in \mathbb{R}\) and \(i < j\), we have \(\mathcal{A}_i(y) > \mathcal{A}_j(y)\). The function \(\mathcal{A}\) is called the \textbf{(parabolic) Airy line ensemble.} 
\end{theorem}

We now recall the Brownian Gibbs resampling property enjoyed by the Airy line ensemble (see Figure \ref{fig: Gibbs}), first established in \cite{corwin2014brownian}. Informally, it states that for $a< b$, $k\in \N$, the law of the Airy line ensemble restricted to $\{1,2,\cdots,k\}\times(a,b)$, ${\mathcal{A}}|_{\{1,2,\cdots,k\}\times(a,b)}$, conditionally on all the data generated by the Airy line ensemble outside of this region, $\mathscr{F}_k\equiv \sigma(\{A_i(x): (i,x)\notin \llbracket 1,k\rrbracket \times (a,b)\})$, is given by non-intersecting Brownian bridges with entry data $\underline{x} = (\mathcal{A}_i(a))_{1\leq i\leq k}$, $\underline{y} = (\mathcal{A}_i(b))_{1\leq i\leq k}$ and also conditioned to stay above $f = \mathcal{A}_{k+1}$ on $(a,b)$. More precisely, the Brownian Gibbs property allows us to specify the regular conditional distribution
\begin{equation*}
\mathrm{Law}\Bigl(\mathbf{\mathcal{A}}|_{\{1,2,\cdots,k\}\times(a,b)}\, \mathrm{ conditioned\, on }\, \mathscr{F}_k\Bigr) = \mathfrak{B}_{\underline{x},\underline{y}}^{f, [a,b]}\,,
\end{equation*}

\begin{figure}[t]
    \centering
    \begin{tikzpicture}[scale = 1.2, 
    block/.style={draw, thick, minimum height=3.5cm, minimum width=3.5cm}]
    
    \draw[black, thick] (0,0) -- (0.5,1) -- (0.75,0.7) -- (1,1.3) -- (1.5,1.2) -- (2,2.4) -- (2.5, 2.7) -- (3, 3) node[above right] {\(WB_1\)};
    \draw[black, thick] (0,0) -- (0.5,0.7) -- (0.75,0.5) -- (1,0.5) -- (1.5,1) -- (2,1.4) -- (2.5, 2) -- (3, 1.5) node[above right] {\(WB_2\)};
    \draw[black, thick] (0,0) -- (0.5,-0.3) -- (1,0.1) -- (1.5,0.5) -- (1.75, 0) -- (2,-0.2) -- (2.5, 0.8) -- (2.75, 0.3) -- (3, -0.3) node[above right] {\(WB_3\)};
    \draw[black, thick] (0,0) -- (0.5,-1) -- (1,-0.3) -- (1.5,-0.2) -- (2,-0.7) -- (2.5, -1.2) -- (3, -1) node[above right] {\(WB_4\)};
    
    \begin{scope}[on background layer]
    \draw[blue, thick, fill=blue!10] (1, 0.8) -- (2.5, 1.3) -- (2.5, 2.3) -- (1, 1.8) -- cycle;
    \draw[blue, thick] (0.8,0.8) -- (0.8,1.8) node[midway, left] {\(O(n^{-1/6})\)};
    \draw[blue, thick] (0.75,1.8) -- (0.85,1.8);
    \draw[blue, thick] (0.75,0.8) -- (0.85,0.8);

    \draw[blue, thick] (1,2) -- (2.5,2.5) node[midway, above, pos = 0.7] {\(O(n^{-1/3})\)};
    \draw[blue, thick] (1,1.95) -- (1,2.05);
    \draw[blue, thick] (2.5,2.45) -- (2.5,2.55);
    
    \end{scope}

    \end{tikzpicture}
    \caption{Brownian melon scaling limit. Above is a realisation of the $WB^4$ melon. `Zooming in' \space on the parallelogram at small scales and taking the limit as \(n\to\infty\) yields the convergence in law to the (parabolic) Airy line ensemble.}
    \label{fig: Brownian melon}
\end{figure}

\begin{figure}[t]
  \centering
  \includegraphics[width = 0.4\linewidth]{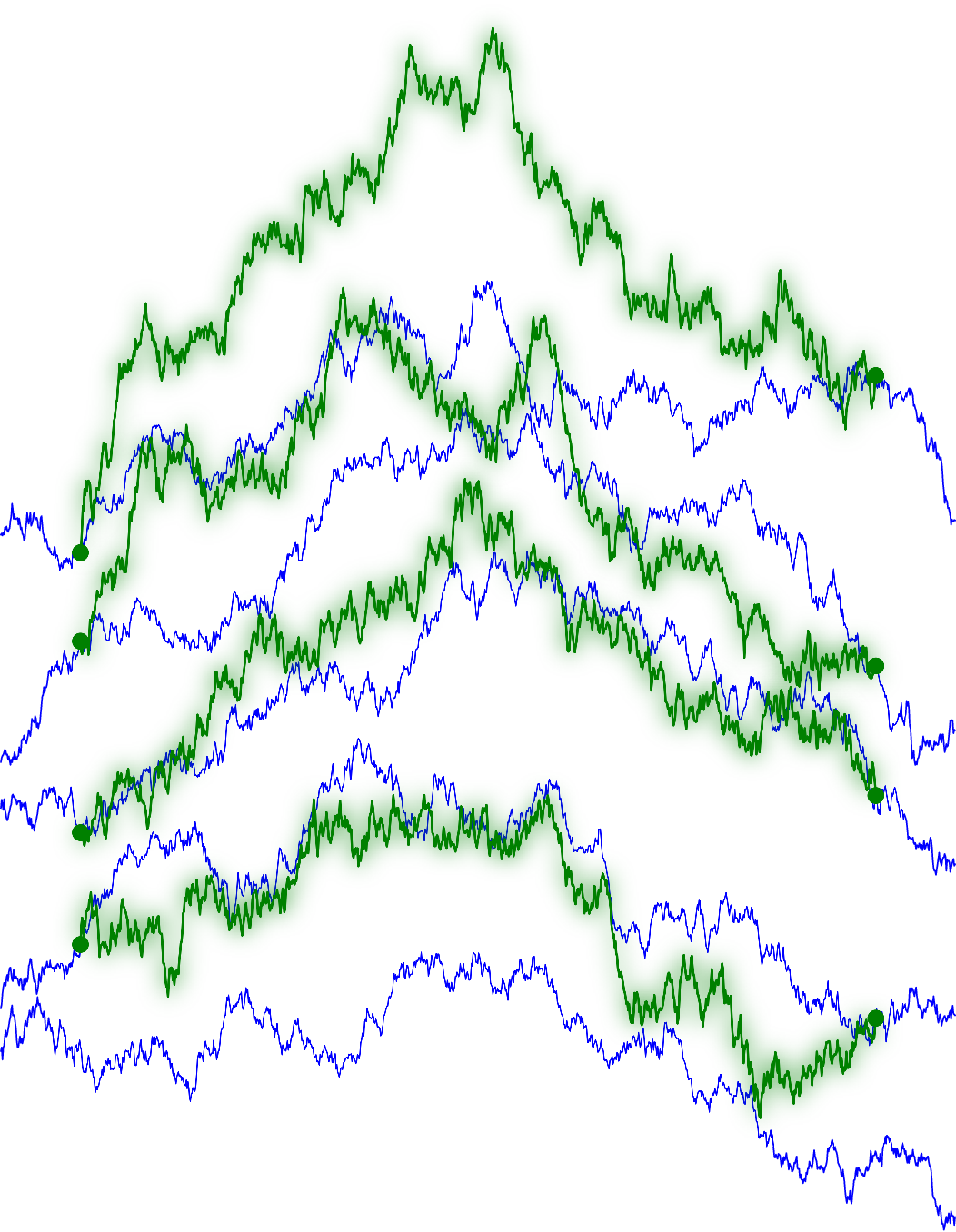}
  \caption{Figure illustrating the Brownian Gibbs property on the first four lines of the parabolic Airy Line ensemble \(\mathcal{A} = \{\mathcal{A}_1>\mathcal{A}_2>\dots\}\) (in \color{blue}blue\color{black}) between two sets of points (black dots). The \color{Green} green \color{black} curves represent resampled versions of first four lines in the ensemble between the endpoints, conditioning on $\mathscr{F}_4$.}
  \label{fig: Gibbs}
\end{figure}

where
\begin{equation*}
\mathfrak{B}_{\underline{x},\underline{y}}^{f, [a,b]} (\cdot):=\frac{\mathfrak{B}^{[a,b]}_{\underline{x},\underline{y}}(\cdot\cap \mathrm{NoInt}([a,b], f)}{\mathfrak{B}^{[a,b]}_{\underline{x},\underline{y}}(\mathrm{NoInt}([a,b], f))}\,.\end{equation*}

Notice that for fixed data $\underline{x}, \underline{y}, f$ with $x_k> f(a), y_k> f(b)$ , the measure $\mathfrak{B}_{\underline{x},\underline{y}}^{f, [a,b]}$ is absolutely continuous with respect to $\mathfrak{B}^{[a, b]}_{\underline{x},\underline{y}}$, that is the law of $k$ independent Brownian bridges on $[a,b]$ starting at $(a,x_i)$ and ending at $(b,y_i)$ respectively, for $1\leq i\leq k$. 

We now include the following global modulus of continuity result from \cite{wu2025applicationsoptimaltransportdyson} obtained using techniques from optimal transport, using the fact that the Dyson Brownian motion can be viewed as a log-concave perturbation of Brownian motion, and is inherited by a large class of random ensembles, including the stationary Airy line ensemble (see Theorem~\ref{thm: melon scaling Airy} and the remark thereafter). It essentially shows that lines in the stationary Airy line ensemble have the same modulus of continuity as that of Brownian motion.

\begin{proposition}(\cite[Corollary 1.4]{wu2025applicationsoptimaltransportdyson})\label{prop: global modulus airy}
    There exist universal constants $C_1, C_2>0$ such that for any $a<b$, $j\in\mathbb{N}$, and  $K\geq 0$, 
    \begin{align}\label{equ:Airybeta2}
    \mathbb{P}\left( \sup_{t,s\in [a,b], t\neq s}  \frac{|\mathcal{A}_j(t)-\mathcal{A}_j(s)+t^2-s^2|}{\sqrt{|t-s| \log (2(b-a)/|t-s|)}} > K \right)\leq C_1\mathrm{e}^{-C_2K^2}.
\end{align}  
\end{proposition}

These improved bounds on the modulus of continuity of the stationary line ensemble allow us to state the following refinement of \cite[Lemma 2.3]{dauvergne2024wienerdensitiesairyline} that will be needed in the later sections. It gives sub-Gaussian tails for the fluctuations of the Airy lines across indices, while also improving the dependence on the depth of the Airy line ensemble.

\begin{corollary}\label{cor: global mod airy M}
    Fix $t > 0$, then for every $m \in \N$, let
    \begin{align*}
    M = \max_{r, r' \in [-t, t]} |\mathcal{A}_{m+1}(r) - \mathcal{A}_{m+1}(r')| + \max_{i \in \llbracket 1, m\rrbracket } |\mathcal{A}_i(t) - \mathcal{A}_i(-t)|\,.
	\end{align*}
    We have that there exist some positive constants $C_1, C_2,d > 0$ independent of $t, m$ such that for all $a>0$,
     \begin{equation*}
     \PP(M > a) \le C_1 m \mathrm{e}^{dt^3} \mathrm{e}^{-C_2 a^2/t}\,.
     \end{equation*}
\end{corollary}
\begin{proof}
    To prove the bound on $M$, we apply Proposition \ref{prop: global modulus airy} to the process $\mathcal{A}_{m+1}(r), r \in [0, t]$ using the estimates for $r, r+\varepsilon\in [0,t]$ from Proposition \ref{prop: global modulus airy} (with $a = 0, b = t$ in that proposition), to obtain for all $s> 0$
    \begin{align*}
    &\displaystyle\PP\left(\max_{r,r+\varepsilon\in [0,t]}|\mathcal{A}_{m+1}(r) - \mathcal{A}_{m+1}(r + \varepsilon)|> s\sqrt{t}\right)\\
    &\le \displaystyle\PP\left(\max_{r,r+\varepsilon\in [0,t]}|\mathcal{A}_{m+1}(r) + r^2 - \mathcal{A}_{m+1}(r + \varepsilon)-(r+\varepsilon)^2| + \varepsilon^2 + 2\varepsilon r > s\sqrt{t}\right)\\
    &\le \displaystyle\PP\left(\max_{r,r+\varepsilon\in [0,t]}|\mathcal{A}_{m+1}(r) + r^2 - \mathcal{A}_{m+1}(r + \varepsilon)-(r+\varepsilon)^2| + 2\varepsilon t > s\sqrt{t}\right)\,,
    \end{align*}
    (since $\varepsilon^2 + 2\varepsilon r \le 2\varepsilon t$). Moreover, as $\varepsilon \in [0,t]$,
    \begin{equation*}
    \sqrt{\varepsilon\log 2t/\varepsilon}\le \displaystyle\sup_{x\in (0,1]}\sqrt{x\log2/x} t^{1/2} \le Ct^{1/2}
    \end{equation*}
    for some constant $C>0$;
    multiplying and dividing both sides of the term containing the Airy process by $\sqrt{\varepsilon\log 2t/\varepsilon}$ shows that the above probability is
    \begin{align*} 
    & \quad \le  \displaystyle\PP\left(\max_{r,r+\varepsilon\in [0,t]}\frac{|\mathcal{A}_{m+1}(r) + r^2 - \mathcal{A}_{m+1}(r + \varepsilon)-(r+\varepsilon)^2|}{\sqrt{\varepsilon\log(2t/\varepsilon)}} > c(s- 2t^{3/2})_+\right)\\
    & \stackrel{\text{Prop.}\ref{prop: global modulus airy}}{\le} C_1 \exp \left( - C_2 (s - 2t^{3/2})^2_+\right)\\
    & \quad \le C_1 \mathrm{e}^{dt^3}\exp(-C_2 s^2)\,,
    \end{align*}
     for some $C_1, C_2, c, d>0$ universal constants. We thus obtain by a union bound that there exist some $C_1, C_2,d > 0$ independent of $t, m$ such that for all $a>0$,
     \begin{equation*}
     \PP\left(\max_{r, r' \in [-t, t]} |\mathcal{A}_{m+1}(r) - \mathcal{A}_{m+1}(r')|> a\right) \le C_1 \mathrm{e}^{dt^3} \mathrm{e}^{-C_2 a^2/t}\,.
     \end{equation*}
     A similar argument for each of the Airy lines $\mathcal{A}_i, i \in \llbracket 1, m\rrbracket $, and union bounds, give the result.
\end{proof}

We also obtain the following proposition which is a slight variation of Proposition \ref{prop: global modulus airy}, giving sub-Gaussian concentration for the modulus of continuity of the parabolic Airy line ensemble over a fixed interval at any given depth. 
\begin{proposition}\label{prop: parabolic airy level conc}
    Fix $0<s<t$, then for every $m \in \N$, the following tail bounds hold for all $a>0$.
    \begin{align*}
    &\displaystyle\PP\left(\max_{r,r+\varepsilon\in [s,t]}|\mathcal{A}_{m}(r) - \mathcal{A}_{m}(r + \varepsilon)|> a\right)\le C_1 \mathrm{e}^{dt^2(t-s)}\exp(-C_2 a^2/(t-s))\,,
    \end{align*}
     for some $C_1, C_2, c, d>0$ universal constants.  
\end{proposition}

\begin{proof}
       Recall the definition of the stationary Airy line ensemble $\mathcal{A}^{\mathrm{stat}}$ from the Remark after Theorem~\ref{thm: melon scaling Airy}. Now, use Proposition \ref{prop: global modulus airy} (applied for $r, r+\varepsilon\in [s,t]$ ) to obtain for all $a> 0$ 
    \begin{align*}
    &\displaystyle\PP\left(\max_{r,r+\varepsilon\in [s,t]}|\mathcal{A}_{m}(r) - \mathcal{A}_{m}(r + \varepsilon)|> a\sqrt{t-s}\right)\\
    &\le \displaystyle\PP\left(\max_{r,r+\varepsilon\in [s,t]}|\mathcal{A}_{m}(r) + r^2 - \mathcal{A}_{m}(r + \varepsilon)-(r+\varepsilon)^2| + \varepsilon^2 + 2\varepsilon r > a\sqrt{t-s}\right)\\
    &\le \displaystyle\PP\left(\max_{r,r+\varepsilon\in [s,t]}|\mathcal{A}_{m}(r) + r^2 - \mathcal{A}_{m}(r + \varepsilon)-(r+\varepsilon)^2| + 2 t (t-s) > a\sqrt{t-s}\right)\,,
    \end{align*}
    (since $\varepsilon^2 + 2\varepsilon r \le 2t(t-s)$). Moreover, as $\varepsilon \in [0,t-s]$,
    \begin{equation*}
    \sqrt{\varepsilon\log 2(t-s)/\varepsilon}\le \displaystyle\sup_{x\in (0,1]}\sqrt{x\log2/x} (t-s)^{1/2} \le C(t-s)^{1/2}
    \end{equation*}
    for some constant $C>0$,
    multiplying and dividing both sides of the term containing the Airy process by $\sqrt{\varepsilon\log 2(t-s)/\varepsilon}$ shows that the above probability is
    \begin{align*} 
    & \quad \le  \displaystyle\PP\left(\max_{r,r+\varepsilon\in [0,t]}\frac{|\mathcal{A}_{m}(r) + r^2 - \mathcal{A}_{m}(r + \varepsilon)-(r+\varepsilon)^2|}{\sqrt{\varepsilon\log(2t/\varepsilon)}} > c(a - 2t(t-s)^{1/2})_+\right)\\
    & \stackrel{\text{Prop.}\ref{prop: global modulus airy}}{\le} C_1 \exp \left( - C_2 (a - 2t(t-s)^{1/2})^2_+\right)\\
    & \quad \le C_1 \mathrm{e}^{dt^2(t-s)}\exp(-C_2 a^2)\,,
    \end{align*}
     for some $C_1, C_2, c, d>0$ universal constants. 
\end{proof}

\subsection{Brownian bridge properties and lemmas}\label{subsection: some basic lemmas}

Here we put together a few standard facts and basic lemmas on Brownian bridges, that will be needed in the later sections.

We first record a key monotonicity lemma for Brownian bridges.

\begin{lemma}(Monotonic coupling)\label{lemma: bridge monotonicity}
Let $[s, t], J$ be closed intervals in $\R$ with $J \subseteq [s, t]$, let $\underline{x}^1 \le \underline{x}^2, \underline{y}^1 \le \underline{y}^2 \in \R^k_>$ where $\le$ is the coordinate-wise partial order, and let $g_1, g_2$ be two bounded Borel measurable functions from $[s,t] \to \R \cup \{-\infty\}$ such that $g_1(x)\le g_2(x)$ for all $x \in [s,t]$. For $i = 1, 2$, let $B^i$ be a $k$-tuple of Brownian bridges from $(s, \underline{x}^i)$ to $(t, \underline{y}^i)$, conditioned on the event $\mathrm{NoInt}(J, g_i)$ (recall the definition from \eqref{eq: noint}). Then there exists a coupling such that $B^1_j(r) \le B^2_j(r)$ for all $r \in [s, t], j \in \llbracket 1, k\rrbracket$.
\end{lemma}

For a sketch of a proof, see the proof of Lemmas $2.6$ and $2.7$ in \cite{corwin2014brownian}. For a more complete argument, see the proof of Lemma $2.15$ in \cite{dimitrov2021characterisation}. The key idea behind their proof is to first establish a similar result in the discrete setting of random walk bridges which is easier, and then pass to a suitable limit where the random walks converge to Brownian bridges.     

The following basic lemma computes the Radon-Nikodym derivative of a Brownian bridge with respect to a Brownian motion. 

\begin{lemma}\label{lemma: brownian bridge comparison lemma}
    Fix $0<x<y$, $m\in \N$ and let $W(\cdot)$ be a rate two Brownian bridge on $[0, y]$ with endpoints $\underline{0}, \underline{a}\in \R^m$, with law $\mathfrak{B}^{[0, y]}_{\underline{0}, \underline{a}}(\cdot)$ on $\mathscr{C}_{\underline{0}, \underline{a}}([0, y])$. Then the law $\mathfrak{B}^{[0, y]}_{\underline{0}, \underline{a}}(\cdot)$ restricted to $[0,x]$ is absolutely continuous with respect to that of a rate two Brownian motion with law $\mathfrak{B}^{[0, x]}_{\underline{0}, *}(\cdot)$ with Radon-Nikodym derivative for $\mathfrak{B}^{[0, x]}_{\underline{0}, *}$-almost all $\omega$ in $\mathscr{C}_{\underline{0},*}([0, x])$,
    \begin{equation*}
        \frac{\diff\mathfrak{B}^{[0, y]}_{\underline{0}, \underline{a}}|_{[0,x]}}{\diff\mathfrak{B}^{[0, x]}_{\underline{0}, *}}(\omega) = (y/(y-x))^{\frac{m}{2}}\cdot\exp\left(-\frac{y\norm{\omega(x)-x/y\underline{a}}^2}{4x(y-x)}\right)\cdot \exp\left(\frac{\norm{\omega(x)}^2}{4x}\right).
     \end{equation*}
        Moreover, we have that $\diff\mathfrak{B}^{[0, y]}_{\underline{0}, \underline{a}}|_{[0,x]}/\diff\mathfrak{B}^{[0, x]}_{\underline{0}, *}$ is in $L^\infty(\mathfrak{B}^{[0, x]}_{\underline{0}, *})$ with norm estimates
        \begin{equation*}
            \begin{array}{cc}
                 & \norm{\frac{\diff\mathfrak{B}^{[0, y]}_{\underline{0}, \underline{a}}|_{[0,x]}}{\diff\mathfrak{B}^{[0, x]}_{\underline{0}, *}}}_{L^p\left(\mathfrak{B}^{[0, x]}_{\underline{0}, *}\right)} = \frac{(y/(y-x))^{\frac{m}{2}}}{(px/(y-x)+1)^{\frac{m}{2}}}\cdot \exp\left({\frac{x\norm{\underline{a}}^2}{4(y-x)}\big(\frac{p}{(p-1)x+y}-\frac{1}{y}\big)}\right)
            \end{array}
        \end{equation*}
        for all $p>1$ and letting $p\to \infty$,
    \begin{equation*}
        \norm{\frac{\diff\mathfrak{B}^{[0, y]}_{\underline{0}, \underline{a}}|_{[0,x]}}{\diff\mathfrak{B}^{[0, x]}_{\underline{0}, *}}}_{L^\infty\left(\mathfrak{B}^{[0, x]}_{\underline{0}, *}\right)} \le (y/(y-x))^{\frac{m}{2}}\cdot \exp\left({\frac{\norm{\underline{a}}^2}{4y}}\right). 
    \end{equation*}
\end{lemma}
\begin{proof}
    Recalling the notation $f^{[a,b]}$  for an affine shift of a function $f$ on an interval $[a,b]$ vanishing at its endpoints, see (\ref{eq: affine shift}) in Section \ref{sec: notation}, we can couple a Brownian motion $B$  and a Brownian bridge $W$ with endpoints $\underline{0}, \underline{a}\in \R^m$ on $[0,y]$ by performing an affine shift and setting
    $$
    W(\cdot) = B^{[0, y]}(\cdot)+\frac{(\cdot)}{y}\underline{a}\,
    $$
    which we can re-express as 
    $$
    W(\cdot) = B^{[0, x]}(\cdot)+\frac{(\cdot)}{x}N
    $$
    on $[0, x]$ for some $m-$dimensional Gaussian vector $N$ with independent entries having mean $x\underline{a}/y$ and variance $2(y-x)x/y$, that is independent of the affine shift $B^{[0, x]}(\cdot)$ on $[0,x]$ (this can be seen by simply checking that the covariances vanish). Observe that if one were to replace $N$ with $B_x$, one would recover the original Brownian motion; now, a straight-forward computation shows that
    \begin{equation*}
        \frac{\diff\mathfrak{B}^{[0, y]}_{\underline{0}, \underline{a}}|_{[0,x]}}{\diff\mathfrak{B}^{[0, x]}_{\underline{0}, *}} = \frac{\diff N}{\diff B_x}\,
    \end{equation*}
    whence we derive the desired almost sure equality for the Radon-Nikodym derivative and conclude the proof of the first part. Now we fix any $p>1$ and compute
    \begin{align*}
            \norm{\frac{\diff\mathfrak{B}^{[0, y]}_{\underline{0}, \underline{a}}|_{[0,x]}}{\diff\mathfrak{B}^{[0, x]}_{\underline{0}, *}}}^p_{L^p\left(\mathfrak{B}^{[0, x]}_{\underline{0}, *}\right)} & = \frac{(y/(y-x))^{\frac{pm}{2}}}{(4\pi x)^{\frac{m}{2}}}\displaystyle\int_{\R^m}\exp\left(-\frac{py\norm{\underline{z}-x/y\underline{a}}^2}{4x(y-x)}\right)\\
            & \cdot \exp\left(\frac{(p-1)\norm{\underline{z}}^2}{4x}\right)\diff \underline{z} \\
            &= \frac{(y/(y-x))^{\frac{pm}{2}}}{(4\pi x)^{\frac{m}{2}}}\displaystyle\int_{\R^m}\exp\left(-\left(\frac{py}{4x(y-x)}-\frac{(p-1)}{4x}\right)\norm{\underline{z}}^2\right) \\
            &\cdot  \exp\left(\frac{p}{2(y-x)}\underline{z}\cdot \underline{a}-\frac{px}{4y(y-x)}\norm{a}^2\right)\diff\underline{z}\\
            &= \frac{(y/(y-x))^{\frac{pm}{2}}}{(4\pi x)^{\frac{m}{2}}}\displaystyle\int_{\R^m}\exp\left(-\left(\frac{p}{4(y-x)}+\frac{1}{4x}\right)\norm{\underline{z}}^2\right) \\
            &\cdot  \exp\left(\frac{p}{2(y-x)}\underline{z}\cdot \underline{a}-\frac{px}{4y(y-x)}\norm{a}^2\right)\diff\underline{z}\\
             &= \frac{(y/(y-x))^{\frac{pm}{2}}}{(px/(y-x)+1)^{\frac{m}{2}}}\cdot \exp\left({\frac{px\norm{\underline{a}}^2}{4(y-x)}\big(\frac{p}{(p-1)x+y}-\frac{1}{y}\big)}\right)\,.
    \end{align*}
    We now have a uniform bound which allows us to pass to $p\to \infty$ and conclude.
\end{proof}

We now slightly generalise the above, comparing the Brownian bridge in an interval in the interior of its domain to a Brownian motion. 
\begin{lemma}\label{lemma: brownian bridge comparison lemma general endpoints}
    Fix $x<y<z<w$, $m\in \N$ and let $W(\cdot)$ be a rate two Brownian bridge on $[x, w]$ with endpoints $\underline{a}, \underline{b}\in \R^m$ with law $\mathfrak{B}^{[x, w]}_{\underline{a}, \underline{b}}(\cdot)$ on $\mathscr{C}_{\underline{a}, \underline{b}}([x, w])$. Then the law of $W(\cdot)-W(y)$ restricted to $[y,z]$ is absolutely continuous with respect to that of a rate two Brownian motion with law $\mathfrak{B}^{[y, z]}_{\underline{0}, *}(\cdot)$ with Radon-Nikodym derivative for $\mathfrak{B}^{[y, z]}_{\underline{0}, *}$-almost all $\omega$ in $\mathscr{C}_{\underline{0},*}([y, z])$,
    \begin{equation*}
    \begin{array}{cc} 
        \frac{\diff\mathfrak{B}^{[x, w]}_{\underline{a}, \underline{b}}|_{[y,z]}}{\diff\mathfrak{B}^{[0, z-y]}_{\underline{0}, *}}(\omega) &= ((w-x)/(w-x-z+y))^{\frac{m}{2}}\cdot\exp\left(-\frac{(w-x)\norm{\omega(z-y)-(z-y)/(w-x)\underline{a}}^2}{4(z-y)(w-x-z+y)}\right)\\
        & \cdot \exp\left(\frac{\norm{\omega(z-y)}^2}{4(z-y)}\right).
        \end{array}
     \end{equation*}
        Moreover, we have that $\diff\mathfrak{B}^{[x, w]}_{\underline{a}, \underline{b}}|_{[y,z]}/\diff\mathfrak{B}^{[y, z]}_{\underline{0}, *}$ is in $L^\infty(\mathfrak{B}^{[y, z]}_{\underline{0}, *})$ with norm estimates
        \begin{equation*}
            \begin{array}{cc}
                 \norm{\frac{\diff\mathfrak{B}^{[x, w]}_{\underline{a}, \underline{b}}|_{[y,z]}}{\diff\mathfrak{B}^{[y, z]}_{\underline{0}, *}}}_{L^p\left(\mathfrak{B}^{[y, z]}_{\underline{0}, *}\right)} = \frac{((w-x)/(w-x-z+y))^{\frac{m}{2}}}{(px/(w-x-z+y)+1)^{\frac{m}{2}}}
                  \cdot \exp\left({\frac{(z-y)\norm{\underline{a}}^2}{4(w-x-z+y)}\big(\frac{p}{(p-1)(z-y)+w-x}-\frac{1}{w-x}\big)}\right)
            \end{array}
        \end{equation*}
        for all $p>1$ and letting $p\to \infty$,
    \begin{equation*}
        \norm{\frac{\diff\mathfrak{B}^{[x, w]}_{\underline{a}, \underline{b}}|_{[y,z]}}{\diff\mathfrak{B}^{[y, z]}_{\underline{0}, *}}}_{L^\infty\left(\mathfrak{B}^{[y, z]}_{\underline{0}, *}\right)} \le ((w-x)/(w-x-z+y))^{\frac{m}{2}}\cdot \exp\left({\frac{\norm{\underline{a}-\underline{b}}^2}{4(w-x)}}\right). 
    \end{equation*}
\end{lemma}
\begin{proof}
    By translation, it suffices to prove the lemma for $x=0$. Observe that we can realise a Brownian bridge $W$ with endpoints $\underline{a}, \underline{b}\in \R^m$ on $[0,w]$ using a Brownian motion $B$ by performing an affine shift and setting
    $$
    W(\cdot) = B^{[0, w]}(\cdot)+\frac{(\cdot)}{w}\underline{b} + \frac{w-\cdot}{w}\underline{a} .
    $$
    Thus, we observe that on $[0,z-y]$, $W(\cdot+y)-W(y)$ has the law of $m$ independent Brownian bridges starting from $\underline{0}$ and  
    $$
    W(z)-W(y) = B(z)-B(y)-\frac{z-y}{w}B_y + \frac{z-y}{w}\underline{a}-\frac{z-y}{w}\underline{b},
    $$
    which has the distribution of a Gaussian vector having independent entries with mean $\frac{z-y}{w}\underline{a}-\frac{z-y}{w}\underline{b}$ and variance $2(z-y)(1-z/w+y/w)$. 
    Hence, we obtain the decomposition
    \begin{equation*}
    W(\cdot+z)-W(y) = (W(\cdot+z)-W(y))^{[0,z-y]} +  \frac{(\cdot)}{z-y}(W(z)-W(y))
    \end{equation*}
    on $[0, z-y]$, where the terms on the right hand side are independent (zero mean and one can check the covariance vanishes). Observe that if one were to replace $(W(z)-W(y))$ with an independent $m$-dimensional Gaussian vector $N$ with coordinate-wise independent entries and having mean zero and variance $2(z-y)$, one would recover the original Brownian motion; now, a straight-forward computation shows that
    \begin{equation}
        \frac{\diff\mathfrak{B}^{[0, w]}_{\underline{a}, \underline{b}}|_{[y,z]}}{\diff\mathfrak{B}^{[0, z-y]}_{\underline{0}, *}} = \frac{\diff\, (W(z)-W(y))}{\diff N}\,,
    \end{equation}
    whence we derive the desired almost sure equality for the Radon-Nikodym derivative and conclude the proof of the first part. For the remaining parts, one proceeds as in the previous lemma.
\end{proof}

We finally end with an estimate on the distribution of the maximum absolute value of a rate two Brownian bridge vanishing at its endpoints, also known as the \emph{Kolmogorov-Smirnov distribution}. Incidentally, this distribution is related to the classical non-parametric
test statistic (bearing the same name) commonly used to compare an empirical distribution function with a reference
distribution, \cite{berger2014kolmogorov}.

\begin{lemma}\label{lemma: brownian bridge maximum}
    Let $T>0$ and consider a rate two Brownian bridge $(W_t)_{t\in [0,T]}$ vanishing at both endpoints, then there is a universal constant $c>0$ such that for all $a>0$, 
    \begin{equation*}
    \PP\left(\displaystyle\max_{0\le t\le T}|W_t|\le a\right) \ge \frac{c\sqrt{T}}{a} \exp \left( - \frac{\pi^2 T}{4a^2}\right)\,.
    \end{equation*}
\end{lemma}
\begin{proof}
Observe that for a rate two Brownian motion $(B_t)_{t\ge 0}$, one has that $(B_t)_{t\ge 0} \stackrel{d}{=} (\sqrt{2}B'_{t})_{t\ge 0}$ where $(B'_{t})_{t\ge 0}$ is a standard Brownian motion. By Brownian scaling and the above, we thus obtain the distributional identities 
    \begin{equation*}
    (W_t)_{t\in [0,T]} \stackrel{d}{=} (B_t-tB_1)_{t\in [0,T]} \stackrel{d}{=} (\sqrt{2}B'_{t}-\sqrt{2}tB'_{1})_{t\in [0,T]}\stackrel{d}{=} (\sqrt{2}W'_t)_{t\in [0,T]}\,,
    \end{equation*}
    where $ (W'_t)_{t\in [0,T]}$ is a standard Brownian bridge vanishing at both endpoint. Hence, by another application of Brownian scaling, we have the distributional equality
    \begin{equation*}
    \displaystyle\max_{0\le t\le T}|W_t| \stackrel{d}{=} \sqrt{2T}\displaystyle\max_{0\le t\le 1}|\tilde{W}_t| 
    \end{equation*}
 where $\tilde{W}$ is a standard (rate one) Brownian bridge vanishing at $0$ and $1$. So it suffices to prove the lower bound for a rate one Brownian bridge and $T=1$. Recall the tail probabilities for the \emph{Kolmogorov-Smirnov distribution}, (see \cite{berger2014kolmogorov})
 \[
 \PP\left(\displaystyle\max_{0\le t\le 1}|\tilde{W}_t|\le a \right) = \frac{\sqrt{2\pi}}{a}\displaystyle\sum_{k=1}^\infty \mathrm{e}^{-(2k-1)^2\pi^2/(8a^2)}\ge \frac{\sqrt{2\pi}}{a}\exp \left( -\frac{\pi^2}{8a^2}\right)\,,\quad a > 0\,.
 \]
Finally, we have the estimate
\begin{equation*}
\PP\left(\displaystyle\max_{0\le t\le T}|W_t|\le a\right)  = \PP\left(\displaystyle\max_{0\le t\le 1}|\tilde{W}_t|\le \frac{a}{\sqrt{2T}}\right) \ge \frac{c\sqrt{T}}{a} \exp \left( - \frac{\pi^2 T}{4a^2}\right)\,,\quad \mbox{ for some constant }c>0\,.
\end{equation*}
\end{proof}

\subsection{Airy line ensemble}\label{sec: Airy two point bounds}

Using the refined modulus of continuity estimates for the Airy line ensemble in Proposition \ref{prop: global modulus airy}, one can obtain control over the fluctuations of the Airy last passage values about the typical Brownian counterpart after some normalisation. This follows from sub-additivity properties of last passage percolation and the Brownian bridge representation for the Airy line ensemble in \cite{dauvergne2021}.
In the following theorem, to ease notation, we will write for $a<b$ and $k\in \N$, the last passage percolation values of the Airy line ensemble to the first line by 
\begin{equation}\label{eq: airy last passage values}
\langle (a,k)\to b \rangle \stackrel{\mathrm{def}}{=}\mathcal{A}[(a,k)\to (b,1)]\,.
\end{equation}

Now, there are two regimes regarding the fluctuations of the value of the Airy line ensemble LPP around its Brownian counterpart's mean on compact intervals,
\begin{equation*}
\frac{|\langle (0, k) \rightarrow x \rangle-{2\sqrt{2kx}}|}{k^{1/2}} > \varepsilon\,\qquad \mbox{ for }\varepsilon >0, k\ge 1, x>0\,,
\end{equation*}
 in which we will be interested: namely when $\varepsilon < k^{1/126}$ and when $\varepsilon > O_x(1)\lor k^{1/126}$. We will be exploiting the bridge representation to study the former and concentration of measure plus sub-Gaussian tails of the moduli of continuity of lines in the Airy line ensemble for the latter. Also note that the parameters in the tail exponents were not optimised and so it may most likely be possible to improve them. This is the content of the following Theorem~.

\begin{theorem}[Theorem 6.7 \cite{DOV}]\label{thm: Airy LPP deviation}
Fix $x>0$, and recall that $\langle (0, k) \rightarrow x \rangle$ is the last passage value across the Airy line ensemble $\mathcal{A}$ from line $k$ at time $0$ to line $1$ at time $x$. Then for all $\varepsilon < k ^{1/126}$,
$$
\PP \left(\frac{|\langle (0, k) \rightarrow x \rangle-{2\sqrt{2kx}}|}{k^{1/2}} > \varepsilon\right)$$
$$
\le c k^2 \mathrm{e}^{dx^3}\left(\exp(-d\varepsilon^{1/2}  k^{1/28}) + \exp\left(-d\varepsilon k^{1/126}/(x\land x^{3/4})\right)\right)\,,
$$
for some universal constants $c,d>0$. Alternatively, in the regime where $\varepsilon > 4\sqrt{2x}$, then
$$
\PP \left(\frac{|\langle (0, k) \rightarrow x \rangle-{2\sqrt{2kx}}|}{k^{1/2}} > \varepsilon\right) \le C \mathrm{e}^{dx^3} k\exp \left(-\frac{d\varepsilon^2}{kx}\right) \qquad k\ge 1\,,
$$
for some universal constants $C, d>0$.
\end{theorem}

\begin{proof}[Proof of Theorem~\ref{thm: Airy LPP deviation}]
We will essentially adapt the arguments from the proof of \cite[Theorem~6.7]{DOV} making use of the improved modulus of continuity estimates for the Airy line ensemble from \cite{wu2025applicationsoptimaltransportdyson}, simplifying parts of the proof, paying close attention to tail probabilities.

First, consider the regime where $\varepsilon > 4\sqrt{2x}$, using Proposition \ref{prop: parabolic airy level conc} and a union bound,
\begin{align*}
& \PP \left(\frac{|\langle (0, k) \rightarrow x \rangle-{2\sqrt{2kx}}|}{k^{1/2}} > \varepsilon\right) \le \PP \left(|\langle (0, k) \rightarrow x \rangle| > \frac{\varepsilon}{2}k^{1/2}\right)\\
& \le \displaystyle\sum_{i=1}^k\PP\left( \displaystyle\sup_{r,r+\theta \in [0,x]}|\mathcal{A}_{i}(r) - \mathcal{A}_{i}(r + \theta)|> \frac{\varepsilon}{2k^{1/2}}\right)\le C_1\mathrm{e}^{dx^3}k\exp \left( -\frac{d\varepsilon^2}{kx}\right) \,,\qquad k\ge 1\,,
\end{align*}
for some positive constant $d> 0$.

In the remainder, set $x=1$ for notational simplicity as the value of $x$ plays no important role. Further note that in all estimates the dependence of coefficients of tail bounds on $x$ is continuous, so one can harmlessly take suprema of such bounds for $x$ in compacts at the expense of some weaker constants, keeping the functional form of the tails the same.

Let
$
\mathfrak{B}^k =  \mathfrak{B}^k(x, \ceil{k^{2/3 + \gamma}}, k^{-1/3 - \gamma/4})
$
be the bridge representation induced by the division of time $\{s_r = rx/\ceil{k^{2/3 + \gamma}}: r \in \{1, \dots, \ceil{k^{2/3 + \gamma}}\}\}$ and the graph
$$
G_{2k} = G_{2k}(x, \ceil{k^{2/3 + \gamma}}, k^{-1/3 - \gamma/4}).
$$
Here $\gamma \in (0, 1/3)$ is a parameter that we will optimize over later in the proof. By \cite{dauvergne2021}{Theorem~7.2}, we can couple all the representations $\mathfrak{B}^k$ with the Airy line ensemble $\mathcal{A}$ so that for some universal constant $d$ and all $k\ge 1$
\begin{equation}
\label{eq: scrBkcouple}
\PP\left(\mathfrak{B}^k|_{\{1, \dots, k\} \times [0, x]} \ne \mathcal{A}|_{\{1, \dots, k\} \times [0, x]} \right) \le x\lceil k^{2/3+\gamma}\rceil \mathrm{e}^{-d\gamma k^{\gamma/12}}\,.
\end{equation}
Hence it suffices to analyse the last passage time $L(\mathfrak{B}^k)$ from $(0, k)$ to $(x, 1)$.

\medskip

\textbf{Step 1: Path decomposition.} By representing each of the Brownian bridges used to create $\mathfrak{B}^k = (\mathfrak{B}_{k, 1}, \dots, \mathfrak{B}_{k, k})$ as a Brownian motion minus a random linear term, we can write
$$
\mathfrak{B}_{k, i} = H_{k, i} + R_{k, i} + X_{k, i}
$$
Here the $k$-tuple $H_k = (H_{k, 1}, \dots, H_{k, k})$ consists of $k$ independent rate two Brownian motions on $[0, x]$. The functions $R_{k, i}$ are piecewise linear with pieces defined on the time intervals $[s_{r-1}, s_{r}]$ for $r \in \{0, \dots, \ceil{k^{2/3 + \gamma}}\}$, and the error term $X_{k, i}$ is equal to zero except for on intervals $[s_{r-1}, s_{r}]$ where the vertex $(i, r)$ is in a component of size greater than one in the graph $G_{2k}$. On such intervals, $X_{k, i}$ is the difference between a Brownian bridge from $0$ to $0$ and a Brownian bridge conditioned to avoid $U_{i, r} -1$ other Brownian bridges with certain start and endpoints. Here $U_{i, r}$ is the size of the component of $(i, r)$ in $G_{2k}$ and the two Brownian bridges used in the definition of $X_{k, i}$ are independent.

\medskip

By \cite[Lemma 6.9]{DOV} applied twice, we have that the last passage values, here denoted by $L(\cdot)$,
\begin{equation}
\label{eq: subadd}
L(H_k)+F(R_k)+F(X_k) \le L(\mathfrak{B}^k)\le L(H_k)+L(R_k)+L(X_k)\,,
\end{equation}
where for an environment $f$, $F(f) = -L(-f)$.
By Theorem~2.5 in \cite{dauvergne2021}, the main term
\begin{equation}
\label{eq: Wk-2}
L(H_k)=2\sqrt{2kx}+Y_kk^{-1/6},
\end{equation}
where $\{Y_k\}_{k \in \N}$ is a sequence of random variables satisfying a tail bound
$$
\PP(|Y_k| > m) \le c\mathrm{e}^{-dm^{3/2}/x^{3/4}}
$$
for $c, d >0$  universal constants. To translate Theorem~2.5 in \cite{DOV} to a bound on last passage values, we have used the preservation of last passage values under the melon operation.

\medskip

\noindent \textbf{Step 2: Bounds on the piecewise linear term.}  Observe we have the bound
$$
|L(R_k)|, |F(R_k)| \le M_k,
$$
where $M_k$ is the maximum absolute slope of any of the piecewise linear segments in $R_k$. The slopes in $R_k$ come from increments in the Airy line ensemble minus the increments of the Brownian motions $H_k$ on the grid points. With the notation $S_k(\ell) = \{1, \dots, k\} \times \{1, \dots, \ell\}$, $\ell_k = \lceil k^{2/3+\gamma}\rceil$ and $s_i = ix/\ell_k\,, i\in \{0,\cdots, \ell_k\}$, we have the following upper bound for $M_k$:
\begin{align*}
&\ceil{k^{2/3 + \gamma}}\left[\max_{(i, r) \in S_k(\ceil{k^{2/3 + \gamma}})} |H_{k, i}(s_{r}) - H_{k, i}(s_{r-1})|+ \max_{(i, r) \in S_k(\ceil{k^{2/3 + \gamma}})} |\mathcal{A}_i(s_{r}) - \mathcal{A}_i(s_{r-1})|\right].
\end{align*}
By a standard Gaussian bound on the first term and Proposition \ref{prop: parabolic airy level conc} for the second term, for some $d \in \N$ we have that for all $\delta \in (0, 1/2-1/3-\gamma/2)$
\begin{equation}
\label{eq: Mk}
\PP \left(M_k \ge \varepsilon k^{1/3 + \gamma/2 + \delta} \right) \le c \mathrm{e}^{dx^3}k\lceil k^{2/3+\gamma}\rceil\exp \left(-d \varepsilon^2 k^{2/3+\gamma + 2\delta}/x\right)\,, \qquad k\ge 1\,,
\end{equation}
for some possibly $\delta$-dependent $c, d>0$.
\medskip

\noindent \textbf{Step 3: What's left.}
To bound $L(X_k)$ and $F(X_k)$, we divide $\{1,\dots,k\}$ into $n= \ceil{k^{2/3+\gamma}}$ intervals
$$
I_{k, i} = \left\{ \bigg\lfloor\frac{(i-1)k}n\bigg\rfloor + 1, \dots, \bigg\lfloor\frac{ik}n\bigg\rfloor\right\}, \qquad i \in \{1, \dots, n\}\,.
$$
This, and the division of time into the intervals $[s_{r-1}, s_r]$ for $r \in \{1, \dots, n\}$ breaks the line ensemble $X_k$ into $n^2$ boxes. Each last passage path can meet at most $2n-1$ of these boxes. So we have that
\begin{equation}
\label{eq: Xk}
L(X_k) \le (2n-1)Z_k,
\end{equation}
where $Z_k$ is the maximal last passage value among all values that start and end in the same box (including the boundary). Specifically,
$$
Z_k = \max_{(i, r) \in [1, n]^2} \max \left\{ X_k[(\ell_1, t_1) \rightarrow (\ell_2, t_2)] : \ell_1, \ell_2 \in I_{k, i}, \; t_1, t_2 \in [s_{r-1}, s_r]\right\}.
$$
We have that $Z_k \le N_kD_k$, where
\begin{align*}
N_k &= \max_{(i, r) \in [1, n]^2} \mathrm{card}{\left\{ \ell \in I_{k, i} : X_{k,\ell}|_{[s_{r-1}, s_{r}]} \ne 0\right\}} \quad \mathrm{and} \\
 D_k &= \max \bigg\{ |X_{k, \ell}(t) - X_{k, \ell}(m)| : \ell \in [1, k], t, m \in [s_{r-1}, s_r] \text{ for some } r \in \{1, \dots, n\} \bigg\}
\end{align*}
and $\mathrm{card}$ denotes the cardinality of a (finite) set, i.e. the number of elements it contains.

That is, $N_k$ is the maximum number of non-zero line segments in any box, and $D_k$ is the maximum increment over any line segment in a box. Since $X_{k, \ell} = \mathfrak{B}_{k, \ell} - H_{k, \ell} - R_{k, \ell}$, we can bound $D_k$ in terms of the deviations of the other paths. To bound the deviation of $R_{k, \ell}$, we use the bound on $M_k$ above. The deviation of $H_{k, \ell}$ can be bounded with standard bounds on Gaussian random variables. On the event where $\mathfrak{B}^k|_{\{1, \dots, k\} \times [0, x]} = \mathcal{A}|_{\{1, \dots, k\} \times [0, x]}$, we can bound the deviation of $\mathfrak{B}_{k, \ell}$ using Proposition \ref{prop: global modulus airy}. Thus, we have for all $\delta >0$, $k\ge 1$
\begin{equation}
\label{eq: Dka}
\PP\left(D_k > \varepsilon k^{-1/3 - \gamma/2+\delta }, \; \mathfrak{B}^k = \mathcal{A}|_{\{1, \dots, k\} \times [0, 1]} \right) \le c k\ceil{k^{2/3+\gamma}}\exp \left(-d\varepsilon^2  k^{2/3+\gamma + 2\delta}/x\right)\,,
\end{equation}
for some $d>0$. Combining equations \eqref{eq: Dka} and \eqref{eq: scrBkcouple} gives for all $\delta >0$
\begin{equation}
\label{eq: Dk}
\PP\left(D_k > \varepsilon k^{-1/3 - \gamma/2 +\delta}\right) \le c k\ceil{k^{2/3+\gamma}} \left(\exp \left(-d \varepsilon^2 k^{2/3+\gamma + 2\delta}/x\right) + \mathrm{e}^{-d\gamma k^{\gamma/12}}\right)\,, k\ge 1\,.
\end{equation}
The quantity $N_k$ is equal to the maximum number of edges in the graph $G_k$ in a region of the form $I_{k, i} \times \{r\}$ for some $r \in \{1, \dots, n\}$. This can be bounded using \cite[Proposition 6.7]{DOV} and a union bound, which yields for all $\delta > 0, \varepsilon < k ^{1/6-\gamma/2 - \delta}$ and $k\ge 1$
$$
\PP \left(N_k > \varepsilon k^{1/3 - \gamma} k^{- 3 \gamma/4} k^{\delta} \right) \le c x\ceil{k^{2/3+\gamma}} \exp(-d \varepsilon k^\delta)\,,
$$
for some constant $d>0$. Combining this with the bound in \eqref{eq: Xk} and \eqref{eq: Dk} implies that for all $\delta > 0, \varepsilon < k ^{1/6-\gamma/2 - \delta}$
\begin{align}
\label{eq: Mk-2}
&\PP\left(L(X_k) > \varepsilon k^{2/3 - 5\gamma/4} k^{\delta} \right) \le c k \ceil{k^{2/3+\gamma}} \\
&\cdot \left(\exp \left(-d \varepsilon k^{2/3+\gamma + \delta/2}/x\right) + \exp(-d\varepsilon^{1/2}  k^{\delta/2}) + \mathrm{e}^{-d\gamma k^{\gamma/12}}\right)\,.\nonumber
\end{align}

We can symmetrically bound $F(X_k)$.

\medskip

\noindent \textbf{Step 4: Combining everything together.}
By combining the inequalities \eqref{eq: subadd}, \eqref{eq: Wk-2}, \eqref{eq: Mk} and \eqref{eq: Mk-2}, we get that for all $\delta >0 $ and $\varepsilon < k ^{1/6-\gamma/2 - \delta}$
\begin{align*}
&\PP\left( |L(\mathfrak{B}^k) - 2\sqrt{2kx}| >  \varepsilon k^{2/3 - 5\gamma/4 + \delta} + \varepsilon k^{1/3 + \gamma /2 + \delta} \right)\\
& \le c k\ceil{k^{2/3+\gamma}} \mathrm{e}^{dx^3}\left(\exp \left(-d \varepsilon k^{2/3+\gamma + \delta/2}/(x\land x^{3/4})\right)+\exp(-d\varepsilon^{1/2}  k^{\delta/2}) + \mathrm{e}^{-d\gamma k^{\gamma/12}}\right)
\end{align*}
for positive constants $c,d$. Taking $\gamma = 4/21, \delta = 1/14-1/126$ completes the proof of the first regime of `small' $\varepsilon$. 
\end{proof}

\section{Transversal fluctuations of semi-infinite geodesics in the  Airy line ensemble}\label{sec: geod geom}

In this section, we investigate geodesic geometry in the Airy line ensemble. The main result is Theorem~\ref{thm: intercept tail bound} where we obtain \emph{exponentially stretched} tail bounds on intercepts of semi-infinite geodesics. We start with providing the concentration result for semi-infinite geodesic intercepts in the Airy line ensemble.

\subsection{Tail bounds on geodesic intercepts} Recall Theorem~\ref{thm: melon scaling Airy} which gives the Airy line ensemble as a scaling limit of rescaled Brownian melons. With this in mind, we will start in the prelimiting environment and obtain some more refined structural properties of the prelimiting jump times of geodesics on such melons. By the weak convergence already established, they easily translate to the limiting objects.
    
    Now, using the notation established in \cite{DOV}, for $n\in \N$, let
\begin{equation*}\underline{x}=2xn^{-1/3}\,,\qquad \mbox{ and } \qquad \hat{y}=1+2yn^{-1/3}\,.\end{equation*}
Furthermore, let $\gamma_n = \pi\{\underline{x}\to \hat{y}\}_n$ be the rightmost last passage path between $\underline{x}$ and $\hat{y}$ in the melon $W B^n$ with last passage value (recall Definition \eqref{Definition: last passage})
\[
\{\underline{x}\to \hat{y}\}_n \equiv WB^n[(\underline{x}, n)\to (\hat{y}, 1)]\]
(we will henceforth drop the subscript from the curly brackets as it is clear from the context).
For $n\in \N$ and $k\in \{1,2,\ldots, n\}$, let $\hat{Z}^n_k(x,y) = 1+n^{-\frac{1}{3}}Z_k^n(x,y)$ be the supremum of $w$ so that $(w,k)$ lies along $\gamma_n$. Then, by \cite{DOV}, Lemma \ref{lemma: jump times parabola conc}, it follows that for each $k\in \N$, the sequence $(Z^n_k(x,y))_n$ is tight. Let $Z_k(x,y)$ denote the sub-sequential limits of $(Z^n_k(x,y))_n$ for any $x,y$.

\begin{lemma}\label{lemma: jump times parabola conc}
Let $K$ be a compact countable subset of $(0, \infty) \times \R$. Then for any $\varepsilon > 0$
\begin{align*}
\label{eq: xyKxyK}
& \PP\left(
    \sup_{(x,y)\in K} \left| Z_k(x, y) + \sqrt{\frac{k}{2x}} \right| \ge \varepsilon \sqrt{k}
\right) \\
& \le C (\varepsilon^2 \lor 1/\varepsilon) \Biggl(
    \sup_{x\in K} \PP \left(
        \frac{\left|\mathcal{A}[ (0, k) \rightarrow (x,1) ] - 2\sqrt{2kx} \right|}{k^{1/2}} > \varepsilon
    \right)\\
& + \exp\left(
        - d \frac{
            \varepsilon^{3/4} k^{3/4}
        }{
            \bigl(\sup_{(x,y)\in K} (|x| + |y| + 1)\bigr)^2
        }
    \right)
\Biggr) \,, \quad k \ge 1
\end{align*}
for some universal $d>0$ and $C>0$.
\end{lemma}

\begin{proof}
First, fix $x,y\in K$. Now, rescale by $n^{1/6}$ and centre so that the triangle inequality
\begin{equation*}
\{\bar x \rightarrow (\hat z,k)\} + \{(\hat z,k)\rightarrow \hat y\} \le   \{\bar x\rightarrow \hat y \}
\end{equation*}
reads 
\begin{equation}\label{eq: triangle2}
F^n_k(x, z)+G^n_k(z, y)\le H_n(x, y) 
\end{equation}
with
\begin{eqnarray*}
H_n(x, y)&=&n^{1/6}\{\bar x\rightarrow \hat y \}-2n^{2/3}-2(y-x)n^{1/3}, \\
F^n_k(x, z) &=& n^{1/6}(\{\bar x \rightarrow (\hat z,k)\}-W^n_k(\hat z))+2xn^{1/3},
\\ G^n_k(z, y)&=&n^{1/6}(W^n_k(\hat z) + \{(\hat z,k )\rightarrow \hat y\})-2yn^{1/3}-2n^{2/3}.
\end{eqnarray*}
The basic proof strategy for bounding $Z^n_k(x, y)$ is as follows. On the one hand,
\begin{equation*}
F^n_k(x, Z^n_k(x, y)) + G^n_k(Z^n_k(x, y),y)=H_n(x, y)
\end{equation*}

We will show that for every $\varepsilon \in (0, 1)$ we have
\begin{equation}
\label{eq: zzz}
\sup_{z : \;|z+\sqrt{k/(2x)}|>\varepsilon \sqrt{k}} F^n_k(x, z)+G^n_k(z, y) \le -\varepsilon^2\sqrt{kx}/2 + \mathfrak{o}(\sqrt{k})\,,
\end{equation}
where the error term $\mathfrak{o}(\sqrt{k})$ is asymptotically `small' with respect to the scale of $\sqrt{k}$. By \cite[Lemma 3.3]{DOV},
$F^n_k(x, \cdot)$ is monotonically increasing and $G^n_k(\cdot, y)$ is monotonically decreasing. We can use this monotonicity to bound the left hand side of \eqref{eq: zzz} by a supremum over a finite set. 
 Let
$
A = (12\varepsilon^2)^{-1}\Z \cap [1/4, 2],
$
and for $z \in [-n^{1/3} + x, y]$, define 
\begin{equation*}
\label{eq: floorznk}
\begin{split}
    \floor{z}_{n, k} &= \max \{w \in -\sqrt{k/x}A \cup \{x - n^{1/3}\} : w < z\} \qquad \mathrm{and}\\
\ceil{z}_{n, k} &= \min \{w \in -\sqrt{k/x}A \cup \{y\} : w > z\}.
\end{split}
\end{equation*}
We also set $\lfloor x - n^{1/3}\rfloor_{n, k} = x - n^{1/3}$ and $\ceil{y}_{n, k} = y$. The monotonicity of $F^n_k(x, \cdot)$ and  $G^n_k(\cdot, y)$ implies that the left hand side of \eqref{eq: zzz} is bounded above by
\begin{equation}
\label{eq: with-part}
    \sup_{z : \;|z+\sqrt{k/(2x)}|>\varepsilon \sqrt{k}} F^n_k(x, \ceil{z}_{n, k})+G^n_k(\floor{z}_{n, k}, y)\,.
\end{equation}
Notice that the number of terms is uniformly bounded in $n$ and $k$, so it is enough to control the terms individually. There are three cases to consider, namely, 

\begin{align}\label{eq: with-part-2}
    \begin{cases}
        &F^n_k(x, z_{k,a})+G^n_k(x-n^{1/3}, y) \le -\varepsilon^2\sqrt{kx}/2 + \mathfrak{o}(\sqrt{k})\\
        &F^n_k(x, z_{k, a})+G^n_k(z_{k, a}, y) \le -\varepsilon^2\sqrt{kx}/2 + \mathfrak{o}(\sqrt{k})\\
        &F^n_k(x, y)+G^n_k(z_{k, a}, y) \le -\varepsilon^2\sqrt{kx}/2 + \mathfrak{o}(\sqrt{k})\,,
    \end{cases}
\end{align}
for every fixed $a\in A$, with $z_{k, a} =  - a\sqrt{k/x}$.  
\medskip

To prove \eqref{eq: with-part-2}, we establish pointwise bounds on $F^n_k$ and $G^n_k$. \cite[Proposition 6.1]{DOV} gives that for a fixed $a > 0$ we have
\begin{equation*}
    \label{eq: Fnk-bound-2}
    F^n_k(x, z_{k, a}) \le  2\sqrt{kx}(\sqrt{2} - a) +  R^{1,a}_{n,k}\,.
\end{equation*}
\cite[Proposition 6.1]{DOV} also yields the bound
\begin{equation}
    \label{eq: Fnk-bound}
    F^n_k(x, y) = 2\sqrt{2kx}  + R^{2}_{n,k}\,.
\end{equation}
Observe that Theorem~\ref{thm: Airy LPP deviation}, \cite[Proposition 6.1]{DOV} and weak convergence give the following uniform bounds with respect to $y$ for any fixed $\varepsilon > 0$ 
\begin{align*}
& \displaystyle\limsup_{n\to\infty}\left(\PP(R^{1,a}_{n,k}>\varepsilon \sqrt{k})+ \PP(R^{2}_{n,k}>\varepsilon \sqrt{k})\right)\le 2\PP \left(\frac{|\mathcal{A}[ (0, k) \rightarrow (x,1) ]-{2\sqrt{2kx}}|}{k^{1/2}} > \varepsilon\right)
\end{align*}
The triangle inequality \eqref{eq: triangle2} with   $x'=  x/(2a^2)$ gives
\begin{equation*}
    G^n_k(z_{k, a}, y) \le H_n(x', y) - F^n_k(x', z_{k, a}).
\end{equation*}
Now, $H_n(x', y)$ is equal to a rescaled and shifted Brownian last passage value by Proposition \cite[Proposition 4.1]{DOV}. Therefore, by Theorem~\cite[Theorem~2.5]{DOV} which gives bounds on single Brownian last passage values, it is tight in $n$ and hence $H_n(x', y)  = \mathfrak{o}(\sqrt{k})$. In particular, making this more quantitative gives for any $\varepsilon > 0$
\begin{equation*}
\displaystyle\limsup_{n\to \infty}\PP\left(\frac{|H_n(x,y)|}{\sqrt{k}}\ge 
\varepsilon\right)\lesssim \exp\left(-d\frac{\varepsilon^{3/8}k^{3/4}}{(|x-y|\lor 1)^{3/2}}\right)\,,
\end{equation*}
for some universal $d>0$. 

Moreover, \cite[Proposition 6.1]{DOV} gives that
$$
F^n_k(x', z_{k, a}) = 2 \sqrt{2 k x'} + 2 z_{k, a}  x' + \mathfrak{o}(\sqrt{k}) = \frac{\sqrt{kx}}{a}  + \mathfrak{o}(\sqrt{k})
$$
and so
\begin{equation*}
    G^n_k(z_{k, a}, y) \le -\frac{\sqrt{kx}}a  + R^{3,a}_{n,k}.
\end{equation*}
where the following uniform bounds wrt $y$ for any fixed $\varepsilon > 0$ are satisfied
\begin{equation*}
\displaystyle\limsup_{n\to\infty}\PP(R^{3,a}_{n,k}>\varepsilon \sqrt{k})\le \PP \left(\frac{|\mathcal{A}[ (0, k) \rightarrow (x,1) ]-{2\sqrt{2kx}}|}{k^{1/2}} > \varepsilon\right)\,.
\end{equation*}
We also have the bound
\begin{equation*}
    G^n_k(x - n^{1/3}, y) \le H_n(0, y) - F^n_k(0, x - n^{1/3}) = H_n(0, y) = \mathfrak{o}(\sqrt{k}).
\end{equation*}
The first equality here follows from the fact that  $F^n_k(0,\cdot)=0$, and the second equality again follows from \cite[Theorem~2.5]{DOV}.

Having now established the bound in \eqref{eq: with-part-2}, one obtains by a union bound and the convergence in distribution of $Z^n_k(x,y)\stackrel{d}{\rightarrow}Z_k(x,y)\,,n\to \infty$,

\begin{align*}
\PP\left( \left|Z_k(x, y) +\sqrt{\frac{k}{2x}}\right| \ge \varepsilon \sqrt{k}\right) &\le \displaystyle\liminf_{n\to \infty}\PP\bigg(\bigg|Z^n_k(x,y)+\sqrt{\frac{k}{2x}}\bigg|\ge  \varepsilon \sqrt{k}\bigg)\\
&\le \displaystyle\liminf_{n\to\infty}\PP\left( H^n(x,y)\le -\varepsilon^2\sqrt{kx}/2 + \displaystyle\displaystyle\max_{a\in A}\displaystyle\sum_{i=1}^3 |R^{i,a}_{n,k}|\right)\\
&  \le \liminf_{n\to\infty}\displaystyle\sum_{a\in A}\PP\left( |H^n(x, y)| + |R^{1,a}_{n,k}|+ |R^{2}_{n,k}|+ |R^{3,a}_{n,k}| \ge \varepsilon^2\sqrt{xk}/2\right)\\
& \le C (\varepsilon^2 \lor 1)\left( \PP \left(\frac{|\mathcal{A}[ (0, k) \rightarrow (x,1) ]-{2\sqrt{2kx}}|}{k^{1/2}} > \varepsilon\right)\right.\\
& \left.+  \exp\left(-d\frac{(\varepsilon^{3/4}) k^{3/4}}{(\sup_{(x,y)\in K} (|x| + |y| + 1)^2}\right)\right)\,,
\end{align*}
for some universal $C>0$. Now, to obtain the uniform bound, observe that by the monotonicity of $Z_k(\cdot, \cdot)$ in each of its of its arguments and the continuity of $1/\sqrt{2\cdot}$, it suffices to fix any $\varepsilon$ cover of $K$ with at most $\lceil 1/\varepsilon \rceil$ elements and use the pointwise estimates for fixed $x,y\in K$ at the expense of the $\lceil 1/\varepsilon \rceil$ term that comes from a union bound.
\end{proof}

We now introduce the coupling between the Airy sheet and the Airy line ensemble last passage values quoted from \cite[Definition 1.2]{DOV}, that will be used throughout the paper.

\begin{definition}(Airy sheet coupling)\label{def: Airy sheet}
The Airy sheet $\mathcal{S}(\cdot, \cdot) = \mathcal{L}(\cdot, 0; \cdot, 1)$ can be coupled with the (parabolic) Airy line ensemble $\mathcal{A}$ so that $\mathcal{S}(0,\cdot)=\mathcal{A}_1(\cdot)$ and almost surely for all $(x,y,z)\in K \subseteq (0, \infty)\times \R^2$, there exists a random integer $K_{x,y,z}$ such that for all $k\ge K_{x,y,z}$ \begin{equation*}\label{eq: defS}
\mathcal{A}[x_k\to (z,1)]-\mathcal{A}[x_k\to (y,1)]=\mathcal{S}(x,z)-\mathcal{S}(x,y)\,,
\end{equation*}
where $x_k=(-\sqrt{k/2x},k)$.
\end{definition}
We shall use this coupling of the Airy sheet throughout the paper. For $x\leq y\in \R$ and $\ell\in \N$, we shall denote the rightmost geodesic between $(x,\ell)$ and $(y,1)$ in the Airy line ensemble $\mathcal{A}$ by $\pi[(x,\ell)\to y]$. Next we define the infinite geodesics in the Airy line ensemble.

\begin{definition}\label{def: semi-inf geo} For any $x\in\R^+$ and $y\in \R$ with $x_k=(-\sqrt{k/2x},k)$, we define the geodesic $\pi[x\to y]$ as the almost sure pointwise limit of $\pi[x_k\to y]$ as $k\to \infty$, whenever the limit exists. We define the length of the geodesic $\pi[x\to y]$ as $\mathcal{S}(x,y)$.
\end{definition}
\begin{remark}
    The fact that these limits exist almost surely for all $(x,y)$ in a countable dense set of $\R^+\times \R$ is the content of \cite[Lemma 3.4]{sarkar2021brownian}.
\end{remark}

In the absolute continuity paper of \cite{sarkar2021brownian}, the authors obtain, using a coupling with the Airy sheet, the following semi-discrete variational characterisation of the Airy sheet in terms of semi-infinite geodesic intercepts (see Figure \ref{fig: Airy geodesic} for an illustration). For a proof of the following lemma, see \cite[Lemma 3.10]{sarkar2021brownian}.

\begin{figure}[t]
\includegraphics[width= 280pt]{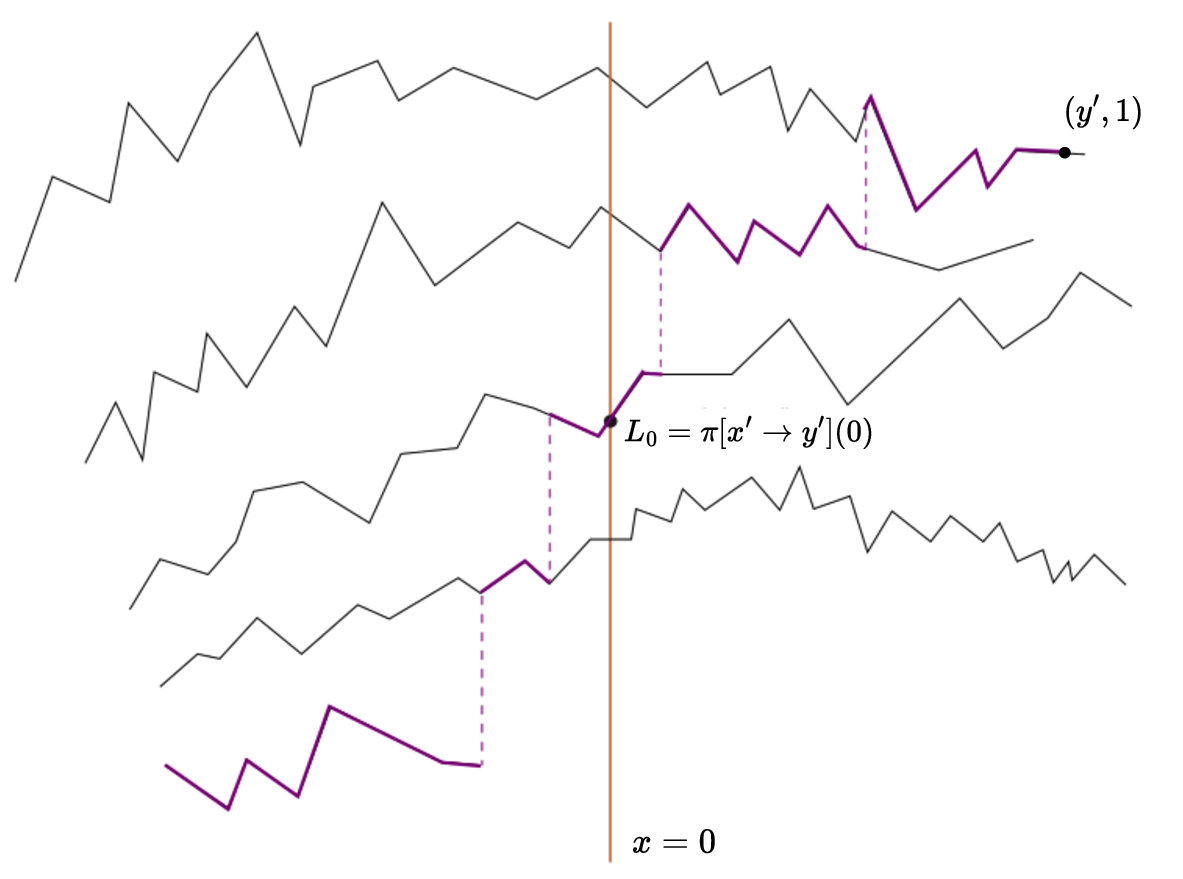}
\centering
\caption{Above is displayed the point $(0,L_0)$ at which the last passage path $\pi[x'\to y']$ on the Airy line ensemble $\mathcal{A} = (\mathcal{A}_1, \mathcal{A}_2,\cdots)$ (\color{purple} purple\color{black}) meets with the axis $\{x=0\}$, where $y'>1$. Here $L_0 = 3$ and the first four lines of $\mathcal{A}$ are shown. The last passage path $\pi[x'\to y']$ is defined in Definition 3.3 in \cite{sarkar2021brownian} and Definition \ref{def: semi-inf geo}.}
\label{fig: Airy geodesic}
\end{figure}

\begin{lemma}\label{lemma: Airy sheet variational formula}
Let $x_0>1$ and $y_0>1$ and $K\subseteq \R$ be a countable dense set. Let
\begin{equation*}L_0=\pi[x_0'\to y_0'](0)\,,\end{equation*}
for some $x_0',y_0'\in K$ with $x_0'\geq x_0$ and $y_0'\geq y_0$. Then almost surely for all $x\in [1,x_0]\cap K$ and all $y\in[1,y_0]$,
\begin{equation*}\mathcal{S}(x,y)=\max_{1\le \ell\le L_0}(\mathcal{A}[x\to(0,\ell)]+\mathcal{A}[(0,\ell)\to (y,1)])\,,\end{equation*}
where $\mathcal{A}$ is the Airy line ensemble that is coupled to the Airy sheet $\mathcal {S}$ as in Definition \ref{def: Airy sheet} and the semi-infinite last passage values
    \begin{equation}\label{eq: Airy limit}
    \mathcal{A}[x\to(0,\ell)]\equiv 
    \begin{cases}
        \mathcal{S}(x,0) & \ell = 1\\
        \displaystyle\lim_{k\to\infty}\mathcal{A}[x_k\to (0,\ell)]-\mathcal{A}[x_k\to (0,1)]+\mathcal{S}(x,0) & \ell >1\,,
    \end{cases}
    \end{equation}
   with $x_k = (-\sqrt{k/2x},k), k\in\N$. 
\end{lemma}

\begin{remark}
    The fact that the semi-infinite last passage values $\mathcal{A}[x \to (0, \ell)]$ exist (in fact the limits almost surely eventually stabilise) for any fixed $x> 0, \ell \ge 1$ almost surely and are well defined is the crux of Theorem~3.7 in \cite{sarkar2021brownian} and uses geometric properties of geodesics in the Airy line ensemble.

    For $x> 0$ and $\ell = 1$, $ \mathcal{A}[x\to(0,\ell)] = \mathcal{S}(x,0)$, which has H\"older $1/2-$ continuous sample paths and is the parabolic $\text{Airy}_{2}$ process. For $\ell >1$, the pathwise properties of $ \mathcal{A}[x\to(0,\ell)]$ for $x>0$ become a lot less clear. However, the uniform modulus of continuity estimates for the Airy line ensemble \ref{prop: global modulus airy} and geodesic coalescence in the Airy line ensemble allow one to obtain that the process $x\mapsto \mathcal{A}[x\to(0,\ell)]$ is \emph{continuous in probability}. 

    For any $x> 0, \ell\in\N$, the $\mathcal{A}[x\to(0,\ell)]$ attains a finite value and is
    \[
    \mathcal{F}_{-}\equiv \sigma(\{\mathcal{A}_i(x):x\leq0,i=1,2,\cdots\})
    \]
    measurable. This is the content of \cite[Lemma 3.8]{sarkar2021brownian} and \cite[Lemma 3.9]{sarkar2021brownian} respectively. The latter property follows from the Definition \ref{eq: Airy limit} and ergodic properties of the Airy line ensemble.

    One can observe by inspecting the proof of Lemma 3.10 in \cite{sarkar2021brownian} that for any $m\in \N $, on the event $\{L_0\leq m\}$ one has the almost sure equality 
    \begin{equation*}
    \mathcal{S}(x,y) = \max_{1\le \ell\le m}(\mathcal{A}[x\to(0,\ell)]+\mathcal{A}[(0,\ell)\to (y,1)])\,.
    \end{equation*}
\end{remark}

Thus, good control over $L_0$ should translate to control over the Airy sheet and hence the KPZ fixed point, owing to its variational characterisation \eqref{eq: KPZ fixed point}. The structure of jump times of (semi-infinite) geodesics on the Airy line ensemble and Lemma \ref{lemma: jump times parabola conc} give the following Theorem~which is the main result of this subsection.

\begin{theorem}\label{thm: intercept tail bound}
For any $x_0>1, y_0>1$, there exists  $d>0$ such that the semi-infinite geodesic intercept $L_0 = \pi[x_0\to y_0](0)$ as given in the statement of Lemma \ref{lemma: Airy sheet variational formula} satisfies the tail bounds
    \begin{equation*}
        \displaystyle\sup_{k\in \N} \exp(dk^{1/126})\cdot \PP(L_0\geq k)<\infty\,.
    \end{equation*}
for some universal $x_0, y_0$-dependent constant $d>0$. Keeping track of the $x_0, y_0$-dependence (at the expense of the exponent) gives that there exist universal $\theta , \eta, c,d > 0$ such that for all $k\ge 1$
 \begin{equation*}
        \PP(L_0\geq k)\le \mathrm{e}^{c(x_0+y_0)^\theta}\mathrm{e}^{-dk^{\eta}}\,.
    \end{equation*}
    \end{theorem}

\begin{proof}
    First observe that for any $k\in \N$, by the Skorokhod coupling in \cite[p.43]{DOV}, the almost sure pointwise limits  $Z_k(x_0,y_0)$ of the jump times $Z^n_k(x_0,y_0)$ correspond to the jump times of the semi-infinite geodesic $\pi[x_0 \to y_0]$. Thus, by Lemma \ref{lemma: jump times parabola conc} for $k\ge 1$ and $\varepsilon = 1/\sqrt{2x_0}$,
    \begin{align*}
        \PP(L_0\ge k) = \PP(Z_k(x_0,y_0)\ge 0)&\le \displaystyle\liminf_{n\to \infty}\PP(Z^n_k(x_0,y_0)\ge 0)\\
        & \le \PP\left( \left|Z_k(x_0,y_0) +\sqrt{\frac{k}{2x_0}}\right| \ge \varepsilon \sqrt{k}\right)\\
        &\le c\sqrt{x_0}\lor 1/x_0 \left( \PP \left(\frac{|\mathcal{A}[ (0, k) \rightarrow (x,1) ]-{2\sqrt{2kx}}|}{k^{1/2}} > \varepsilon\right)\right.\\
        &\left.+  \exp\left(-d\frac{\varepsilon^{3/4} k^{3/4}}{(|x_0| + |y_0| + 1)^2}\right)\right)\\
        & \stackrel{\mathrm{Thrm} \ref{thm: Airy LPP deviation}}{\le} c\exp(d (x_0)^3)k^2\left(\exp(-d\varepsilon^{1/2}  k^{1/28})\right.\\
        &\left. + \exp\left(-d\varepsilon k^{1/126}/(x_0)^{3/8}\right)+ \exp\left(-d\frac{\varepsilon^{3/4} k^{3/4}}{(|x_0| + |y_0| + 1)^2}\right)\right)\,,
    \end{align*}
    for some $c,d>0$ and all $1/\sqrt{2x_0}< k ^{1/126}$, whence the result follows.
\end{proof}

\section{Regularity of finite-depth truncations of the KPZ fixed point}\label{sec: finite depth truncation}

In this section, we obtain a quantitative comparison of the spatial increments of `finite depth truncations' of the KPZ fixed point in terms of the Wiener measure and last passage values of semi-infinite geodesics in the Airy line ensemble. We crucially use the variational formula for the KPZ fixed point \eqref{eq: KPZ fixed point} and the coupling in Definition \ref{def: Airy sheet}. This is achieved through the Brownian Gibbs property of the Airy line ensemble, which further reduces the problem to estimating the Radon-Nikodym derivatives of inhomogeneous Brownian LPP with non-decreasing initial data. This is done in Theorem~\ref{thm: main companion}. Technical input from \cite{dauvergne2024wienerdensitiesairyline} allows us to estimate inverse acceptance probabilities that appear in the estimates. Combining the above leads to Theorem~\ref{thm: finite depth KPZ estimates a priori}.

First, we need an appropriate definition of `support' compatible with the `max-plus' nature of the directed landscape.
\begin{definition}(max-plus support)\label{def: max support}
    Let $f:\R\to \R\cup \{-\infty\}$ be a Borel function. We define the \textbf{max-plus} support of $f$ to be the set
    \begin{equation*}
    \mathrm{supp}_{-\infty}(f) : = \{x\in \R: f(x)\neq -\infty\}\,.
    \end{equation*}
\end{definition}

By $1:2:3$ scaling, we lose no generality in considering the KPZ fixed point at unit time, $\mathfrak{h}_1(\cdot)$ with initial data $h_0:\R\to \R\cup \{-\infty\}$, which can be written more explicitly as
    \begin{equation*}
    \mathfrak{h}(y) = \sup_{x\in \mathrm{supp}_{-\infty}(h_0)}(h_0(x)+\mathcal{S}(x,y))\,,
    \end{equation*}
    where $\mathcal{S}(\cdot, \cdot)$ denotes the Airy sheet, see Definition \ref{def: Airy sheet}. If the `max-plus' support of initial data is compact, by translation symmetries of the Airy sheet \cite[Lemma 9.1]{DOV} we do not lose generality if we translate the support of the initial data to lie in $[1,x_0]$, for some $x_0>1$. We can additionally set the interval of comparison to $[1,y_0]$, $y_0>1$ (by possibly enlarging it) as in Lemma \ref{lemma: Airy sheet variational formula}, at no loss of generality. Moreover, the metric composition law for the Airy sheet \cite[Proposition 9.2]{DOV} will allow us to effectively only need to consider continuous initial data for which we can replace $\supp{h_0}$ with $\Q\cap \supp{h_0}$. Thus, as a first step when initial data is compactly supported, we will henceforth make the following assumptions on the initial data:
    \begin{itemize}\label{assumption initial data}
        \item[1.] $\mathrm{supp}_{-\infty}(h_0)\subseteq \Q$ bounded and
        \item[2.] $\overline{\mathrm{supp}_{-\infty}(h_0)}\subseteq(0,\infty)$. 
    \end{itemize}
    
    Now, from Lemma \ref{lemma: Airy sheet variational formula} with $K = \mathrm{supp}_{-\infty}(h_0)$ and $x_0, y_0>1$ as above, there is a random constant $L_0$ with tails as in Theorem~\ref{thm: intercept tail bound}, such that almost surely for all $y\in[1,y_0]$
    \begin{equation*}
    \mathfrak{h}(y) = \max_{\ell\leq L_0}(G_{\ell}+\mathcal{A}[(0,\ell)\to (y,1)])\,,
    \end{equation*}
    where
    \begin{equation}\label{eq: bdry data}
    G_\ell \equiv \sup_{ x\in \mathrm{supp}_{-\infty}(h_0)}(h_0(x)+\mathcal{A}[x\to(0,\ell)])\,.
    \end{equation}
    Observe that using the Remark following Lemma \ref{lemma: Airy sheet variational formula} and the notation there, $G_\ell< \infty$ and $\mathscr{F}_-$-measurable.

Thus, on the event $\{L_0\leq m\}$ the KPZ fixed point (at unit time) has the expression 
\begin{equation*}
\begin{array}{ll}
    \mathfrak{h}(y) &= \displaystyle\max_{x\in \mathrm{supp}_{-\infty}(h_0)}(h_0(x)+\mathcal{S}(x,y))= \displaystyle\max_{\ell \le m}(G_\ell + \mathcal{A}[(0,\ell)\to (y,1)])\,.
\end{array}
\end{equation*}

Since we can control the geodesic intercept $L_0$ using Theorem \ref{thm: intercept tail bound}, we are now led to studying the quantitative Brownian regularity of the the laws of the `truncated' profiles
\begin{equation}\label{eq: fin depth KPZ}
H_m(\cdot) = \displaystyle\max_{\ell \le m}(G_\ell + \mathcal{A}[(0,\ell)\to (\cdot, 1)])\,,\qquad m \ge 1
\end{equation}
against that of rate two Brownian motion on compacts, which are already known to satisfy the absolute continuity relation $\mathrm{Law}_{H_m}\ll \mathfrak{B}_{*, *}$ on compacts, \cite[Proposition 5.1]{sarkar2021brownian}.

We show in particular that the moments of the Wiener densities (of the laws restricted to any fixed compact $K\subset \R$)
$$\norm{\frac{\diff\mathrm{Law}_{H_m}}{\diff\mathfrak{B}_{*, *}}}_{L^p(\mathfrak{B}_{*, *})}< \infty $$ for $m\ge 1$ and importantly obtain asymptotics in $m\ge 1$. 

Such estimates are obtained in Theorem \ref{thm: main companion} in \cite{tassopoulos2025inhomogeneousbrownian} when one replaces the Airy environment with independent Brownian motions. To reduce the former to the latter, we will use the \textbf{Brownian Gibbs property} which the parabolic Airy line ensemble satisfies. However, this leads to the technical challenge of estimating Brownian inverse acceptance probabilities with Airy line ensemble endpoints, see Subsection \ref{subsec: Airy line ensemble}.

The crucial technical input that allows us to overcome this challenge comes from \cite{dauvergne2024wienerdensitiesairyline}, in order to estimate the inverse acceptance probability that comes from using the Brownian Gibbs property. In particular, we will need the following slight modification of \cite[Lemmas 3.2, 3.3]{dauvergne2024wienerdensitiesairyline} 
The goal of the next few lemmas is to estimate Brownian inverse acceptance probabilities (recall the terminology from Section \ref{sec: notation}) in terms of analytically tractable random variables involving a possibly resampled Airy line ensemble, which we can estimate using pathwise properties thereof.

\begin{lemma}{\cite[Lemma 3.3]{dauvergne2024wienerdensitiesairyline}}
	\label{lemma: bridge-shift-calc}
	Fix $t>1$, $a<s<t<b$ and let $\underline{x}, \underline{y} \in \R^m_>$. Let $g \in \mathscr{C}_{*,*}([a, b])$ be such that $g(a) < x_m, g(b) < y_m$.
	
	Let $B$ be a $m$-tuple of independent Brownian bridges from $(a, \underline{x})$ to $(b, \underline{y})$, conditioned on the event
	$$
	\mathrm{NoInt}([a, s] \cup [t, b], g) \quad \mathrm{ or }\quad \mathrm{NoInt}([a, b], g) 
	$$ 
	Fix $\varepsilon \in (0,1)$ and define $\iota = (1/m, 1/(m+1), \cdots, 1/(2m)) ,  {\bf 1} = (1, 1, \dots, 1) \in \R^m$, and for $\alpha, \beta \ge 0$ define $f^{\alpha, \beta} \in \mathscr{C}^m_{*, *}([a, b])$ by letting 
	$$
	f^{\alpha, \beta}(a) = 0, \quad f^{\alpha, \beta}(s) = f^{\alpha, \beta}(t) = \alpha \iota + \beta {\bf 1}, \quad f^{\alpha, \beta}(b) = 0,
	$$
	and so that $f^{\alpha, \beta}$ is linear on each of the pieces $[a, s], [s, t], [t, b]$.
	
	Then for $f \in \mathrm{NoInt}([a, s] \cup [t, b], g)$ (or $\mathrm{NoInt}([a, b], g) $)  we have the pointwise lower bound on the density of the law $\mu_B'$ of $B - f^{\alpha, \beta}$ against the law $\mu_B$ of $B$, on the set $\mathrm{NoInt}([a, 1] \cup [t, b], g)$ (or $\mathrm{NoInt}([a, b], g) $) where $\mu_B'$ is absolutely continuous with respect to $\mu_B$
	\begin{align}
	\label{eq: Radon-Nikodym -B-bd}
	\frac{\diff \mu^\prime_B}{\diff \mu_B}(f) &\ge \exp \left(-\zeta^2\frac{m(\alpha/m + \beta)^2}{4} - \zeta\frac{(\alpha/m + \beta)\sum_{i=1}^m [(f_i(s) - x_i)^+ + (f_i(t) - y_i)^+]}{4} \right)\\
	   &\ge \exp \left(-\zeta^2\frac{m(\alpha/m + \beta)^2}{4} - \zeta\frac{(\alpha/m + \beta)\sum_{i=1}^m [(f_i(s) - x_i)^2_+ + (f_i(t) - y_i)^2_+]}{4} \right)\,,
	\end{align} 
    where $\zeta = \frac{1}{\min(s-a, b-t)}$.
\end{lemma}

\begin{proof}
	Let $\nu$ be the law of $m$ independent Brownian bridges from $(a, \underline{x})$ to $(b, \underline{y})$. Observe first that $f^{\alpha, \beta}$ being piecewise linear, it is in the Sobolev space $W^{1,2}([a,b])$ and so by Girsanov's Theorem, for a rate two Brownian motion $W$ starting from $x\in \R$ on $[a, b]$, the Radon-Nikodym derivative of the process $W-f^{\alpha, \beta}$ against $W$ is given by
    \begin{equation*}
    \exp\left(-\frac{1}{2}\displaystyle\int_{[a,b]} \dot{f}^{\alpha, \beta}(s) \diff W_s - \frac{1}{4}\int_{[a,b]} (\dot{f}^{\alpha, \beta})^2(s) \diff s\right)
    \end{equation*}
    \begin{equation*}
     = \exp\left(-\frac{1}{2}(W_s-x)\frac{ f^{\alpha, \beta}(s)}{s-a} + \frac{1}{2}(W_{b}-W_{t})\cdot \frac{f^{\alpha, \beta}(t)}{b-t}  - \frac{1}{4}\left(\left(\frac{f^{\alpha, \beta}(s)}{s-a}\right)^2+ \left(\frac{f^{\alpha, \beta}(t)}{b-(t)}\right)^2\right)\right).
    \end{equation*}
    Now conditioning on $W_{b}=y\in \R$ and using the uniqueness of regular conditional distributions and the regularity of the conditional measures for Brownian bridges and the above Radon-Nikodym transform thereof, we can conclude by independence that $\mu_B, \mu_B'$ are absolutely continuous with respect to $\nu$ with densities
    \begin{equation*}
	\frac{d \mu_B}{d \nu}(f) = \frac{1}Z \mathbf{1}(f \in \mathrm{NoInt}([a, s] \cup [t, b], g))\,,
    \end{equation*}
    \begin{equation}
    \label{eq: Radon-Nikodym -nus}
    \begin{array}{cc}
	&\frac{d \mu_B'}{d \nu}(f) = \frac{1}{Z} \mathbf{1}(f + f^{\alpha, \beta} \in \mathrm{NoInt}([a, s] \cup [t, b], g)) \\
	&\displaystyle\cdot \exp \left(-c\frac{2 (f(s) -\underline{x}) \cdot \frac{f^{\alpha, \beta}(s)}{s-a} + \|\frac{f^{\alpha, \beta}(s)}{s-a}\|^2}{4} -c\frac{2 (f(t) -\underline{y}) \cdot \frac{f^{\alpha, \beta}(t)}{b-t)} + \| \frac{f^{\alpha, \beta}(t)}{b-t}\|^2}{4}\right)\,. 
	\end{array}
    \end{equation}
	where  $Z = \PP_{a, b}(\underline{x}, \underline{y}, g, [a, s] \cup [t, b])$ is a normalizing factor. Now, if $f$ is in the set $\mathrm{NoInt}([a, s] \cup [t, b], g)$, then so is $f + f^{\alpha, \beta}$. Hence the right-hand side of \eqref{eq: Radon-Nikodym -B-bd} is bounded below by the exponential factor \eqref{eq: Radon-Nikodym -nus}. We can bound \eqref{eq: Radon-Nikodym -nus} below by using $0 \le f^{\alpha, \beta} \le (\alpha/m + \beta) {\bf 1}$, which yields the desired bound.
\end{proof}

We now record \cite[Lemma 3.1]{dauvergne2024wienerdensitiesairyline}, stated slightly more generally, which allows one to estimate the conditional inverse acceptance probability by the inverse of a conditional probability over a resampled ensemble, by `stepping out' of the original interval and applying the Brownian Gibbs on that larger interval. For $a<b\in \R, k\in \N$ set $\mathcal{F}^{[a,b]}_k\equiv \sigma(\{A_i(x): (i,x)\notin \llbracket 1,m\rrbracket \times (a,b)\})$.

\begin{lemma}{\cite[Lemma 3.1]{dauvergne2024wienerdensitiesairyline}}\label{lemma: inv acceptance prob cond dauv}
	For every $a<y_0+1<y_0+2<b$ $k\in \N$, we have for any non-negative measurable functional $F$ on $\mathscr{C}(\llbracket 1, k\rrbracket \times [a,y_0+2]\cup [y_0+1,b])$ and any $\mathscr{F}^{[a,b]}_m$-measurable $U$
	\begin{equation}
    \begin{array}{ll}
	\label{E:big-eqn}
		&\displaystyle\mathbb{E}_{\mathscr{F}^{[a,b]}_k} \left(\frac{F(U, \mathcal{A}|_{\llbracket 1, k\rrbracket \times [a,y_0+1]\cup [y_0+2,b]})}{\mathfrak{B}^{[y_0+1,y_0+2]}_{\mathcal{A}^k(y_0+1),\mathcal{A}^k(y_0+2)}(\mathrm{NoInt}([y_0+1, y_0+2]\,, \mathcal{A}_{k+1}))} \right)\\
        & = \displaystyle\left(\frac{\mathbb{E}_{\mathscr{F}^{[a,b]}_k} [ F(U, \mathfrak{B}|_{\llbracket 1, k\rrbracket \times [a,y_0+1]\cup [y_0+2,b]}) ]}{\mathbb{E}_{\mathscr{F}^{[a,b]}_k} [\mathfrak{B}^{[y_0+1,y_0+2]}_{\mathfrak{A}^k(y_0+1),\mathfrak{A}^k(y_0+2)}(\mathrm{NoInt}([y_0+1, y_0+2]\,, \mathcal{A}_{k+1}))]} \right)\,,
        \end{array}
	\end{equation}
        where $\mathfrak{A}$ is a line ensemble constructed by letting $\mathfrak{A}_{i}(r) = \mathcal{A}_i(r)$ for $(i, r) \notin \llbracket 1, k\rrbracket \times [a, b]$, and let $\mathfrak{B}^k|_{[a, b]}$ have the conditional law given $\mathcal{A}$ outside of $\llbracket 1, k\rrbracket \times [a, b]$ of $k$ independent Brownian bridges from $(a, \mathcal{A}^k(a))$ to $(b, \mathcal{A}^k(b))$ conditioned on the event $\mathrm{NoInt}(\mathcal{A}_{k+1}, [a, y_0+1] \cup [y_0+2, b])$.
\end{lemma}

The following lemma is a refinement of \cite[lemma 3.2]{dauvergne2024wienerdensitiesairyline}, where non-intersection probabilities are estimated from below by more analytically tractable quantities. They are in turn controlled by the modulus of continuity estimates for the Airy line ensemble in Corollary \ref{cor: global mod airy M}.
\begin{lemma}{\cite[Lemma 3.2]{dauvergne2024wienerdensitiesairyline}}
	\label{lemma: acceptance prob LB}
	Fix $a<s<t<b$, $\varepsilon >0 $ and define $\mathscr{F}^{[a, b]}_m$-measurable random variables
	\begin{align*}
	D &= D(m, t)= 1 + \max_{r, r' \in [s, t]} |\mathcal{A}_{m+1}(r) - \mathcal{A}_{m+1}(r')|, \\
M &= M(m, y_0) = 1 + \max_{r, r' \in [a, b]} |\mathcal{A}_{m+1}(r)- \mathcal{A}_{m+1}(r')| + \max_{i \in \llbracket 1, m\rrbracket } |\mathcal{A}_i(b) - \mathcal{A}_i(a)|.
	\end{align*}
	Then with $\mathfrak{A}$ as in Lemma \ref{lemma: inv acceptance prob cond dauv}, we have
	\begin{equation*}
    \begin{array}{ll}
	& \mathbb{E}_{\mathscr{F}^{[a, b]}_m} [\mathfrak{B}^{[s,t]}_{\mathfrak{A}^m(s),\mathfrak{A}^m(t)}(\mathrm{NoInt}([s, t]\,, \mathcal{A}_{m+1}))]]\\
    & \ge C\exp\left(- m^{1+\varepsilon}(\zeta^2D^2 + \zeta MD) - m^{2+\varepsilon} \zeta D - c m\log (b-a) \right) \cdot c^m |t-s|^{\frac{m}{2}}\cdot \exp \left( - \frac{dm^{3-\varepsilon} (t-s)}{\varepsilon^2}\right)
	\end{array}
    \end{equation*}
    for some  $\zeta = \frac{3}{\min(s-a,t-b)}$ and constants $c, d, C>0$ independent of $m\in \N, \varepsilon >0$ (suppressing the dependence on $s, t$).
\end{lemma}

\begin{proof}
    All statements in the proof are conditional on $\mathscr{F}^{[a, b]}_m$. Define the $\mathscr{F}^{[a, b]}_m$-measurable vector 
	$$
	\underline{z} = (D + m^{\varepsilon/2}, D + (m-1)^{\varepsilon/2}, \dots, D + 1)
	$$
and the $\mathscr{F}^{[a, b]}_m$-measurable set
\begin{equation}
\label{eq: x-y}
O = \{(\underline{x}, \underline{y}) \in \R^m_> \times \R^m_> : x_m > \mathcal{A}_{m+1}(s), y_m > \mathcal{A}_{m+1}(t) \}.
\end{equation} 
By the definition of $D$, for $(\underline{x}, \underline{y}) \in O$ we have noting that $m^{\varepsilon/2}-(m-1)^{\varepsilon/2}\ge \varepsilon/(2m^{1-\varepsilon/2})$, $m\ge 2$ inclusion and independence
\begin{equation}
\label{eq: P0t}
\PP_{s, t}(\mathrm{NoInt}(\underline{x} + \underline{z}, \underline{y} + \underline{z}, \mathcal{A}_{m+1})) \ge \PP\Big (\sup_{s \le r \le t} |B(r)| \le \varepsilon/(4m^{1-\varepsilon/2})\Big)^m,
\end{equation}
where $B$ is a rate two Brownian bridge from $(s,0)$ to $(t, 0)$. By Lemma \ref{lemma: brownian bridge maximum}, we have 
\begin{equation*}\PP(\displaystyle\max_{s\le r\le t}|B_t|\le \varepsilon/(4m^{1-\varepsilon/2}))\ge c |t-s|^{\frac{1}{2}}\exp \left( - \frac{dm^{2-\varepsilon} (t-s)}{\varepsilon^2}\right)\,,\quad \text{ for all }\varepsilon>0, m\ge 1\end{equation*}
and so the right hand side in (\ref{eq: P0t}) is bounded below by
\begin{equation*} c^m |t-s|^{\frac{m}{2}} \exp \left( - \frac{dm^{3-\varepsilon} (t-s)}{\varepsilon^2}\right)\end{equation*}
for some constant $c>0$, which may change from line to line. Therefore letting $\mu_\mathfrak{A}$ denote the conditional law of $(\mathfrak{A}^{m}(s), \mathfrak{A}^{m}(t))$ given $\mathscr{F}^{[a, b]}_m$, to complete the proof it suffices to find a set $A \subseteq O$ such that 
$
\mu_\mathfrak{A}(A + (\mathbf{z}, \mathbf{z}))
$
is large. Fix $\Delta > 0$ and let $A_\Delta$ be the $\mathscr{F}^{[a, b]}_m$-measurable subset of $(\underline{x}, \underline{y}) \in O$  where 
$$
x_i \le \mathcal{A}_i(a) + \Delta, \qquad \qquad y_i \le \mathcal{A}_i(b) + \Delta 
$$
for all $i \in \llbracket 1, m \rrbracket$. Then by Lemma \ref{lemma: bridge-shift-calc} with $\alpha = 1, \beta = D$, we have
\begin{align}
\nonumber
\mu_\mathfrak{A}(A_\Delta + (\mathbf{z}, \mathbf{z})) &\ge \mu_\mathfrak{A}(A_\Delta) \inf_{(\underline{x}, \underline{y}) \in A_\Delta} \exp\biggl(
-\frac{\zeta^2 m^{1+\varepsilon}(1 + D)^2}{4} \\
&- \zeta \frac{m^{\varepsilon/2}(1 + D) \sum_{i=1}^m ((x_i - \mathcal{A}_i(a))^+ + (y_i - \mathcal{A}_i(b))^+)}{4}
\biggr) \\
\label{eq: mufB}
& \ge \mu_\mathfrak{A}(A_\Delta) \exp \left(-\zeta^2 m^{1+\varepsilon} D^2 - \zeta m^{1+\varepsilon/2} D \Delta \right)
\end{align}

where $\zeta = \frac{3}{\min(s-a, b-t)}$. In the final line we have used that $1 + D \le 2D$.
It remains to find $\Delta$ where $\mu_\mathfrak{A}(A_\Delta)$ is large.

Define vectors $\underline{w}^{a,b}$ for $i\in \llbracket 1, m \rrbracket$ at $a, b$ respectively, where 
\begin{equation*}
    \underline{w}_i^{a,b} =  M + i + \mathcal{A}_i(\{a, b\})\,.
\end{equation*}

 By a monotonic coupling for Brownian bridges, see \ref{lemma: bridge monotonicity}, on the interval $[a, b]$, the $m$-tuple $(\mathfrak{A}_{1}, \dots, \mathfrak{A}_{m})$ is stochastically dominated by $m$ independent Brownian bridges $B = (B_1, \dots, B_m)$ from $(a, \underline{w}^{a})$ to $(b, \underline{w}^{b})$ conditioned on the event
 $$
  \mathrm{NoInt}([a, s] \cup [t, b], \mathcal{A}_{m+1}).
 $$
 Now, let $L \in \mathscr{C}^m([a, b]; \R)$ be the function whose $i$th coordinate $L_i$ is the linear function satisfying $L_i(a,b) = w_i^{a, b}$. By \cite[Lemma 2.5]{dauvergne2024wienerdensitiesairyline}, we have $f \in  \mathrm{NoInt}([a, s] \cup [t, b], \mathcal{A}_{m+1})$ for any sequence of bridges $f$ from $(a, \underline{w}^{a})$ to $(b, \underline{w}^{b})$ when $\norm{f-L}_{\infty, [a,b]} \le 1/100$ with probability bounded below by $c\mathrm{e}^{-dm\log (b-a) }$ for positive constants $c,d> 0$. This allows us to estimate
 \begin{align*}
 \PP(B_i(r) &\le  M + i + 2 + \mathcal{A}_i(-b)\lor \mathcal{A}_i(b)\quad \forall i \in \llbracket 1, m \rrbracket, r = s, t)\\
 & \ge \PP(\norm{B-L}_{\infty, [a,b]} < 1/100)\ge c\mathrm{e}^{-d m\log (b-a)}\,.
 \end{align*}
 Observing that 
 $$
 M + i+ 2 + \mathcal{A}_i(a)\lor \mathcal{A}_i(b) - \mathcal{A}_i(\{a, b\})\le 5M+ 3m
 $$
 for all $i$, we can conclude that 
 $$
 \mu_\mathfrak{A}(A_{5M+3m}) \ge \PP((B(s), B(t)) \in A_{5M+3m}) \ge c\mathrm{e}^{-dm\log (b-a)}. 
 $$
 Combining this with the bound on \eqref{eq: P0t} and \eqref{eq: mufB} and simplifying yields the result.
\end{proof}

Before proving the quantitative Brownian regularity of finite depth truncations of the KPZ fixed point against Brownian motion, we need one final preliminary result estimating the expected value of the inverse acceptance probability that appears in the conditioning when applying the Brownian Gibbs property to the Airy line ensemble, which is the content of the following lemma.

\begin{lemma}\label{lemma: inv acceptance prob cond}
Fix $m\in \N$, $t > 0$ $\varepsilon > 0$, then there exists some universal $\eta, \theta >0$ such that the following estimate holds
\begin{equation*}
\mathbb{E}\left [\displaystyle\frac{1}{\mathfrak{B}^{[0,t]}_{(\mathcal{A}_i(0))_{i=1}^m, (\mathcal{A}_i(t))_{i=1}^m}(\mathrm{NoInt}([0,t], \mathcal{A}_{m+1}))} \right ] = O_t(\mathrm{e}^{d_{t, \varepsilon}m^{6 + \varepsilon}}) = O(\mathrm{e}^{d_{\varepsilon}t^\eta}\mathrm{e}^{dm^{\theta}}).
\end{equation*}
for some $d_{\varepsilon} > 0$. 
\end{lemma}
\begin{proof}
    
We first begin by `stepping outside' of the interval $[0,t]$ and condition on $\mathscr{F}^{[-T_m,U_m]}\subseteq \mathscr{F}^{[0,t]}_m$, for $T_m, U_m>0 $ sufficiently large, to be chosen later. To control the inverse acceptance probability (\ref{eq: inv acc prob}) conditional on $\mathscr{F}^{[-T_m,U_m]}$, we use Lemma \ref{lemma: acceptance prob LB} and the lower bound provided by Lemma \ref{lemma: acceptance prob LB} to obtain for all $\varepsilon>0$
\begin{align*}
&\mathbb{E}\left [\displaystyle\frac{1}{\mathfrak{B}^{[0,t]}_{(\mathcal{A}_i(0))_{i=1}^m, (\mathcal{A}_i(t))_{i=1}^m}(\mathrm{NoInt}([0,t], \mathcal{A}_{m+1}))} \right ]\le c^m t^{-\frac{m}{2}} \\
&\cdot \mathbb{E}\left [\exp\left(\big(m^{1+\varepsilon}(\zeta^2 D^2 + \zeta MD+ m^{2+\varepsilon} \zeta D)\big) + c'm\log(U_m+T_m)  \right)\cdot \exp \left( \frac{dm^{3-\varepsilon} t}{\varepsilon^2}\right)  \right ]
\end{align*}
for some constants $c, c'> 0$, and
\begin{itemize}
    \item[1.] the $\mathscr{F}^{[-T_m, U_m]}_m$-measurable random variables
	\begin{align*}
	D &= D(m, y_0)= 1 + \max_{r, r' \in [0, t]} |\mathcal{A}_{m+1}(r) - \mathcal{A}_{m+1}(r')|, \\
M &= M(m, T_m, U_m) = 1 + \max_{r, r' \in [-T_m, U_m]} |\mathcal{A}_{m+1}(r) - \mathcal{A}_{m+1}(r')|\\
&+ \max_{i \in \llbracket 1, m\rrbracket } |\mathcal{A}_i(U_m) - \mathcal{A}_i(-T_m)|
	\end{align*}
    \item[2.] $\zeta = \frac{3}{\min(T_m, U_m-t)}$.
\end{itemize}

We will henceforth take $T_m = U_m = O(m^{\alpha})+t$ for some $\alpha >0$ so that $\zeta = m^{-\alpha}$. In particular, taking $\alpha = 2+2\varepsilon + \eta$, with $\eta >0 $, we estimate using the elementary inequality for $a,b\ge 0$, $2ab\le a^2 + b^2$
\begin{align*}
         \mathbb{E}\left [\exp\left(m^{1+\varepsilon}(\zeta^2 D^2 + \zeta MD) + m^{2+\varepsilon} \zeta D\right) \right ]&\le \frac{1}{2}\mathbb{E}\left [\exp\left(2m^{-1}D^2 \right) \right ]\\
         &+ \frac{1}{4}\mathbb{E}\left [\exp\left(4m^{1+\varepsilon}\zeta MD \right) \right ] + \frac{1}{4}\mathbb{E}\left[\exp\left( 4 m^{2+\varepsilon} \zeta D\right)\right]\\
         & \le \mathbb{E}\left [\exp\left(2m^{-1}D^2 \right) \right ]+ \frac{1}{2}\mathbb{E}\left [\exp\left(c m^{2+2\varepsilon}\zeta^2 M^2 \right) \right ]\\
         & \le O(\mathrm{e}^{dt^\theta}\mathrm{e}^{dm^{2}\log m }) + \mathbb{E}\left [\exp\left(cm^{-\eta}\zeta M^2 \right) \right ]\\
    \end{align*}
for some constants $c, d, \theta>0$.

By Corollary \ref{cor: global mod airy M}, we  have that there exist some positive $C_1, C_2,d > 0$ independent of $m$ such that for all $a>0$,
     \begin{equation*}
     \PP(M> a) \le C_1 \mathrm{e}^{dm^{3\alpha}} \mathrm{e}^{-C_2 a^2/m^\alpha}\,.
     \end{equation*}
In particular, keeping track of $t$-dependence, we obtain
   \begin{equation*}
     \PP(M> a) \le C_1 \mathrm{e}^{dm^{\theta}+dt^\theta} \mathrm{e}^{-C_2 a^2}\,,
     \end{equation*}
for some absolute constant $\theta>0$.

We now estimate
\begin{align*}
    O_t(\mathrm{e}^{dm^{2}\log m }) &+ \mathbb{E}\left [\exp\left(cm^{-\eta}\zeta M^2 \right) \right ]\\
         & \le O_t(\mathrm{e}^{dm^{2}\log m }) + 2c\displaystyle\int_{0}^{\infty} a\exp\left(cm^{-2-2\varepsilon -2\eta} a^2 \right)\PP(M > a)\diff a \\
         & \le O_t(\mathrm{e}^{dm^{2}\log m }) + O(\mathrm{e}^{dm^{3\alpha}})\displaystyle\int_{0}^{\infty} a\exp\left(cm^{-2-2\varepsilon -2\eta} a^2 - C_2  m^{-2-2\varepsilon -\eta}a^2 \right)\diff a \\
         & = O_t(\mathrm{e}^{dm^{6+6\varepsilon + 6\eta}})\,,
\end{align*}
for positive constants $c>0$, concluding the proof. One can obtain analogous expressions, keeping track of the $t$-dependence to finally conclude the proof.
\end{proof}

We are now in a position to obtain the quantitative control of the spatial increments of finite depth truncations of the KPZ fixed point, \eqref{eq: fin depth KPZ} started from compactly supported initial data in terms of the Wiener measure and Airy line ensemble data.

Recall, the notation
\begin{equation*}
    G_\ell \equiv \sup_{ x\in \mathrm{supp}_{-\infty}(h_0)}(h_0(x)+\mathcal{A}[x\to(0,\ell)])\,,\quad \ell \ge 1
    \end{equation*}
    for the `boundary data' appearing in \eqref{eq: fin depth KPZ}.
\begin{theorem}\label{thm: finite depth KPZ estimates a priori}
 Fix $m\in \N$, $T_m>0 , U_m>y_0+2$ and define the random continuous function
\begin{equation*}
    H_m(y) = \displaystyle\max_{\ell \le m}(G_\ell + \mathcal{A}[(0,\ell)\to (y,1)]), \quad \text{ for } y\in [1,y_0].
\end{equation*}

Then with $\mu$ the rate two Wiener measure $\mu$ on $[0,y_0-1]$, $H_m(\cdot+1)-H_m(1)$, satisfies for all $p, r> 1$, $A\subseteq \mathscr{C}_{*, *}([0,y_0-1])$ Borel and $a> 0$ the estimates
\begin{align*}
    \PP ( H_m(\cdot + 1)&-H_m(1)\in A ) \le O_{y_0}(\exp(m^7))\cdot\exp\left({\frac{y_0m^2 a^2}{4}\big(\frac{r/(r-1)}{ry_0/(r-1)+1}-\frac{1}{y_0+1}\big)}\right)\\
        & \cdot \displaystyle\sup_{\max_{1\le \ell \le m}|G_\ell - G_1| \le a}\norm{Q^{m, G}}_{L^{2r/(r-1)}(\mu)} \cdot \mu(A)^{\frac{1}{r}(1-\frac{1}{p})}\\
        & + O_{y_0}(\exp(m^7))\cdot \mathbb{P}\left(\displaystyle \max_{1\le \ell \le m}|G_\ell - G_1| + \max_{1\le i \le m}|\mathcal{A}(y_0+1)-\mathcal{A}(0)|\ge a\right)^{1/p}\\
        & = O(\exp(m^\theta + y_0^\theta))\cdot\exp\left({\frac{y_0m^2 a^2}{4}\big(\frac{r/(r-1)}{ry_0/(r-1)+1}-\frac{1}{y_0+1}\big)}\right)\\
        & \cdot \displaystyle\sup_{\max_{1\le \ell \le m}|G_\ell - G_1| \le a}\norm{Q^{m, G}}_{L^{2r/(r-1)}(\mu)} \cdot \mu(A)^{\frac{1}{r}(1-\frac{1}{p})}\\
        & + O(\exp(m^\theta + y_0^\theta))\cdot \mathbb{P}\left(\displaystyle \max_{1\le \ell \le m}|G_\ell - G_1| + \max_{1\le i \le m}|\mathcal{A}(y_0+1)-\mathcal{A}(0)|\ge a\right)^{1/p} \,.
\end{align*}

for some $\theta > 0$, where 
\begin{itemize}
\item[1.] $Q^{m, G}$ is the Radon-Nikodym  derivative of 
\begin{equation*}
Y^{m, G}\equiv \displaystyle\max_{1\le \ell \le m}(G_{\ell} -G_1 + B[(0,\ell)\to (\cdot+1,1)])-\displaystyle\max_{1\le \ell \le m}(G_{\ell} -G_1 + B[(0,\ell)\to (1,1)])
\end{equation*}
against rate two Brownian motion on $[0, y_0-1]$
\item[2.] $G$ denotes the boundary data $G = (G_{\ell})_{\ell=1}^m$, $G_\ell = \displaystyle\max_{x\in \mathrm{supp}_{-\infty}(h_0)}(h_0(x)+\mathcal{A}[x\to (0,\ell)])$
    \item[3.]  $\zeta = \frac{3}{\min(1+T_m, U_m-t)}$.
\end{itemize}

\end{theorem}

\begin{proof}
    First, fix $t>y_0$ and condition on the sigma algebra $\mathscr{F}^{[0, t]}_m$.

By the Brownian Gibbs property enjoyed by the Airy line ensemble, we get that conditioning on the sigma algebra $\mathcal{F}^{[0,t]}_m$, the law of $\mathcal{A}$ on $\llbracket 1, k\rrbracket\times[0,t]$ has the law of $m$ independent Brownian bridges with starting points $(\mathcal{A}_i(0))_{i=1}^m$ and ending at $(\mathcal{A}_i(t))_{i=1}^m$ conditioned to not intersect each other and the bottom line $\mathcal{A}_{m+1}$, an event in $\mathscr{C}^m_{*, *}([0, t])$. This conditional law has Radon-Nikodym Derivative against $m$ independent Brownian bridges with starting points $(\mathcal{A}_i(0))_{i=1}^m$ and ending at $(\mathcal{A}_i(t))_{i=1}^m$ 
\begin{equation}
\label{eq: inv acc prob}
\displaystyle\frac{\mathbf{1}_{\mathrm{NoInt}([0,t], \mathcal{A}_{m+1})}(\omega)}{\mathfrak{B}^{[0,t]}_{(\mathcal{A}_i(0))_{i=1}^m, (\mathcal{A}_i(t))_{i=1}^m}(\mathrm{NoInt}([0,t], \mathcal{A}_{m+1}))}
\end{equation}
for paths $\omega$ in $\mathscr{C}^m_{*, *}([0, t])$.

Now, by the by metric composition for LPP, and the $\mathscr{F}^{[0,t]}_m$-measurability of $G_\ell, 1 \le \ell\le m$, we obtain

\begin{equation*}
    \begin{array}{ll}
         & \PP \left( \displaystyle\max_{1\le \ell \le m}(G_\ell + \mathcal{A}[(0,\ell)\to (\cdot,1)])\in A \big| \mathcal{F}^{[0,t]}_m\right) \\
         & =  \mathfrak{B}^{[0,t]}_{(\mathcal{A}_i(0))_{i=1}^m, (\mathcal{A}_i(t))_{i=1}^m} \left(\Bigg\{ \omega \in C^m_{(\mathcal{A}_i(0))_{i=1}^m, (\mathcal{A}_i(t))_{i=1}^m}([0,t]):\displaystyle\max_{1\le \ell \le m}(G_\ell + \omega[(0,\ell)\to (\cdot,1)])\in A \Bigg\}\right)
    \end{array}
\end{equation*}
where
\begin{equation*}
G_\ell \equiv \displaystyle\max_{x\in \mathrm{supp}_{-\infty}(h_0)}(h_0(x)+\mathcal{A}[x\to (0,\ell)])\,.
\end{equation*}

Now, by Lemma \ref{lemma: brownian bridge comparison lemma}, we have that the law of the first $m$ lines of $\mathcal{A}(\cdot)-\mathcal{A}(0)$ on $[0,y_0]$ conditional on $\mathcal{F}^{[0,t]}_m$ is absolutely continuous with respect to the law of $m$ independent rate two Brownian motions on $[0,t]$ with bounded Radon-Nikodym derivative
\begin{equation*}
\frac{\diff\mathfrak{B}^{[0, t]}_{\underline{0}, \underline{\mathcal{A}}}|_{[0,y_0]}}{\diff\mathfrak{B}^{[0, y_0]}_{\underline{0}, *}}
\end{equation*}
against rate two Brownian motion on paths in $\mathscr{C}_{0,*}([0,t-1])^{m}$ with norms
\begin{equation*}
    \begin{array}{cc}
         & \norm{\frac{\diff\mathfrak{B}^{[0, t]}_{\underline{0}, \mathcal{A}}|_{[0,y_0]}}{\diff\mathfrak{B}^{[0, y_0]}_{\underline{0}, *}}}_{L^p\left(\mathfrak{B}^{[0, y_0]}_{\underline{0}, *}\right)} = \frac{(t/(t-y_0))^{\frac{m}{2}}}{(px/(t-y_0)+1)^{\frac{m}{2}}}\cdot \exp\left({\frac{y_0\norm{\mathcal{A}^m(t)-\mathcal{A}^m(0)}^2}{4(t-y_0)}\big(\frac{p}{(p-1)y_0+t}-\frac{1}{t}\big)}\right)
    \end{array}
\end{equation*}
        for all $p>1$ and 
\begin{equation*}
  \begin{array}{cc}
    \norm{\frac{\diff\mathfrak{B}^{[0, t]}_{\underline{0}, \underline{\mathcal{A}}}|_{[0,y_0]}}{\diff\mathfrak{B}^{[0, y_0]}_{\underline{0}, *}}}_{L^\infty\left(\mathfrak{B}^{[0, y_0]}_{\underline{0}, *}\right)} = (t/(t-y_0))^{\frac{m}{2}}\cdot \exp\left(\frac{\norm{\mathcal{A}^m(t)-\mathcal{A}^m(0)}^2}{4t}\right). 
        \end{array}
\end{equation*}

Combining all of the above, we deduce almost surely for any $A\subseteq \mathscr{C}_{0,*}([0,y_0-1])$ Borel measurable that
\begin{align*}
     \PP \bigg(\displaystyle\max_{1\le \ell \le m}(G_\ell &+ \mathcal{A}[(0,\ell)\to (\cdot+1,1)])-\displaystyle\max_{1\le \ell \le m}(G_\ell + \mathcal{A}[(0,\ell)\to (1,1)])\in A \bigg| \mathcal{F}^{[0,t]}_m\bigg)\\
     & \le \displaystyle\frac{\mu^{\mathcal{A}(0), \mathcal{A}(t), G}(A)^{1-1/p}}{\mathfrak{B}^{[0,t]}_{(\mathcal{A}_i(0))_{i=1}^m, (\mathcal{A}_i(t))_{i=1}^m}(\mathrm{NoInt}([0,t], \mathcal{A}_{m+1}))^{(p-1)/p}} \,,
\end{align*}
for all $p>1$, where $\mu^{\mathcal{A}(0), \mathcal{A}(t), G}(\cdot)$ denotes the law of 
\begin{align*}
\mu^{\mathcal{A}(0), \mathcal{A}(t), G}(\cdot) \equiv \PP \bigg( &\displaystyle\max_{1\le \ell \le m}(G_\ell + \mathfrak{B}[(0,\ell)\to (\cdot+1,1)])\\
-&\displaystyle\max_{1\le \ell \le m}(G_\ell + \mathfrak{B}[(0,\ell)\to (1,1)])\in \cdot \bigg)
\end{align*}
where $\mathfrak{B}$ is an ensemble of $m$ independent Brownian bridges with starting and ending points $(0, \mathcal{A})$ and $(t, \mathcal{A}(t))$ respectively.

Thus, by H\"{o}lder inequality, the unconditional probability can be estimated as 
\begin{equation*}
    \begin{array}{ll}
         & \PP \left( \displaystyle\max_{1\le \ell \le m}(G_\ell + \mathcal{A}[(0,\ell)\to (\cdot+1,1)])-\displaystyle\max_{1\le \ell \le m}(G_\ell + \mathcal{A}[(0,\ell)\to (1,1)])\in A \right)\\
         & \le\mathbb{E}\left [\displaystyle\frac{1}{\mathfrak{B}^{[0,t]}_{(\mathcal{A}_i(0))_{i=1}^m, (\mathcal{A}_i(t))_{i=1}^m}(\mathrm{NoInt}([0,t], \mathcal{A}_{m+1}))} \right ]^{(p-1)/p}  \cdot  \mathbb{E}\left[\mu^{\mathcal{A}(0), \mathcal{A}(t), G}(A)^{p-1}\right]^{1/p}
    \end{array}
\end{equation*}
for all $p>1$.

Now, to estimate the first term, we `step outside' of the interval $[0,t]$ and condition on $\mathscr{F}^{[-T_m,U_m]}\subseteq \mathscr{F}^{[0,t]}_m$, for $T_m, U_m$ sufficiently large, to be chosen later. To control the inverse acceptance probability (\ref{eq: inv acc prob}) conditional on $\mathscr{F}^{[-T_m,U_m]}$, we use Lemma \ref{lemma: acceptance prob LB} and the lower bound provided by Lemmas \ref{lemma: acceptance prob LB} and \ref{lemma: inv acceptance prob cond} to obtain
\begin{align}\label{eq: estimate prob}
         \PP ( \displaystyle\max_{1\le \ell \le m}(G_\ell &+ \mathcal{A}[(0,\ell)\to (\cdot+1,1)])-\displaystyle\max_{1\le \ell \le m}(G_\ell + \mathcal{A}[(0,\ell)\to (1,1)])\in A )\nonumber\\
         & \le O(\exp(m^\theta + y_0^\theta))  \cdot  \mathbb{E}\left[\mu^{\mathcal{A}(0), \mathcal{A}(t), G}(A)^{p-1}\right]^{1/p}
    \end{align}
for all $p>1$ and some universal constant $c>0$.

To estimate the second term in (\ref{eq: estimate prob}), let $Q^{m, G}$ be the Radon-Nikodym  derivative of 
\begin{equation*}
Y^{m, G}\equiv \displaystyle\max_{1\le \ell \le m}(G_{\ell} -G_1 + B[(0,\ell)\to (\cdot+1,1)])-\displaystyle\max_{1\le \ell \le m}(G_{\ell} -G_1 + B[(0,\ell)\to (1,1)])
\end{equation*}
against rate two Brownian motion on $[0, y_0-1]$ (here we treat the initial data $G$ as fixed in $\R^m_>$). Note that by \cite[Theorem~58]{dellacherie2011probabilities}, we can take $Q^{m, G}$ to be jointly measurable in $\tilde{G}$ and paths $\xi$ in Wiener space on $[0,y_0-1]$. Now, by \cite[Theorem~4.3]{sarkar2021brownian} $Y^{m, G}$ can be expressed as the top line of a sequence of upwardly reflected Brownian motions with boundary data $G_{\ell} -G_1$ $1\le \ell\le m$, hence its Radon-Nikodym  derivative against Brownian motion can be estimated from Theorem~\ref{thm: main companion}, and in particular, $Q^{m, G}\in L^{\infty-}(\mu)$ for all choices of boundary data $G$, where $\mu$ is the restriction of the (rate two) Wiener measure on $[0,y_0-1]$.

Combining all of the above, we deduce the following norm estimates for the Radon-Nikodym derivatives $Q^{\underline{x}, \underline{y}, G}$ of $\mu^{x, y, G}(\cdot)$ for all data $\underline{x}, \underline{y}, G\in \R^m_{>}$
\begin{equation*}
    \begin{array}{cc}
         \norm{Q^{\underline{x}, \underline{y}, G}}_{L^{p}(\mu)}& \le \displaystyle\frac{(t/(t-y_0))^{\frac{m}{2}}}{(py_0/(t-y_0)+1)^{\frac{m}{2}}}\cdot \exp\left({\frac{y_0\norm{\underline{x-y}}^2}{4(t-y_0)}\big(\frac{p}{(p-1)y_0+t}-\frac{1}{t}\big)}\right)\\
         & \cdot \norm{Q^{m, G}}_{L^{2p}(\mu)}\cdot\mu(A)^{1-\frac{1}{p}}
    \end{array}
\end{equation*}
for all $p>1$. 

For the second term in (\ref{eq: estimate prob}), we estimate with $t = y_0+1$ for all $a>0$ by H\"{o}lder's inequality
    \begin{equation*}
        \begin{array}{ll}
             &\mathbb{E}[\mu^{\mathcal{A}(0), \mathcal{A}(t), G}(A)^{p-1}]^{1/p}\\
             &\le \mathbb{E}\left[\mu^{\mathcal{A}(0), \mathcal{A}(t), G}(A)^{\frac{p-1}{p}}\mathbf{1}\left(\displaystyle \max_{1\le \ell \le m}|G_\ell - G_1| + \max_{1\le i \le m}|\mathcal{A}(y_0+1)-\mathcal{A}(0)|<a\right)\right]^{1/p}\\
             & + \mathbb{P}\left(\displaystyle \max_{1\le \ell \le m}|G_\ell - G_1| + \max_{1\le i \le m}|\mathcal{A}(y_0+1)-\mathcal{A}(0)|\ge a\right)^{1/p}\\
        \end{array}
    \end{equation*}

    Now, for $r\in (1,\infty)$, combining the two estimates above, we obtain
\begin{align*}
     &\mathbb{E}\left[\mu^{\mathcal{A}(0), \mathcal{A}(t), G}(A)^{p-1}\right]^{1/p} \\
    &\le \mathbb{E}\left[\exp\left({\frac{py_0\norm{\underline{\mathcal{A}(y_0+1)-\mathcal{A}(0)}}^2}{4}\big(\frac{r/(r-1)}{ry_0/(r-1)+1}-\frac{1}{y_0+1}\big)}\right)\cdot \norm{Q^{m, G}}^{p}_{L^{2r/(r-1)}(\mu)}\right.\\
    & \left.\cdot \mathbf{1}\left(\displaystyle \max_{1\le \ell \le m}|G_\ell - G_1| + \max_{1\le i \le m}|\mathcal{A}(y_0+1)-\mathcal{A}(0)|<a\right)\right]^{1/p}\cdot\mu(A)^{\frac{1}{r}(1-\frac{1}{p})}\\
    & + \mathbb{P}\left(\displaystyle \max_{1\le \ell \le m}|G_\ell - G_1| + \max_{1\le i \le m}|\mathcal{A}(y_0+1)-\mathcal{A}(0)|\ge a\right)^{1/p}\\
    &\le \exp\left({\frac{y_0m^2 a^a}{4}\big(\frac{r/(r-1)}{ry_0/(r-1)+1}-\frac{1}{y_0+1}\big)}\right)\\
    &\cdot \displaystyle\sup_{\max_{1\le \ell \le m}|G_\ell - G_1| \le a}\norm{Q^{m, G}}_{L^{2r/(r-1)}(\mu)} \cdot \mu(A)^{\frac{1}{r}(1-\frac{1}{p})}\\
     & + \mathbb{P}\left(\displaystyle \max_{1\le \ell \le m}|G_\ell - G_1| + \max_{1\le i \le m}|\mathcal{A}(y_0+1)-\mathcal{A}(0)|\ge a\right)^{1/p}\,,\\
\end{align*}
    which when combined with (\ref{eq: estimate prob}), concludes the proof of the second part.
\end{proof}

\section{Putting it all together: quantitative Brownian regularity}\label{sec: brownian regularity combined}

In this section, we establish the quantitative Brownian regularity of the KPZ fixed point started from arbitrary (finitary) initial data, Theorem~\ref{thm: KPZ reg finitary}.
 
First, recall the definition of the semi-infinite last passage values from (\ref{eq: Airy limit}). To summarise what we have obtained so far, recall that having established the quantitative comparison in Theorem~\ref{thm: finite depth KPZ estimates a priori}, we have estimated for $m\ge 1$, the truncated, finite-depth KPZ fixed point
\begin{equation}\label{eq: finite depth trun}
H_m(\cdot) = \displaystyle\max_{1 \le \ell \le m}(G_\ell + \mathcal{A}[(0,\ell)\to (\cdot ,1)]), \qquad y\in [1,y_0]
\end{equation}
in terms of 
\begin{itemize} 
\item[1.] the boundary data $G = (G_{\ell})_{\ell=1}^m$, $G_\ell = \displaystyle\max_{x\in \mathrm{supp}_{-\infty}(h_0)}(h_0(x)+\mathcal{A}[x\to (0,\ell)])$
\item[2.] $Q^{m, \tilde{G}^a}$, the Radon-Nikodym  derivatives of 
\begin{equation*}
Y^{m, \tilde{G}^a}\equiv \displaystyle\max_{1\le \ell \le m}(\tilde{G}^a_{\ell}  + B[(0,\ell)\to (\cdot+1,1)])-\displaystyle\max_{1\le \ell \le m}(\tilde{G}^a_{\ell} + B[(0,\ell)\to (1,1)])
\end{equation*}
against rate two Brownian motion on $[0, y_0-1]$, where $\tilde{G}^a = (-a\land (G_{\ell} -G_1) \lor a)_{\ell=1}^m$, for some $a > 0$
\item[3.] and the tails of $\max_{1\le \ell \le m} |G_\ell - G_1|$.
\end{itemize}

Now, Theorem~\ref{thm: main companion} allows us to estimate $L^p(\mu)$-norms of $Y^{m, \tilde{G}^a}$ for all $a>0$, $p>1$. Thus, the only missing ingredient is to estimate the tails of 
\begin{equation*}
G_\ell - G_1 \equiv \displaystyle\max_{x\in \mathrm{supp}_{-\infty}(h_0)}(h_0(x)+\mathcal{A}[x\to (0,\ell)]) - \displaystyle\max_{x\in \mathrm{supp}_{-\infty}(h_0)}(h_0(x)+\mathcal{A}[x\to (0,1)])\,.
\end{equation*}

This can be done uniformly in the initial data in terms of differences between levels in the Airy line ensemble. This follows from the interplay between the non-intersecting nature of the Airy line ensemble and the pathwise construction of the finite depth truncations of the KPZ fixed point \eqref{eq: finite depth trun} in terms of iterated Pitman transforms/Skorokhod reflections. We record the following lemma (which we state in slightly more generality), which captures this `cutoff'.

\begin{lemma}\label{lemma: bdry data trunc}
    Consider the processes
    \begin{equation*}
    H_m(y) = \max_{\ell\leq m}(G_{\ell}+\mathcal{A}[(0,\ell)\to (y,1)])\,, \qquad y\ge 0\,, m\ge 1
    \end{equation*}
    where the `boundary data' $G_\ell$ are almost surely finite and 
    \begin{equation*}
    \mathcal{F}_{-}\equiv \sigma(\{\mathcal{A}_i(x):x\leq 0,i=1,2,\cdots\})\end{equation*}
    -measurable. 
    
    Then, one can estimate for all $A$ Borel subsets on paths $\mathscr{C}_{0,*}(I; \R)$ for $I\subseteq (0,\infty)$ compact, 
    \begin{equation*}
    \PP(H_m(\cdot + \inf I)-H_m (\inf I)\in A) \le c\mu(A)\end{equation*}
    \begin{equation*}
    + \displaystyle\sum_{m^* = 1}^{m-1}\PP(\max_{1\le \ell\leq m^*}(G'_{\ell}+\mathcal{A}[(0,\ell)\to (\cdot + \inf I ,1)])-\max_{1\le \ell\leq m^*}(G'_{\ell}+\mathcal{A}[(0,\ell)\to (\inf I ,1)])\in A)\,,
    \end{equation*}
    for some $c>0$, where 
    \begin{equation*}
    0\le G'_{\ell} = m \cdot\left(|G_\ell - G_m|\land |\mathcal{A}_1(0) - \mathcal{A}_m (0)|\right)\,, \quad 1 \le y \le m^*.
    \end{equation*}
\end{lemma}
\begin{proof}
    Fix $m\ge 1$. Since the Airy lines are non-intersecting, if for some $1\le i \le m-1$
    \begin{equation*}
    |G_i-G_{i+1}|\ge |\mathcal{A}_i(0)-\mathcal{A}_m(0)|\,,
    \end{equation*}
    one has
    \begin{equation*}
    H_m(y) = \max_{1\le \ell\leq i}(G_{\ell}+\mathcal{A}[(0,\ell)\to (y,1)])\,, y\ge 0\,.
    \end{equation*}
    This is because by \cite[Theorem~4.3]{sarkar2021brownian} one can obtain $H_m$ as the upward reflection (top line of the Pitman transform of)
    \begin{equation*}
    G_1 + \mathcal{A}_1(y)\,,y\ge 0 
    \end{equation*}
    against
    \begin{equation*}
     \max_{2\le \ell\leq m}(G_{\ell}+\mathcal{A}[(0,\ell)\to (y,2)])\,, y\ge 0\,.
    \end{equation*}
    Thus, by an easy induction we can express (see Figure \ref{fig: pitman gap})
    \begin{equation*}
    H_m(y) = \max_{1\le \ell\leq m^*}(G_{\ell}+\mathcal{A}[(0,\ell)\to (y,1)])\,,y\ge 0\,,
    \end{equation*}
    where
    \begin{equation*}
    m^* = \inf\{1\le i\le m-1 : |G_i-G_{i+1}| > |\mathcal{A}_i(0)-\mathcal{A}_m(0)|\}\,.
    \end{equation*}
    \begin{figure}
        \centering
        \includegraphics[width=0.7\linewidth]{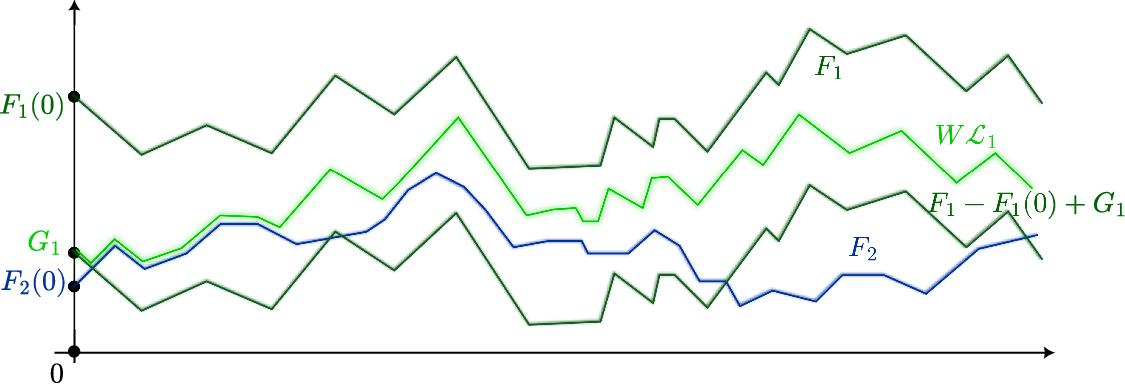}
        \caption{Cartoon illustration of top lines of Pitman transforms of two environments $F = (F_1, F_2)$ and $\mathcal{L} = (F_1-F_1(0)+G_1, F_2)$ for $F_2(0) < G_1$. Notice, since $F_1 > F_2$ pointwise, if $G_2 > F_1(0)$, then $WF_1 = \mathcal{L}$.}
        \label{fig: pitman gap}
    \end{figure}
    Now, if the above set is empty, it means $|G_1-G_{2}|> |\mathcal{A}_1(0)-\mathcal{A}_m(0)|$ which gives 
    \begin{equation*}
    H_m(y) = G_1 + \mathcal{A}_1(y)\,,y\ge 0 \,.
    \end{equation*}
    Thus, we can estimate for any $A$ Borel in $\mathscr{C}_{0,*}([0, \sup I- \inf I])$ for the increment processes,
    \begin{align*}
        \PP(H_m(\cdot + \inf I)-H_m (\inf I)\in A)&\le \PP(m^* = -\infty, H_m(\cdot + \inf I)-H_m (\inf I)\in A)\\
        &+ \PP\left(\max_{1\le \ell\leq m^*}(G_{\ell}+\mathcal{A}[(0,\ell)\to (y,1)])\right)\\
        &\le \PP(\mathcal{A}_1(\cdot + \inf I)-\mathcal{A}_1(\inf I)\in A)\\
        &+ \PP\bigg(\max_{1\le \ell\leq m^*}(G_{\ell}-G_m+\mathcal{A}[(0,\ell)\to (\cdot + \inf I ,1)])\\
        &-\max_{1\le \ell\leq m^*}(G_{\ell}-G_m+\mathcal{A}[(0,\ell)\to (\inf I ,1)])\in A\bigg)\\
        &\le c\mu(A)+ \displaystyle\sum_{m^* = 1}^{m-1}\PP\bigg(\max_{1\le \ell\leq m*}(G'_{\ell}+\mathcal{A}[(0,\ell)\to (\cdot + \inf I ,1)])\\
        &-\max_{1\le \ell\leq m^*}(G'_{\ell}+\mathcal{A}[(0,\ell)\to (\inf I ,1)])\in A\bigg)\,,
    \end{align*}
    for some $c>0$, where 
    \begin{equation*}
    0\le G'_{\ell} = (G_\ell - G_m)\land m \cdot|\mathcal{A}_1(0) - \mathcal{A}_m (0)|\,,
    \end{equation*}
    using the triangle inequality in the last step and the fact that the $\mathcal{A}_1$ has the law of the Airy$_2$ process, which has a bounded density against Brownian motion, by \cite[Theorem~1.1]{dauvergne2024wienerdensitiesairyline}.
\end{proof}

Thus, from now on, we will implicitly assume the boundary data $G_1 - G_\ell$ are bounded above by 
\[
m \cdot |\mathcal{A}_1(0) - \mathcal{A}_m (0)|\,.
\]
This will allow us to obtain stretched exponential tails for these boundary value differences uniformly in the initial data. Indeed, by \cite[Corollary 5.3]{dauvergne2021}, there exists an exponent $\eta > 1$ such that for all $a > 1, m\ge 1$,
\begin{equation}\label{eq: airy one point diff sub exp}
    \PP(|\mathcal{A}_1(0) - \mathcal{A}_m (0)| \ge a\cdot m^\eta)\le c\mathrm{e}^{-da}
\end{equation}
for universal $c, d> 0$.

Now we are in a position to state the following concentration result for differences in boundary values $G_1-G_\ell$ in the following lemma. 

\begin{lemma}\label{lemma: boundary data control}
    For $\ell\in \N$, $\alpha > 0,$ with $G_\ell \equiv \displaystyle\max_{x\in \mathrm{supp}_{-\infty}(h_0)}(h_0(x)+\mathcal{A}[x\to (0,\ell)])$. Suppose further that the max-plus support of $h_0$ is countable. Then, there exists some universal exponent $\delta > 0$ such that for all $m\ge 1$,
    \begin{equation*}\mathbb{P}\bigg(\displaystyle\max_{1\le \ell\le m}|G_\ell-G_1|\ge m \bigg)\lesssim \mathrm{e}^{-dm^\delta}\,,
    \end{equation*}
    for some universal $d > 0$.
\end{lemma}
\begin{proof}
This is a direct application of \eqref{eq: airy one point diff sub exp} and a union bound in conjunction with the bound from Lemma \ref{lemma: bdry data trunc}.
\end{proof}

\color{black}
Combining the above, we are now in a position to provide a quantitative Brownian comparison for the spatial increments of finite depth truncations of the KPZ fixed point.

\begin{theorem}\label{thm: finite depth KPZ bounds}
    Let $m\in \N, y_0 \ge 1, \alpha > 0$ and let the finite depth truncations $H_m(\cdot)$ be as in (\ref{eq: finite depth trun}) with continuous and bounded initial data $h_0$. Then, one has for all $p>1$, $r>1$ and $A$ Borel the estimates
\begin{equation*}
    \begin{array}{cc}
        &\PP \left( H_m(\cdot + 1)-H_m(1)\in A \right)\le O_{y_0, K, p,r}\left(\exp(m^\eta)\cdot \mu(A)^{\frac{1}{r}(1-\frac{1}{p})}+ \mathrm{e}^{-d_{y_0, p} m}\right)\,,
    \end{array}
\end{equation*}
where $\eta > 1$, $d_{y_0, p}>0$ are universal constants.
\end{theorem}
\begin{proof}
    Theorem~\ref{thm: finite depth KPZ estimates a priori} gives that, with $\mu$ the rate two Wiener measure on $[0,y_0-1]$, $H_m(\cdot+1)-H_m(1)$ satisfies the norm estimates
the following holds for all $p, r> 1$, $A\subseteq \mathscr{C}_{*, *}([0,y_0-1])$ Borel and $a> 0$
\begin{equation*}
    \begin{array}{ll}
        \PP ( H_m(\cdot + 1)-H_m(1)\in A ) &\le O_{y_0}(\exp(m^7))\cdot\bigg( \mu(A)^{\frac{1}{r}(1-\frac{1}{p})}\cdot\exp\left({\frac{y_0m^2 a^2}{4}\big(\frac{r/(r-1)}{ry_0/(r-1)+1}-\frac{1}{y_0+1}\big)}\right)\\
        &\cdot \displaystyle\sup_{\max_{1\le \ell \le m}|G_\ell - G_1| \le a}\norm{Q^{m, G}}_{L^{2r/(r-1)}(\mu)}  \\
        & + \mathbb{P}\left(\displaystyle \max_{1\le \ell \le m}|G_\ell - G_1| + \max_{1\le i \le m}|\mathcal{A}(y_0+1)-\mathcal{A}(0)|\ge a\right)^{1/p}\bigg) \,,
    \end{array}
\end{equation*}
for all $p,r>1$, $\varepsilon \in(0,1)$ and some universal constant $c>0$, where 
\begin{itemize} 
\item[1.] $G$ denotes the boundary data $G = (G_{\ell})_{\ell=1}^m$,
$G_\ell = \displaystyle\max_{x\in \mathrm{supp}_{-\infty}(h_0)}(h_0(x)+\mathcal{A}[x\to (0,\ell)])$.
\item[2.] $Q^{m, G}$ is the Radon-Nikodym  derivative of 
\begin{equation*}
Y^{m, G}\equiv \displaystyle\max_{1\le \ell \le m}(G_{\ell} -G_1 + B[(0,\ell)\to (\cdot+1,1)])-\displaystyle\max_{1\le \ell \le m}(G_{\ell} -G_1 + B[(0,\ell)\to (1,1)])
\end{equation*}
against a rate two Brownian motion on $[0, y_0-1]$\,.  
\end{itemize}
Now, a union, Lemmas \ref{lemma: bdry data trunc} and \ref{lemma: boundary data control} give for all $p>1$, $r>1$ and $A$ Borel, the estimates
\begin{equation*}
    \begin{array}{ll}
        &\PP \left( H_m(\cdot + 1)-H_m(1)\in A \right) \le O_{y_0, K, h_0}\exp\left(c_{y_0, K, h_0} m^\eta\right)\\
        & \cdot  \displaystyle\sup_{\displaystyle\max_{1\le \ell \le m}|G_\ell - G_1| \le m^\eta}\norm{Q^{m, G}}_{L^{4}(\mu)} \cdot \mu(A)^{\frac{1}{r}(1-\frac{1}{p})}+ O_{y_0, K}(\mathrm{e}^{-d_{y_0, p} m})\,,
    \end{array}
\end{equation*}
for some $d_{y_0, p}, c_{y_0, K, h_0} >0$ and $\eta > 1$. Now, using the control on the Radon-Nikodym  derivative of upward reflections of Brownian motion from Theorem~\ref{thm: main companion} we obtain, 
 \begin{equation*}
     \begin{array}{cc}
         \displaystyle\sup_{\displaystyle\max_{1\le \ell \le m}|G_\ell - G_1| \le m^\eta}\norm{Q^{m, G}}_{L^{2r/(r-1)}(\mu)} &\le O_{y_0, r}m^{\eta m^2} \mathrm{e}^{d_r m^2\log m + c_{y_0, r}m^{2\eta + 1}}\,.
    \end{array}
\end{equation*}

We thus have for all $p>1$, $r>1$ and $A$ Borel the estimates
\begin{equation*}
    \begin{array}{ll}
        &\PP \left( H_m(\cdot + 1)-H_m(1)\in A \right)\\
        & \le O_{y_0, K, h_0, r}\exp\left(c_{y_0, K, r} m^{2\eta}\right) \cdot \mu(A)^{\frac{1}{r}(1-\frac{1}{p})}+ O_{y_0, K, p}(\mathrm{e}^{-d_{y_0, p} m})\,,
    \end{array}
\end{equation*}
for some $d_{y_0, p}> 0$. Now, $c_{y_0, K, r}$ can be absorbed into the $O_{y_0, K, h_0, r}$ term by an application of Young's inequality concluding the proof.
\end{proof}

Having established the quantitative Brownian regularity for spatial increments of finite depth truncations of the KPZ fixed point in the previous Theorem, we now translate this to a quantitative comparison of spatial increments of the actual KPZ fixed point at unit time. This comparison is expressed in terms of the geodesic intercepts $L_0$ as defined in Theorem~\ref{thm: intercept tail bound}, which is the final step before obtaining the main result. Note that the comparison is uniform over a wide class of initial data. 

\begin{corollary}\label{cor: KPZ log comparison BM}
 Fix $y_0>1$ and let $\mathfrak{h}(\cdot)\equiv \mathfrak{h}_1(\cdot)$ be the KPZ fixed point at unit time on $[1,y_0]$ as defined in (\ref{eq: KPZ fixed point}). Then for all initial data with countable `max-plus' supports contained in some fixed compact set $K\subset \R$, and $A$ Borel on has the estimates 
\begin{equation*}
\begin{array}{cc}
     \PP(\mathfrak{h}(\cdot + 1)-\mathfrak{h}(1)\in A) \le O_{y_0, K, p,s} \left(\mu(A)^{\frac{\varepsilon}{s}(1-\frac{1}{p})}+ \mathrm{e}^{-d_{y_0, p}m^*} +  \PP(L_0\geq m^*)\right)\,,
\end{array}
\end{equation*}
for any $p>1$, $s>1$  $\varepsilon \in (0,1)$ and some $d_{y_0, p}> 0$, where
\begin{equation*}
m^* =  \displaystyle \sup \left\{m\in \mathbb{N} :  \left(m^\eta \le \log \left(\frac{1}{ \mu(A)^{\varepsilon\theta}}\right)\right)\right\}<\infty\,,
\end{equation*}
for some universal $\eta > 0$.
\end{corollary}

\begin{proof}[Proof of Corollary \ref{cor: KPZ log comparison BM}]
From Theorem~\ref{thm: finite depth KPZ bounds}, we have the estimates for all $p>1$, $s>1$ and $A$ Borel
\begin{equation*}
    \PP \left( H_m(\cdot + 1)-H_m(1)\in A \right)\le O_{y_0, K, p,s}\left(\exp(m^\eta)\cdot \mu(A)^{\frac{1}{r}(1-\frac{1}{p})}+ \mathrm{e}^{-d_{y_0, p} m}\right)
\end{equation*}
for some $d_{y_0,p}, \eta > 0$.
One thus estimates for all $A$ Borel measurable 
\begin{equation*}
\begin{array}{ll}
     \PP(\mathfrak{h}(\cdot + 1)-\mathfrak{h}(1)\in A) &= \displaystyle\inf_{m\in \mathbb{N}}\PP(\mathfrak{h}(\cdot + 1)-\mathfrak{h}(1)\in A, L_0\le m) + \PP(L_0\geq m+1) \\
     & \le \displaystyle\inf_{m\in \mathbb{N}}\PP(H_m(\cdot+1)-H_m(1)\in A, L_0\le m) + \PP(L_0\geq m)\\
     & \le  \displaystyle\inf_{m\in \mathbb{N}} O_{y_0, K, p,s}\left(\exp(m^\eta)\cdot \mu(A)^{\frac{1}{r}(1-\frac{1}{p})}+ \mathrm{e}^{-d_{y_0, p} m}+\PP(L_0 \ge m)\right)\,,
\end{array}
\end{equation*}
Now, with $\theta = \frac{1}{s}(1-1/p)$ and any $\varepsilon \in (0,1)$, let 
\begin{equation*}
m^* =  \displaystyle \sup \left\{m\in \mathbb{N} :  \left(m^\eta \le \log \left(\frac{1}{ \mu(A)^{\varepsilon\theta}}\right)\right)\right\}<\infty\,.
\end{equation*}
One further estimates for all $A$ Borel
\begin{equation*}
\begin{array}{cc}
     \PP(\mathfrak{h}(\cdot + 1)-\mathfrak{h}(1)\in A) \le O_{y_0, K, p,s}\left(\mu(A)^{\frac{\varepsilon}{s}(1-\frac{1}{p})} + \mathrm{e}^{-d_{y_0, p}m^*} +  \PP(L_0\geq m^{*})\right)\,,
\end{array}
\end{equation*}
for some $d > 0$, concluding the proof. 
\end{proof}

Finally, using the tail bounds on $L_0$ established in Theorem \ref{thm: intercept tail bound}, we establish using Corollary \ref{cor: KPZ log comparison BM} the uniform quantitative Brownian regularity of spatial increments of the KPZ fixed points started from initial data with bounded and countable `max-plus' supports. The uniformity is with respect to a suitable class of initial data. We take the `max-plus' support of the initial data to be countable for measurability reasons (to ensure the boundary data in \eqref{eq: bdry data} are random variables and use the Airy sheet coupling \ref{def: Airy sheet}), though this can always be guaranteed for upper semi-continuous initial data (by picking a suitable countable dense subset of the `max-plus' support of the initial data).

\begin{theorem}\label{thm: KPZ law local uniform comparison}
     Let $\mathfrak{h}_t(\cdot), t\ge 0$ be the KPZ fixed point as defined in \eqref{eq: KPZ fixed point}. Then, fixing $t>0$ and any $\ell<r$ both bounded, with $|\ell|+ |r|\le y_0$ for some $y_0>0$, one obtains the estimates for all  $p>1$,  $s>1$, $A$ Borel measurable $A\subseteq \mathscr{C}_{0,*}([0,r-\ell])$ with $\mu(A) > 0$
\begin{equation*}
\begin{array}{cc}
     \PP(\mathfrak{h}(\cdot + 1)-\mathfrak{h}(1)\in A) \le O_{y_0, K, p,s}\left(\mu(A)^{\frac{\varepsilon}{s}(1-\frac{1}{p})} + \exp\bigg({-d_p y_0^\theta  \log^{\eta} \left(1/\mu(A)\right)}\bigg)\right)\,,
\end{array}
\end{equation*} 
for some $\eta> 0$, $d_p>0$ uniformly in initial data in the class
\begin{align*}
 h_0\in \mathcal{F}_{K}\equiv\{&h_0: \R \to\mathbb{R}\cup \{-\infty\}:\,\mathrm{supp}_{-\infty}(h_0)\subseteq K \text{ and is countable}\}\,.
\end{align*}
\end{theorem}

\begin{proof}[Proof of Theorem~\ref{thm: KPZ law local uniform comparison}]
By the $1:2:3$ scaling invariance of the directed landscape, we can without loss of generality assume that $t = 
\ell = 1$. For ease of notation, let $\mathfrak{h}(\cdot)\equiv \mathfrak{h}_1(\cdot)$ denote the KPZ fixed point at unit time.

Recall the notation from Corollary \ref{cor: KPZ log comparison BM}, we now re-express
\begin{equation*}
m^* =  \displaystyle \sup \left\{m\in \mathbb{N} :  \left(m^\eta \le \log \left(\frac{1}{ \mu(A)^{\varepsilon\theta}}\right)\right)\right\}=\bigg\lfloor \log^{1/\eta} \left(\frac{1}{ \mu(A)^{\varepsilon\theta}}\right)\bigg\rfloor\,.
\end{equation*}

Moreover, Corollary \ref{cor: KPZ log comparison BM} and Theorem \ref{thm: intercept tail bound} (also Proposition \ref{prop: global modulus airy} for the functional form in terms of $y_0$ of the coefficients in the exponent) give for all $A$ Borel the estimates
\begin{equation*}
\begin{array}{cc}
     \PP(\mathfrak{h}(\cdot + 1)-\mathfrak{h}(1)\in A) \le O_{y_0, K, p,s}\left(\mu(A)^{\frac{\varepsilon}{s}(1-\frac{1}{p})} + \mathrm{e}^{-d_p y_0^\theta(m^*)^\delta}\right)\,,
\end{array}
\end{equation*}
for some $\theta, \delta, d_p > 0$. We can thus estimate further, (for all $\mu(A)$ sufficiently small)
\begin{equation*}
\begin{array}{cc}
     \PP(\mathfrak{h}(\cdot + 1)-\mathfrak{h}(1)\in A) \le O_{y_0, K, p,s}\left(\mu(A)^{\frac{\varepsilon}{s}(1-\frac{1}{p})} + \exp\bigg({-d_p y_0^\theta  \log^{\eta} \left(1/\mu(A)\right)}\bigg)\right)\,,
\end{array}
\end{equation*} 
for some $\eta> 0$, concluding the proof.
\end{proof}

We are now in a position to prove the main result of this paper, the extension of quantitative Brownian regularity to all finitary initial data. In brief, this will be achieved through a localisation argument using global shape estimates enjoyed by the directed landscape, allowing us to control the support of the initial data `seen' by the KPZ fixed point on compacts with high probability.

\begin{theorem}\label{thm: KPZ reg finitary}
     Let $\mathfrak{h}_t(\cdot), t\ge 0$ be the KPZ fixed point as defined in \eqref{eq: KPZ fixed point}  where $h_0$ is $t$-finitary. Then, for any fixed $\ell<r$ with $|\ell|+ |r|\le y_0$ for some $y_0>0$, there exists some universal $\eta > 0$ such that the estimates for all $A$ Borel measurable $A\subseteq \mathscr{C}_{0,*}([0,r-\ell])$ with rate two Wiener measure $\mu(A) > 0$
\begin{equation}\label{eq: brown reg fin rate fn}
\begin{array}{cc}
     &\PP(\mathfrak{h}_t(\cdot + \ell)-\mathfrak{h}_t(\ell)\in A)\le c_{t, y_0, h_0}\exp\bigg({-d_{t, y_0, h_0}\log^{\eta} \left(1/\mu(A)\right)}\bigg)\,,
\end{array}
\end{equation} 
for some universal $\eta> 0$ and constants $c_{t, y_0, h_0}, d_{t, y_0, h_0}>0$.
\end{theorem}

\begin{proof}[Proof of Theorem~\ref{thm: KPZ law local uniform comparison}]
By the $1:2:3$ scaling invariance of the directed landscape, we can without loss of generality assume that $t = 2$. Using the metric composition law, we can now express the KPZ fixed point on $[\ell, r]$
\begin{equation*}
\mathfrak{h}_2(y) = \displaystyle\max_{x\in \R}(\mathfrak{h}(x)+\mathcal{S}(x,y))\,,\qquad y\in [\ell, r]\,.
\end{equation*}
where $\mathfrak{h}$ denotes the \emph{random} initial data 
\begin{equation*}
\mathfrak{h}(y) = \displaystyle\max_{x\in \R}(h(x)+\mathcal{S}'(x,y))
\end{equation*}
where $\mathcal{S}'(\cdot,\cdot)$ is an Airy sheet independent of $\mathcal{S}(\cdot, \cdot)$. Recall that, being $2$-finitary, $h_0$ satisfies
\begin{equation*}
\displaystyle\lim_{|x|\to \infty}\frac{h_0(x)-x^2/2}{|x|} = -\infty\,.
\end{equation*}
Thus, there exists some $x_0>0$ deterministic dependent on $h_0$ such that
\begin{equation*}
h_0(x)-x^2/2 \le -(y_0+1)|x|\,, \qquad \text{ for all } |x|\ge x_0\,.
\end{equation*}
Henceforth, we will treat $x_0$ as fixed (only depending on our domain of comparison, which is fixed). Additionally, from \cite{dauvergne2022three} the Airy sheets satisfy almost sure pointwise bounds 
\begin{equation*}
|\mathcal{S}(x,y)+(x-y)^2|\le \mathfrak{C}+c\log^{2/3}(2+|x|+|y|)\,,\qquad \text{ for all } x,y\in \R 
\end{equation*}
and 
\begin{equation*}
|\mathcal{S}'(x,y)+(x-y)^2|\le \mathfrak{C}'+c\log^{2/3}(2+|x|+|y|)\,,\qquad \text{ for all } x,y\in \R 
\end{equation*}
for some universal constant $c>0$ and some $\mathfrak{C}, \mathfrak{C}'$ independent (identically distributed) both satisfying $\mathbb{E}[a^{\mathfrak{C}^{3/2}}+a^{\mathfrak{C}'^{3/2}}]<\infty$ for some $a>1$. By rescaling one obtains analogous estimates for $\mathcal{L}(0,\cdot \,,;\,, \cdot, t)$ for any $t>0$ fixed.

Observe that for $|x|\ge x_0$, $y\in [-y_0,y_0]$
\begin{equation*}
h_0(x) + \mathcal{L}(x,0;y,2)\le -c|x|-y^2/2+c\log(1+|y|)
\end{equation*}
for some $c>0$. Now, (assuming without loss of generality that $h_0$ is supported at the origin)
\begin{equation*}
h_0(x) + \mathcal{L}(x,0;y,2)\le h_0(0)+\mathcal{L}(0,0;y,2)
\end{equation*}
if $|x|\ge \mathfrak{C}''$
for some $\mathfrak{C}''>0$ satisfying $\mathbb{E}[a^{\mathfrak{C}''^{d}}]<\infty$ for some $a>1, d>0$. Thus, there is some random $N\in \N$ satisfying $\mathbb{E}[C_{y_0, h_0}^{N^{d}}]<\infty$ for some $C_{y_0, h_0}>1$ and $d>0$ such that almost surely
\begin{equation*}
\mathfrak{h}_2(y) = \displaystyle\max_{x\in [-N, N]}(h_0(x) + \mathcal{L}(x,0;y,2))\,,\qquad \text{ for all } y\in [-y_0, y_0]\,.
\end{equation*}

Now, we have that $\mathfrak{h}(\cdot)$ satisfies the following almost sure estimates for $y\in\R$
\begin{equation*}
\mathfrak{h}(y)\le \displaystyle\max_{x\in [-N,N]}(h(x)-x^2+2xy+\mathfrak{C}'+c\log^{2/3}(2+|x|+|y|))-y^2
\end{equation*}
\begin{equation*}
\le c_{h_0} + \mathfrak{C}' + 2|y|N + \displaystyle\max_{x\in [-N,N]}(-|x|+c\log(1+|x|))-y^2+c\log(1+|y|)
\end{equation*}
\begin{equation*}
\le c_{h_0} + \mathfrak{C}' -y^2+c|y|N
\end{equation*}
for some constants $c, c_{h_0}>0$ (changing from line to line). Arguing as before, we have that for $y\in [-y_0, y_0]$
\begin{equation*}
\mathfrak{h}(x) + \mathcal{S}(x,y)\le \mathfrak{h}(0)+\mathcal{S}(0,y)
\end{equation*}
for all $|x|\ge N'$
for some $N'$ satisfying $\mathbb{E}[C_{y_0, h_0}^{N'^{d}}]<\infty$ for some $C_{y_0, h_0}>1$ and $d>0$. 

Summarising, we have that
\begin{align}\label{eq: KPZ fixed point loc}
    \begin{cases}
        & \mathfrak{h}_2(y) = \displaystyle\max_{x\in [-N', N']}(\mathfrak{h}(x) + \mathcal{S}(x,y))\,,\qquad \text{ for all } y\in [-y_0, y_0]\,,\\
        & \mathfrak{h}(x) = \displaystyle\max_{z\in [-N, N]}(h_0(z) + \mathcal{S}'(z,x))\,,\qquad \text{ for all } z\in \R\,,
    \end{cases}
\end{align}
for some $N, N'$ satisfying $\mathbb{E}[C_{y_0, h_0}^{N'^{d}}+C_{y_0, h_0}^{N^{d}}]<\infty$ for some $C_{y_0, h_0}>1, d>0$ and independent Airy sheets $\mathcal{S}, \mathcal{S}'$. Notice that one can use \emph{any} $M\ge N, N'$ in their place. Furthermore, note that $\mathfrak{h}$ is a continuous, real-valued function on $\R$.

Now, fix some $A\subseteq \mathscr{C}_{0,*}([\ell, r];\R)$ Borel measurable. We can thus estimate for all $n\ge 1$
\begin{equation*}
\PP(\mathfrak{h}_2(\cdot)-\mathfrak{h}_2(\ell)\in A)\le \PP(\mathfrak{h}_2(\cdot)-\mathfrak{h}_2(\ell)\in A, N+N'\le n)+\PP(N+N'\ge n)\,,
\end{equation*}
effectively localising the support of the initial data in the first term. We can now re-express
\begin{equation*}
\PP(\mathfrak{h}_2(\cdot)-\mathfrak{h}_2(\ell)\in A)\le \PP(\mathfrak{h}^n_2(\cdot+n+1-\ell)-\mathfrak{h}^n_2(n+1)\in A, N+N'\le n)+\PP(N+N'\ge n)\,,
\end{equation*}
where on the event $N+N'\le n$
\begin{equation*}
\mathfrak{h}^n_2(\cdot) = \displaystyle\max_{x\in [-n, n]}(\mathfrak{h}(x) + \mathcal{S}(x,\cdot+\ell-n-1))  = \displaystyle\max_{x\in [1,2n+1-\ell]}(\mathfrak{h}(x+\ell-n-1) + \mathcal{S}(x,\cdot))
\end{equation*}
by the skew-symmetry of the Airy sheet. Note that $\mathfrak{h}$ is independent from $\mathcal{S}$ so the last distributional equality (skew-symmetry) is valid almost surely. Note also, on the event $\{N+N'\le n\}$, the initial data $\mathfrak{h}$ satisfies 
\begin{equation*}
\mathfrak{h}(\cdot) = \displaystyle\max_{x\in [-n, n]}(h_0(x) + \mathcal{S}'(x,\cdot))
\end{equation*}
where the latter is \emph{independent} of $\mathcal{S}$. Now, using Lemma \ref{lemma: Airy sheet variational formula} and Theorem~\ref{thm: intercept tail bound}, we have the estimates
\begin{align*}
    \begin{array}{ll}
         \PP(\mathfrak{h}_2(\cdot)-\mathfrak{h}_2(\ell)\in A)&\le \PP(\displaystyle\max_{1\le k \le m} (G^{\mathfrak{h}}_k+\mathcal{A}[(0,k)\to (\cdot+n+1-\ell, 1)])\\
         &-\displaystyle\max_{1\le k \le m} (G^{\mathfrak{h}}_k+\mathcal{A}[(0,k)\to (n+1, 1)])\in A)\\
         &+\PP(L_0\ge m) + \PP(N+N'\ge n)\,, \\
    \end{array}
\end{align*}
where $\mathcal{A}$ is a parabolic Airy line ensemble independent from $\mathfrak{h}$ and $L_0 = \pi[\ceil{2n+1-\ell}, \ceil{r-\ell+n+1}](0)$ (see Definition \ref{def: semi-inf geo}) and
\begin{equation*}
G^{\mathfrak{h}}_k = \displaystyle\max_{x\in [-n,n]}(\mathfrak{h}(x+\ell-n-1) + \mathcal{A}[x\to (0,k)])\,,\qquad \text{ for all } k\ge 1\,,
\end{equation*}
which we can further estimate as in \eqref{lemma: bdry data trunc} in terms of gaps between levels in the Airy line ensemble at the origin.

Now, proceeding as in the proof of Corollary \ref{cor: KPZ log comparison BM} and using the above discussion, we can now estimate for all $n, m \ge 1$ and some universal $\theta, \eta, d >0$
\begin{align*}
    \PP(\mathfrak{h}_2(\cdot)-\mathfrak{h}_2(\ell)\in A) &\le  \displaystyle O_{y_0, K}\bigg(\exp(d_{\ell, r} (n^\theta + m^\theta))\cdot \mu(A)^{\eta}+ \mathrm{e}^{-d_{y_0} m} +  \PP(L_0\geq m)\bigg)\\
    &+\PP(N+N'\ge n)\,,
\end{align*}
where $\mu(\cdot)$ denotes the rate two Wiener measure on $[\ell, r]$. Optimising over $m, n \in \N$ yields the result.
\end{proof}

\begin{remark}
    Inspecting the proof of Theorem \ref{thm: KPZ reg finitary}, we see the rate function in \eqref{eq: brown reg fin rate fn} is \textbf{uniform} in the initial data with `max-plus' supports contained in some fixed compact set.
\end{remark}

\section{Applications}\label{sec: applications}

We now turn to some applications of the quantitative Brownian regularity of the spatial increments of the KPZ fixed point, Theorem \ref{thm: KPZ reg finitary}. We note the statements below are uniform when we impose the additional assumption that the initial data have `max-plus' supports contained in a \textbf{fixed compact set}.
We start with a one-sided large deviation inequality for spatial increments of the KPZ fixed point on compacts. We refer the reader to \cite{deuschel1989large} for the precise definition of a Large Deviation Principle.

\begin{corollary}\label{cor: ldp one sided}
    Fix $t>0$, $a> 0$ and a bounded set $I\subseteq \R$. Consider the rate function
    \begin{equation*}
      I(\omega) \stackrel{\mathrm{def}}{=}\displaystyle\int_{[0, \sup I -\inf I]}|\dot{\omega}(t)|^2\diff t\,, \quad \omega \in H
    \end{equation*}
    where we denote the Cameron-Martin space by 
    \[
    H \equiv \mathscr{C}_0([0, \sup I -\inf I ])\cap W^{1,2}([0, \sup I -\inf I ])\,,
    \]
    where $W^{1,2}$ denotes the (Sobolev) space of square integrable functions with weak derivative in $L^2$.
    Then, there exists universal $0<r<1$ and $t$-dependent constants $C_t>0$ such that for all $F\subseteq \mathscr{C}_0([0, \sup I -\inf I])$ closed and finitary initial data $h_0:\R\to \R\cup\{-\infty\}$, the laws $\nu^{h_0}_\varepsilon$ of the increment processes $\sqrt{\varepsilon}(\mathfrak{h}_t(\cdot + \inf I) - \mathfrak{h}_t(\inf I))$, where $\mathfrak{h}_t(\cdot), t\ge 0$ is the KPZ fixed point started from $h_0$, as defined in \eqref{eq: KPZ fixed point}, satisfy the estimates
    \begin{equation*}
    \displaystyle\limsup_{\varepsilon\to 0}\varepsilon^r\log\nu^{h_0}_\varepsilon (F) \le -C_t\left(\inf_{\omega\in F}I(\omega)\right)^r = -\inf_{\omega\in F}\left(C_t\cdot I^r(\omega)\right)\,.
    \end{equation*}
\end{corollary}
\begin{proof}
    Apply Schilder's Theorem~and the quantitative Brownian regularity from Theorem~\ref{thm: KPZ reg finitary}.
\end{proof}
\color{black}

Quantitative Brownian regularity of the KPZ fixed point also has implications for the Radon-Nikodym derivative of the KPZ fixed point. In particular, one has for any Borel  $B\subseteq (0,\infty)$,
\begin{equation*}
\mathbb{E}[X\cdot \mathbf{1}_B]\le f(\PP(X\in B))\le c \mathrm{e}^{-d\log^r 1/\PP(X\in B)}\,.
\end{equation*}

We now prove a Corollary which states that the Radon-Nikodym derivatives of the spatial increments of the KPZ fixed point on compacts against the rate two Wiener measure are uniformly (in the initial data) integrable.
\begin{corollary}\label{cor: uniform integrability KPZ fixed point increments}
   Fix $t>0$ and a bounded set $I\subseteq \R$. Then, the family of Radon-Nikodym derivatives of the laws of the increment processes $\mathfrak{h}_t(\cdot + \inf I) - \mathfrak{h}_t(\inf I)$, where $\mathfrak{h}_t(\cdot), t\ge 0$ is the KPZ fixed point as defined in \eqref{eq: KPZ fixed point} and $h_0$ is $t-$finitary, $X^{h_0}_I$, against the rate two Wiener measure are uniformly integrable. 
\end{corollary}

\begin{proof}
    By $1:2:3$ scaling, one can without loss of generality take $t=1$. Then Theorem~\ref{thm: KPZ reg finitary} gives the existence of a universal rate function $f$ such that for any $A\in \mathcal{B}([0,\infty))$, 
    \begin{equation*}
    \displaystyle\sup_{h_0\,,1-\text{finitary }}\mathbb{E}[X^{h_0}_I \cdot \mathbf{1}(X^{h_0}_I\in A)]\le f(\PP(X^{h_0}_I\in A))\,.
    \end{equation*}
    In particular, for any $M\ge 0$, using the unit expectation of the Radon-Nikodym derivatives, $\mathbb{E}X^{h_0}_I = 1$,
    one has
    \begin{equation*}
    \displaystyle\sup_{h_0\,1-\text{finitary }}\mathbb{E}[X^{h_0}_I \cdot \mathbf{1}(|X^{h_0}_I|\ge M)]\le f(|\PP(X^{h_0}_I|\ge M))\le \displaystyle\max_{0\le s \le 1/M}f(s)\to 0 ,\quad\text{as } M\to \infty\,.
    \end{equation*}
\end{proof}

The above improvement on the tails can also be converted into an integrability statement for the Radon-Nikodym derivative itself, which is the content of the next proposition.

\begin{corollary}\label{cor: interpolation orlicz}
    Fix $t>0$, a bounded set $I\subseteq \R$ and $t-$finitary initial data $h_0:\R\to \R\cup\{-\infty\}$. Let $X^{h_0}$ denote the Radon-Nikodym derivative of the law of the increments $\mathfrak{h}_t(\cdot + \inf I) - \mathfrak{h}_t(\inf I)$, where $\mathfrak{h}_t(\cdot), t\ge 0$ is the KPZ fixed point started from $h_0$, as defined in \eqref{eq: KPZ fixed point} against the rate two Wiener measure.
    
    Then we have that there exists an explicit (in terms of the rate function) and universal (in terms of $h_0$) function,
    \begin{equation*}
    \Phi\text{ convex}\,, \Phi(0) = 0\,, \lim_{x\to \infty}\frac{\Phi(x)}{x} = \infty\,, \ \ \mbox {such that}
    \end{equation*}
    \begin{equation*}
    \mathbb{E}[\Phi(X^{h_0})]<\infty\,.
    \end{equation*}

    In particular, there exist $d_{t, I, h_0} > 0$ and universal $r \in (0, 1)$ such that
    \begin{equation*}
    \mathbb{E}[X^{h_0}\cdot \exp(d_{t, I, h_0}\log^{r} (X^{h_0}\lor 1))]<\infty\,.
    \end{equation*}
    In particular, the Radon-Nikodym derivatives of spatial increments of the KPZ fixed point against Brownian motion on compacts have \emph{finite entropy}:
    \begin{equation*}
    \mathbb{E}[X^{h_0}\cdot |\log X^{h_0}|]<\infty\,.
    \end{equation*}
\end{corollary}

\begin{remark}
    We can give $\Phi$ explicitly as a piecewise linear function with increasing slope. For any fixed time $t>0$ and bounded interval $I$, the family of Radon-Nikodym derivatives $X^{h_0}$ for $h_0$ $t$-finitary belong to the same Orlicz space
   \begin{equation*}
   L^\Phi (\mu)\stackrel{\mathrm{def}}{=} \{\xi: \mathscr{C}_{*, *}([0, \sup I - \inf I])\to \R\,: \mathbb{E}[\Phi(k|\xi|)]<\infty\,,\text{ for some } k> 0\}\subsetneq L^1(\mu)\,,
   \end{equation*}
   where $\mu$ is the rate two Wiener measure. Note $L^\Phi (\mu)$ is a Banach space equipped with the norm
   \begin{equation*}
   \norm{f}_{L^\Phi} \equiv \inf\{c\in (0, \infty): \mathbb{E}[\Phi(|\xi|/c)]\le 1\}\,.
   \end{equation*}
    
\end{remark}

\begin{proof}
By $1:2:3$ scaling, we can without loss of generality set $t = 1$. Corollary \ref{cor: uniform integrability KPZ fixed point increments} implies that one can estimate for all $t\ge t_0$, for some fixed $t_0> 0$ with $X\equiv X^{h_0}$
\begin{equation*}
\mathbb{E}[X\cdot \mathbf{1}(X\ge t)]\le c_{I, h_0}\mathrm{e}^{-d_{I, h_0}\log^r t}\,,
\end{equation*}
for positive $c_{I, h_0}, d_{I, h_0}, r> 0$ (which by Theorem~\ref{thm: KPZ reg finitary} do not depend on $h_0$). We henceforth drop dependence of $h_0$ and $I$ for ease of notation.

Now, observe there exist constants $c', d', r'>0$ such that with the sequence $t_0=0$ and $t_j = c'\exp(d'\log^{r'}(j))$, $j\ge 1$ we have
\begin{equation*}
\mathbb{E}[X\cdot \mathbf{1}(X\ge t_j)]\le \frac{1}{j^3}\,,\qquad j\ge 1\,.
\end{equation*}
Consider the continuous function
\begin{align*}
\Phi(x) \stackrel{\mathrm{def}}{=} \begin{cases}
    &0\,,\hfill\qquad x=0\;\\
    &\Phi(t_j) + (j+1)(x-t_j)\,,\qquad x\in [t_j, t_{j+1}]\,,j\ge 0\,.
\end{cases}
\end{align*}
Note one has for all $j\ge 0$, $\Phi(x) \le (j+1)x$, $x\in [t_j, t_{j+1}]$. Moreover, $\Phi$ is clearly convex and almost everywhere differentiable with $\Phi'(x) = (j+1)$ on $(t_j, t_{j+1})$ which tends to $\infty$ as $j\to \infty$, thus, we have the super linear growth
\begin{equation*}
\lim_{x\to \infty}\frac{\Phi(x)}{x} = \infty\,.
\end{equation*}
We now estimate
\begin{align*}
\mathbb{E}[\Phi(X)]&= \sum_{j\ge 0}\mathbb{E}[\Phi(X)\cdot\mathbf{1}(t_j \le X< t_{j+1})]\le \sum_{j\ge 0}(j+1)\cdot\mathbb{E}[X\cdot\mathbf{1}(X\ge t_j)]\le \sum_{j\ge 0}\frac{(j+1)}{j^3}<\infty\,.
\end{align*}
Now, observe that for all $j\ge 3$, (constants may change from line to line)
\begin{align*}
    \Phi(j)  &= \sum_{k = 0}^{j-1} (k+1)(t_{k+1}-t_k) = jt_{j}- \sum_{k = 0}^{j-1}t_k\,.
\end{align*}
Observe that by monotonicity of the $t_k$,
\begin{align*}
    \sum_{k = 0}^{j-1}t_k &= \sum_{k = 0}^{\lfloor j/2\rfloor}t_k + \sum_{k = \lfloor j/2\rfloor+1}^{j-1}t_k \le j/2\cdot t_{\lfloor j/2\rfloor} + j/2 \cdot t_{j}\,.
\end{align*}
We also estimate 
\begin{align*}
    \exp(d'\log^{r'}(j/2))&= \exp(d'\log^{r'}(j))\cdot \exp(d'\log^{r'}(j/2)-d'\log^{r'}(j))\\
    &=\exp(d'\log^{r'}(j))\cdot \exp\bigg(-d'r'\int^{j}_{j/2}\frac{\log^{r'-1}x}{x}\diff x\bigg)\\
    &\le \exp(d'\log^{r'}(j))\cdot \exp\bigg(-d'r'/2\log^{r'-1}(j/2)\bigg)\,, 
\end{align*}
and
\begin{align*}
    \exp(d'\log^{r'}(j+1))&= \exp(d'\log^{r'}(j))\cdot \exp(d'\log^{r'}(j+1)-d'\log^{r'}(j))\\
    &=\exp(d'\log^{r'}(j))\cdot \exp\bigg(d'r'\int^{j+1}_j\frac{\log^{r'-1}x}{x}\diff x\bigg)\,.
\end{align*}
Since
\begin{align*}
    \sup_{j\ge 3}\exp\bigg(d'r'\int^{j+1}_j\frac{\log^{r'-1}x}{x}\diff x\bigg)&\le C
\end{align*}
for some constant $C> 0$, we have (constants may change from line to line)
\begin{align*}
    t_{j+1} = \exp(d'\log^{r'}(j+1))&\le C\cdot \exp(d'\log^{r'}(j)) = Ct_j\,.
\end{align*}
Now, for $x\in [t_j\lor \mathrm{e}^2, t_{j+1}\lor \mathrm{e}^2]$, we have
\begin{align*}
    \Phi(x) = \Phi(j)+(j+1)(x-t_j)&\ge  jt_j- \sum_{k = 0}^{j-1}t_k\ge jt_j - (j/2\cdot t_{\lfloor j/2\rfloor} + j/2 \cdot t_{j})\\
    &\ge jt_j/2\left(1 - \exp\bigg(-d'r'/2\log^{r'-1}(j/2)\bigg)\right)\\
    &\ge C'jt_j \ge C'jt_{j+1}\ge C' x \mathrm{e}^{c\log^r x}\,,
\end{align*}
for some $C', c, r > 0$, concluding the proof.
\end{proof}
\color{black}
\section{Future directions}\label{sec: outlook}

In this section, we discuss possible ways of strengthening the quantitative Brownian comparison of the KPZ fixed point on compacts.

A key to improving Brownian regularity is to strengthen the estimates satisfied by the truncated versions of the KPZ fixed point. This would include improving the inverse acceptance probability estimates as well as the Radon-Nikodym derivative bounds of inhomogeneous BLPP. Next, a refinement of the picture of geodesic geometry in the Airy line ensemble, in particular improving tail bounds on semi-infinite geodesic intercepts and finer control over geodesic coalescence events on Brownian melons would also help improve Brownian regularity.

In particular, for any given $x>0$, if one could strengthen the comparison in Theorem~\ref{thm: Airy LPP deviation} by showing that for every $\varepsilon > 0$
$$
\PP \left(k^{1/6}|\mathcal{A}[(0, k) \rightarrow (x,1)]-{2\sqrt{2kx}}| > \varepsilon\right) \stackrel{\varepsilon \to 0}{\longrightarrow} 0  \qquad k\ge 1\,,
$$
then one would obtain improved tail bounds for $L_0$, compared to those in Theorem~\ref{thm: intercept tail bound}, possibly even showing that for all $\varepsilon > 0$, $L_0$ satisfies the tail bounds
    \begin{equation*}
        \displaystyle\sup_{j\in \N} \mathrm{e}^{j^{(3-\varepsilon)}}\cdot \PP(L_0\geq j)<\infty\,.
    \end{equation*}

Ultimately, the fruits of such an endeavour would be to obtain something along the lines of the following. 
\begin{itemize}
    \item[1.] There exist some $p>1$, $d\ge 1,r>0$ such that we have the estimate for the finite depth truncations $H_m$, $m\ge 1$ as defined in (\ref{eq: finite depth trun}) and $A$ Borel (depending only on increments on some bounded set),
    \begin{equation*}
    \PP(H_m(\cdot)\in A) \le c_p \big( \mathrm{e}^{m^{d}}\cdot \mu(A)^{1-1/p} + 
     \mathrm{e}^{-m^{r}}\big )\,,\qquad m\ge 1\,,
    \end{equation*}
    with $c_p>0$ independent of $m\in \N$, where $\mu$ denotes the law of a rate two Brownian motion;
    \item[2.] $L_0$ satisfies some tail bound
    \begin{equation*}
        \displaystyle\sup_{j\in \N} \mathrm{e}^{j^r}\cdot \PP(L_0\geq j)<\infty\,,
    \end{equation*}
     for the same $r>0$.
\end{itemize}
     
With $r > d$, one can convert the above finite depth bounds to a bound on spatial increments of the KPZ fixed point of the form 
    \begin{equation*}
     \PP(\mathfrak{h}(\cdot + 1)-\mathfrak{h}(1)\in A) \le c'_t  
     \cdot \mu(A)^{1-1/t}
    \end{equation*}
    for any $t\in (1,p)$, for some positive $c'_t > 0$ independent of $m\in \N$. In other words, the Radon-Nikodym derivative of the increment process of the KPZ fixed point $\mathfrak{h}(\cdot + 1)-\mathfrak{h}(1)$ is in $L^{p-}(\mu)$ on compacts. We believe that one can achieve $r = 3$ from transversal fluctuation of geodesics in discrete environments. Moreover, we also expect to have $d<3$ from our already established inhomogeneous Brownian LPP estimates, in addition to an improvement in estimating inverse acceptance probabilities. This would give $r>d$, as required for $L^{p-}$-regularity.

\medskip\medskip
\noindent\textbf{Acknowledgement.} SS would like to thank B\'alint Vir\'ag for some initial helpful discussions.

\bibliographystyle{alpha}
\bibliography{bibliography}

\end{document}

%% file: preamble.tex
\usepackage[T1]{fontenc}
\usepackage{lmodern}
\usepackage{mathtools, amsfonts, amssymb, mathrsfs, bm, stmaryrd}

\usepackage[utf8]{inputenc}
\usepackage[T1]{fontenc}
\usepackage{graphicx}
\usepackage{tikz}
\usetikzlibrary{shapes, arrows.meta, positioning, calc, backgrounds, lindenmayersystems, decorations.pathreplacing}
\usepackage{pgfplots, pgfplotstable}
\pgfplotsset{compat=1.17}

\usepackage{float}
\usepackage{svg}
\usepackage{wrapfig}
\usepackage{subcaption}
\usepackage{hyperref}
\pgfdeclarelindenmayersystem{cantor set}{
  \rule{F -> FfF}
  \rule{f -> fff}
}
\numberwithin{equation}{section}
\raggedright
\allowdisplaybreaks

\usepackage[letterpaper,hmargin=1.0in,vmargin=1.0in]{geometry}
\parskip 	\smallskipamount

\newcounter{dummy} \numberwithin{dummy}{section}
\newtheorem{corollary}[dummy]{Corollary}
\newtheorem{proposition}[dummy]{Proposition}
\newtheorem{theorem}[dummy]{Theorem}
\newtheorem{definition}[dummy]{Definition}
\newtheorem{lemma}[dummy]{Lemma}

\newtheorem*{remark}{Remark}

\newcommand{\N}{\ensuremath{\mathbb{N}}}
\newcommand{\R}{\ensuremath{\mathbb{R}}}
\newcommand{\Z}{\ensuremath{\mathbb{Z}}}
\newcommand{\Q}{\ensuremath{\mathbb{Q}}}

\newcommand{\PP}{\ensuremath{\mathbb{P}}}

\newcommand{\floor}[1]{\lfloor #1 \rfloor}
\newcommand{\ceil}[1]{\lceil #1 \rceil}

\newcommand{\supp}[1]{\mathrm{supp}_{-\infty}(#1)}

\newcommand{\norm}[1]{\left\lVert#1\right\rVert}

 \newcommand\restrict[2]{
   \left.\kern-\nulldelimiterspace 
   #1
   \littletaller 
   \right|_{#2}%
   }
 \newcommand{\littletaller}{\mathchoice{\vphantom{\big|}}{}{}{}}

 \makeatletter
 \newcommand*{\scaleddelims}[3]{%
   \ensuremath{%
     \mathpalette{\@scaleddelims{#1}{#2}}{#3}%
   }%
 }   
 \newcommand*{\@scaleddelims}[4]{%
   \begingroup
     #3%
     \sbox0{$\m@th#3\vphantom{A}#4$}%
     \setbox2\vbox{\hbox{$\m@th#3#1$}\kern\z@}%
     \setbox4\vbox{\hbox{$\m@th#3#2$}\kern\z@}%
     \setbox6\hbox{$#3\vcenter{}$}%
     \ifx\downharpoonleft#1\relax  
       \let\DelimLeft=L%
     \else\ifx\upharpoonleft#1%
       \let\DelimLeft=L%
     \else\ifx\downharpoonright#1%
       \let\DelimLeft=R%
     \else\ifx\upharpoonright#1%
       \let\DelimLeft=R%
     \fi\fi\fi\fi
     \ifx\downharpoonleft#2\relax
       \let\DelimRight=L%
     \else\ifx\upharpoonleft#2\relax
       \let\DelimRight=L%
     \else\ifx\downharpoonright#2\relax
       \let\DelimRight=R%
     \else\ifx\upharpoonright#2\relax
       \let\DelimRight=R%
     \fi\fi\fi\fi
     \ifx\DelimLeft L%
       \wd2=.6\wd2
     \fi
     \ifx\DelimRight L%
       \wd4=.6\wd4
     \fi
     \ifx\DelimLeft R%
       \sbox2{\kern-.4\wd2\box2}%
     \fi
     \ifx\DelimRight R%
       \sbox4{\kern-.4\wd4\box4}%
     \fi
     \dimen0=\ht0 %
     \advance\dimen0 by -\ht6 %
     \dimen2=\dp0 %
     \advance\dimen2 by \ht6 %
     \ifdim\dimen2>\dimen0 %
       \dimen0=\dimen2 %
     \else
       \dimen0=\dimen0 %
     \fi
     \dimen2=\ht6 %
     \advance\dimen2 by -\dimen0 %
     \dimen0=2\dimen0 %
     \def\DelimCorr{%
       \mskip.5\thinmuskip
       \nonscript\mskip.5\thinmuskip
     }%
     \mathopen{%
       \ifx\DelimLeft R\DelimCorr\fi
       \raisebox{\dimen2}{\resizebox{!}{\dimen0}{\box2}}%
       \ifx\DelimLeft L\DelimCorr\fi
     }%
     \begingroup
       #3#4%
     \endgroup
     \mathclose{%
       \ifx\DelimRight R\DelimCorr\fi
       \raisebox{\dimen2}{\resizebox{!}{\dimen0}{\box4}}%
       \ifx\DelimRight L\DelimCorr\fi
     }%
   \endgroup
 }\makeatother

 \newcommand{\restr}[2]{#1\scaleddelims{\kern-0.5\nulldelimiterspace\upharpoonright}{\vphantom{.}}{_{#2}}}
\newcommand{\diff}{\ensuremath{\operatorname{d}\!}}